\documentclass[11pt,a4paper,twoside]{article}
\usepackage[naustrian,UKenglish]{babel}
\usepackage[utf8x]{inputenc}
\usepackage[T1]{fontenc}
\usepackage[shortcuts]{extdash}
\usepackage{latexsym,amsfonts,amsmath,amsthm,amssymb}
\usepackage{fullpage}
\usepackage{xcolor,bm,bbm}
\usepackage{array}
\usepackage{paralist}
\usepackage{mathrsfs}

\newcommand*{\zalmeng}[1]{\mathbb{#1}}
\newcommand*{\NZ}[0]{\zalmeng{N}}
\newcommand*{\RZ}[0]{\zalmeng{R}}
\newcommand*{\CZ}[0]{\zalmeng{C}}
\newcommand*{\Kug}[2]{\mathbb{B}_{#1}^{#2}}
\newcommand*{\nKug}[2]{\mathbb{D}_{#1}^{#2}}
\newcommand*{\Sph}[2]{\mathbb{S}_{#1}^{#2}}
\newcommand*{\ez}[0]{\mathsf{e}}
\newcommand*{\dif}[0]{\mathrm{d}}
\newcommand*{\kompl}[1]{#1^{\mathsf{c}}}
\newcommand*{\Konv}[1]{\xrightarrow[#1]{}}
\newcommand*{\Kfs}[1]{\xrightarrow[#1]{\text{\upshape a.s.}}}
\newcommand*{\KiWsk}[1]{\xrightarrow[#1]{\mathbb{P}}}
\newcommand*{\KiVert}[1]{\xrightarrow[#1]{\text{\upshape d}}}
\newcommand*{\GlVert}[0]{\stackrel{\text{\upshape d}}{=}}
\newcommand*{\ellpe}[2]{\ell_{#1}^{#2}}
\newcommand*{\KugVol}[2]{\omega_{#1}^{#2}}
\newcommand*{\KegM}[2]{\kappa_{#1}^{#2}}
\newcommand*{\Eta}[0]{\mathrm{H}}
\newcommand*{\trapo}[1]{#1^{\mathsf{T}}}
\newcommand*{\inv}[1]{#1^{-1}}
\newcommand*{\Colon}[0]{\,:\,} 
\newcommand*{\vek}[1]{\mathbf{#1}}

\newlength{\absatz}
\newcommand*{\Absatz}[0]{\hspace*{\absatz}}

\DeclareMathOperator{\chF}{\mathbbm{1}}
\DeclareMathOperator{\vol}{\mathit{v}}
\DeclareMathOperator{\Unif}{Unif}
\DeclareMathOperator{\Nvert}{\mathcal{N}}
\DeclareMathOperator{\Gaussp}{\gamma}
\DeclareMathOperator{\Exp}{\mathcal{E}}
\DeclareMathOperator{\Wsk}{\mathbb{P}}
\DeclareMathOperator{\Erw}{\mathbb{E}}
\DeclareMathOperator{\Var}{Var}
\DeclareMathOperator{\Cov}{Cov}
\DeclareMathOperator{\CVar}{Cov}
\DeclareMathOperator{\diag}{diag}
\DeclareMathOperator{\Lip}{Lip}
\DeclareMathOperator{\LipB}{Lip_{\text{\upshape b}}}
\DeclareMathOperator{\Masz}{\mathcal{M}}
\DeclareMathOperator{\Vertl}{\mathcal{L}}
\DeclareMathOperator{\Borel}{\mathcal{B}}
\DeclareMathOperator{\BigO}{\mathcal{O}}
\DeclareMathOperator{\smallO}{\mathit{o}}
\DeclareMathOperator{\BigTheta}{\Theta}

\DeclareMathOperator{\dLP}{\mathit{d}_{\text{\upshape LP}}}
\DeclareMathOperator{\dTV}{\mathit{d}_{\text{\upshape TV}}}

\theoremstyle{plain}
\newtheorem{Sa}{Proposition}

\newtheorem{Kor}[Sa]{Corollary}
\newtheorem{Lem}[Sa]{Lemma}
\newtheorem{thmalpha}{Theorem}

\theoremstyle{definition}

\theoremstyle{remark}
\newtheorem{Bem}[Sa]{Remark}

\makeatletter
\renewcommand{\@upn}{} 
\makeatother

\allowdisplaybreaks

\begin{document}

\title{\textbf{Limit theorems for mixed\-/norm sequence spaces\protect\\with applications to volume distribution}}

\medskip

\author{Michael L.\ Juhos, Zakhar Kabluchko, and Joscha Prochno}

\date{}

\maketitle

\begin{abstract}
\small
Let \(p, q \in (0, \infty]\) and \(\ell_p^m(\ell_q^n)\) be the mixed-norm sequence space of real matrices \(x = (x_{i, j})_{i \leq m, j \leq n}\) endowed with the (quasi-)norm \(\lVert x \rVert_{p, q} := \bigl\lVert \bigl( \lVert (x_{i, j})_{j \leq n} \rVert_q \bigr)_{i \leq m} \bigr\rVert_p\). We shall prove a Poincar\'e--Maxwell--Borel lemma for suitably scaled matrices chosen uniformly at random in the \(\ell_p^m(\ell_q^n)\)-unit balls \(\mathbb{B}_{p, q}^{m, n}\), and obtain both central and non-central limit theorems for their \(\ell_p(\ell_q)\)-norms. We use those limit theorems to study the asymptotic volume distribution in the intersection of two mixed-norm sequence balls. Our approach is based on a new probabilistic representation of the uniform distribution on \(\mathbb{B}_{p, q}^{m, n}\).

\vspace{0.5\baselineskip}
\end{abstract}

\tableofcontents

\section{Introduction and main results}

The asymptotic theory of convex bodies is intimately linked to probability theory whose methods and ideas have been key elements in obtaining numerous deep results of both analytic and geometric flavour. It has led to the development of a quite powerful quantitative methodology in geometric functional analysis and allowed to form a qualitatively new picture of high-dimensional spaces and structures. The role of convexity in high\-/dimensional spaces is similar to the role of independence in probability and guarantees a certain regularity of the otherwise complex structure of a high-dimensional space. One of the most classical results of stochastic-geometric and high\-/dimensional flavor is probably the Poincar\'e\--Maxwell\--Borel lemma, which asserts that any fixed number of coordinates of a vector chosen uniform at random from the boundary of the unit Euclidean ball \(\Kug{2}{n}\) is approximately Gaussian (see, e.g., \cite{DF1987}), and in the more modern spirit there is the pioneering work of V.\,D.\ Milman on the concentration\-/of\-/measure phenomenon, which has led to several major breakthroughs (see, e.g., \cite{AGM2015, MS1986}). The arguably most prominent example of the last two decades is Klartag's central limit theorem for convex bodies, showing that the marginals of a high\-/dimensional isotropic and log\-/concave random vector are approximately Gaussian distributed \cite{K2007}.
Besides Klartag's central limit theorem for convex bodies, a number of other (weak) limit theorems have been obtained for various geometric quantities in the last decades, demonstrating their regularity and universality; we refer to the survey \cite{PTT2020} for references. Several of those results have led to a deeper understanding of the volume distribution in high\-/dimensional convex bodies.

The motivation of the present paper is essentially twofold and will be elaborated upon in view of classical and preceding works before presenting our main results.

\textit{Motivation 1: Poincar\'e\--Maxwell\--Borel type results.} Having its roots in kinetic gas theory, and going back to Maxwell and later Poincaré and Borel, it is observed that the first \(k\) coordinates of a random point on the \((n - 1)\)\=/dimensional Euclidean sphere \(\Sph{2}{n - 1}\) are asymptotically independent and Gaussian as \(n\) tends to infinity; to be precise,
\begin{equation*}
\lim_{n \to \infty} \dTV\bigl( \Vertl(\sqrt{n} (X_i^n)_{i \leq k}), \Vertl((Z_i)_{i \leq k}) \bigr) = 0,
\end{equation*}
where \(\dTV\) denotes the total variation distance,  \((X_i^n)_{i \leq n}\) is sampled uniformly from \(\Sph{2}{n - 1}\) and \(Z_1, \dotsc, Z_k\) are independent standard Gaussian variables, and \(k \in \NZ\) is fixed. We refer to Diaconis and Freedman~\cite[Section~6]{DF1987} for a more detailed account and give a more detailed statement in Proposition~\ref{prop:mpb} below. In~\cite{DF1987}, Diaconis and Freedman prove an analogous result for the simplex and exponential distribution. Generalizations to the \(\ellpe{p}{}\)\=/sphere were obtained by Mogul\('\)ski\u{\i}~\cite{Mog1991}, where the point was distributed according to the normalized Hausdorff measure, and by Rachev and R\"uschendorf~\cite{RR1991} for the cone probability measure. The latter authors exploited a probabilistic representation relating a \(p\)\=/generalized Gaussian distribution to the \(\ellpe{p}{}\)\=/balls, allowing one to make a transition from a random vector with dependent coordinates to one with independent ones. Naor and Romik~\cite{NR2003} showed that the normalized Hausdorff measure and the cone probability measure are asymptotically equal (their equality for \(p \in \{1, 2, \infty\}\) irrespective of dimension being long known prior), thereby unifying the previous results. A further generalization to Orlicz balls (and even beyond) was undertaken recently by Johnston and Prochno~\cite{JP2020}. We stress that all results cited have been proved for the total variation distance of probability measures. In the present article we only consider the weak topology on probability measures (equivalently: convergence in distribution of random variables).

\textit{Motivation 2: Schechtman\--Schmuckenschl\"ager type results.} Instigated by a question of V.\,D.\ Milman, Schechtman and Zinn~\cite{SchechtmanZinn} found an upper bound on the volume left over from an \(\ellpe{p}{}\)\=/ball after cutting out a dilated \(\ellpe{q}{}\)\=/ball; incidentally the authors utilized the same stochastic representation as did Rachev and R\"uschendorf (see above). A few years after, Schechtman and Schmuckenschl\"ager~\cite{SS1991} used that probabilistic representation in order to investigate the limit of the volume of the cut\-/out portion in the very same setting as before, revealing the following threshold behaviour: below a certain critical dilation factor depending only on \(p\) and \(q\) the limit is zero, and above that it is one, provided the \(\ellpe{p}{}\)\=/ball has unit volume. More formally, writing \(\nKug{p}{n}\) for the \(n\)\=/dimensional unit\-/volume \(\ellpe{p}{}\)\=/ball,
\begin{equation*}
\lim_{n \to \infty}\vol_n(\nKug{p}{n} \cap t \nKug{q}{n}) =%
\begin{cases}
0 & \text{if } t A_{p, q} < 1,\\
1 & \text{if } t A_{p, q} > 1.
\end{cases}
\end{equation*}
About a decade later, Schmuckenschl\"ager~\cite{Schmu1998, Schmu2001} determined the asymptotics at the threshold itself and found the limit to be \(1/2\) by proving a central limit theorem that revealed this behaviour. We refer to Proposition~\ref{prop:ss} below for the precise statement. More recently, Kabluchko, Prochno, and Th\"ale~\cite{KPT2019_I, KPT2019_II} revisited the results of Schechtman and Schmuckenschl\"ager, providing a unified framework and also generalizing the previous works in various directions, yet still treating \(\ellpe{p}{}\)\=/balls and using the probabilistic representation. A further step was taken by Kabluchko and Prochno~\cite{KP2021}, studying the intersections of Orlicz balls and observing a similar thresholding  behaviour; here much finer tools from large deviations theory and statistical mechanics where required, and it is not even known whether the limit at the threshold itself exists. Another generalization from \(\ellpe{p}{}\)\=/balls to \(\ellpe{p}{}\)\=/ellipsoids, i.e., axis\-/parallel\-/scaled balls (a case not covered by Orlicz balls), was recently obtained by Juhos and Prochno in~\cite{JuP2022}; the phenomenon of the threshold emerges again. The case of intersections of unit balls from classical random matrix ensembles has been treated by Kabluchko, Prochno, and Th\"ale in \cite{KPT2020}. Let us point out that understanding the asymptotic volume of intersections of scaled unit balls naturally appears, for instance, when studying the curse of dimensionality for high\-/dimensional numerical integration problems \cite{HPU2019}.

Suspecting a universal behaviour among symmetric convex bodies, we tackle another generalization, namely finite\-/dimensional sequence spaces with mixed \(\ellpe{p}{}\)\=/norms, and consider the asymptotic volume of the intersection of two balls: the thresholding behaviour is found to be valid also in this case, and for a wide range of parameters the limit in the critical case is determined; little surprisingly, owing to the larger set of parameters as compared to the \(\ellpe{p}{}\)\=/balls, this limit's value is much more varied and the overall analysis is considerably more delicate.

Let us point out that the study of mixed\-/norm spaces is a classical one in approximation theory and geometric functional analysis and we refer, for instance, to the work of Sch\"utt regarding the symmetric basis constant of these spaces \cite{Sch1981}, the characterization of mixed\-/norm subspaces of $L_1$ by Prochno and Sch\"utt \cite{PSch2012} and Schechtman \cite{Sche2013}, the work on non\-/existence of greedy bases for the mixed\-/norm spaces by Schechtman \cite{Sche2014} and the study of volumetric properties of these spaces by Kempka and Vyb\'iral \cite{KV2017} as well as the recent work of Mayer and Ullrich on the order of entropy numbers of mixed\-/norm unit balls \cite{MU2021}.

We would like to add that, naturally, it would be interesting to consider even more general norms. The main hindrance, though, is that each of the results referenced in the motivation above, and others more, has required tools tailored to the specific problems; to our best knowledge there is no unified theory yet that would allow us to assess such questions ``in one fell swoop.'' Current research is conducted, e.g., for Schatten norms of not necessarily square matrices.

\subsubsection*{The mathematical setup}

In order to be able to present our main results, we shall briefly introduce the most essential setup; more details can be found in Section~\ref{sec:notation} on notation and preliminaries.

For \(p, q \in (0, \infty]\) and \(m, n \in \NZ\) define the finite\-/dimensional mixed\-/norm sequence space \(\ellpe{p}{m}(\ellpe{q}{n})\) to be the space \(\RZ^{m \times n}\) endowed with the \((m \cdot n)\)\-/dimensional Lebesgue measure \(\vol_{m n}\) and the quasinorm
\begin{equation*}
\lVert x \rVert_{p, q} := \bigl\lVert \bigl( \lVert (x_{i, j})_{j \leq n} \rVert_q \bigr)_{i \leq m} \bigr\rVert_p,
\end{equation*}
where \(x = (x_{i, j})_{i \leq m, j \leq n} \in \RZ^{m \times n}\), and \(\lVert \cdot \rVert_p\) is the usual \(\ellpe{p}{}\)\=/norm, that is,
\begin{equation*}
\lVert (x_i)_{i \leq n} \rVert_p :=%
\begin{cases}
\bigl( \sum\limits_{i = 1}^n \lvert x_i \rvert^p \bigr)^{1/p} & \text{if } p < \infty,\\
\max\limits_{i \leq n} \lvert x_i \rvert & \text{if } p = \infty.
\end{cases}
\end{equation*}
In particular, we consider the unit balls
\begin{equation*}
\Kug{p, q}{m, n} := \big\{x \in \RZ^{m \times n} \Colon \lVert x \rVert_{p, q} \leq 1\big\};
\end{equation*}
the \(\ellpe{p}{}\)\-/unit ball and sphere in \(\RZ^n\) are written \(\Kug{p}{n}\) and \(\Sph{p}{n - 1}\), respectively; \(\KugVol{p}{n}\) denotes the volume of \(\Kug{p}{n}\).

We seek to characterize \(\Unif(\Kug{p, q}{m, n})\), the uniform distribution on \(\Kug{p, q}{m, n}\). Given a random matrix \(X = (X_{i, j})_{i \leq m, j \leq n} \sim \Unif(\Kug{p, q}{m, n})\), define
\begin{equation}\label{eq:r_theta}
R_i := \lVert (X_{i, j})_{j \leq n} \rVert_q \quad \text{and} \quad \Theta_i := (\Theta_{i, j})_{j \leq n} := \Bigl( \frac{X_{i, j}}{R_i} \Bigr)_{j \leq n} \quad \text{for } i \in [1, m];
\end{equation}
then clearly \((R_i)_{i \leq m} \in \Kug{p}{m} \cap [0, \infty)^m\),  \(\Theta_i\) is almost surely well-defined and \(\Theta_i \in \Sph{q}{n - 1}\).

The notations \(R_i\) and \(\Theta_i\), \(\Theta_{i, j}\) are used throughout this article with the meaning given in~\eqref{eq:r_theta}; note that they actually depend on the parameters \(p, q, m, n\), but we suppress this in our notation.

For \(p \in (0, \infty]\) the \emph{\(p\)\=/generalized Gaussian distribution,} or \emph{\(p\)\=/Gaussian distribution} for short, is defined to be the probability measure on \(\RZ\) with Lebesgue\-/density
\begin{equation*}
x \mapsto \begin{cases}
\frac{1}{2 p^{1/p} \, \Gamma(\frac{1}{p} + 1)} \, \ez^{-\lvert x \rvert^p/p} & \text{if } p < \infty,\\
\frac{1}{2} \chF_{[-1, 1]}(x) & \text{if } p = \infty.
\end{cases}
\end{equation*}

\subsection{Main results\---a Schechtman\--Zinn probabilistic representation}

The first main result, which facilitates all computations and is essential to our proofs, is a probabilistic representation of the uniform distribution on \(\Kug{p, q}{m, n}\), generalizing the classical result of Schechtman and Zinn~\cite{SchechtmanZinn} and Rachev and R\"uschendorf~\cite{RR1991}. Given the numerous applications of the classical probabilistic representation, the following result clearly is of independent interest.

\begin{Sa}\label{sa:stoch_darst}
Let \(p, q \in (0, \infty]\) and \(m, n \in \NZ\), and let \(X \sim \Unif(\Kug{p, q}{m, n})\).
\begin{asparaenum}[(a)]
\item The distribution of \((R_i)_{i \leq m}\) has Lebesgue-density
\begin{equation*}
f_{R_1, \dotsc, R_m}(r_1, \dotsc, r_m) = \frac{(2 n)^m}{\KugVol{p/n}{m}}\prod_{i = 1}^m r_i^{n - 1} \cdot \chF_{\Kug{p}{m} \cap [0, \infty)^m}(r_1, \dotsc, r_m).
\end{equation*}
Therefore, \((R_i)_{i \leq m}\) can be represented as
\begin{equation*}
(R_i)_{i \leq m} \GlVert
\begin{cases} U^{1/(m n)} \, \bigl( \frac{\lvert \xi_i \rvert^{1/n}}{(\sum_{k = 1}^m \lvert \xi_k \rvert^{p/n})^{1/p}} \bigr)_{i \leq m} & \text{if } p < \infty,\\
(\lvert \xi_i \rvert^{1/n})_{i \leq m} & \text{if } p = \infty,
\end{cases}
\end{equation*}
where \(U, \xi_1, \dotsc, \xi_m\) are independent random variables with \(U\) distributed uniformly on \([0, 1]\), and \(\xi_1, \dotsc, \xi_m\) are \(\frac{p}{n}\)\=/Gaussian.
\item The random vectors \((R_i)_{i \leq m}, \Theta_1, \dotsc, \Theta_m\) are all independent, each \(\Theta_i\) is distributed according to the cone measure on \(\Sph{q}{n - 1}\), for \(i \in [1, m]\), and therefore can be represented as
\begin{equation*}
\Theta_i \GlVert \Bigl( \frac{\eta_{i, j}}{\lVert (\eta_{i, l})_{l \leq n} \rVert_q} \Bigr)_{j \leq n},
\end{equation*}
where \((\eta_{i, j})_{i \leq m, j \leq n}\) is an array of independent \(q\)\=/Gaussian random variables.
\item The components \(X_{i, j}\) of \(X\) have the representation
\begin{equation}\label{eq:stoch_darst}
\begin{split}
X_{i, j} &= R_i \Theta_{i, j}\\
&\GlVert \begin{cases}
U^{1/(m n)} \frac{\lvert \xi_i \rvert^{1/n}}{(\sum_{k = 1}^m \lvert \xi_k \rvert^{p/n})^{1/p}} \, \frac{\eta_{i, j}}{\lVert (\eta_{i, l})_{l \leq n} \rVert_q} & \text{if } p < \infty,\\
\lvert \xi_i \rvert^{1/n} \, \frac{\eta_{i, j}}{\lVert (\eta_{i, l})_{l \leq n} \rVert_q} & \text{if } p = \infty,
\end{cases}
\end{split}
\end{equation}
where \(U, \xi_1, \dotsc, \xi_m, \eta_{1, 1}, \dotsc, \eta_{m, n}\) are as before.
\end{asparaenum}
\end{Sa}

\subsection{Main results\---Poincar\'e\--Maxwell\--Borel principles}
\label{subsec:pmb}

One type of limit theorem which we are considering is a Poincar\'e\--Maxwell\--Borel principle, that is, a statement about the limiting distribution of the first few coordinates of a random vector. In the following two theorems, we shall always assume \((X_{i, j})_{i \leq m, j \leq n} \sim \Unif(\Kug{p, q}{m, n})\).

Owing to the nature of the space \(\ellpe{p}{m}(\ellpe{q}{n})\), having two parameters for dimension, in the sequel limit theorems will usually be considered for three different regimes: firstly, letting \(m \to \infty\) while keeping \(n\) fixed; secondly, vice versa, keeping \(m\) fixed while letting \(n \to \infty\); and thirdly, letting \(n \to \infty\) while treating \(m\) as dependent on \(n\) and going to infinity as well.

In order to keep the amount of case distinctions at a minimum, for the case of the parameter value \(p = \infty\) we agree on these conventions:
\begin{equation*}
\frac{c}{p} := 0 \text{ for any } c \in \RZ, \quad \frac{p}{c} := \infty \text{ for any } c \in (0, \infty), \quad p^{1/p} := 1.
\end{equation*}
For the formulation of our results we introduce the following quantities (whose superscripts denote indices, not powers):
\begin{equation}\label{eq:gp_momente}
\begin{aligned}
M_p^\alpha := \frac{p^{\alpha/p}}{\alpha + 1} \, \frac{\Gamma(\frac{\alpha + 1}{p} + 1)}{\Gamma(\frac{1}{p} + 1)} &\quad \text{for } p \in (0, \infty] \text{ and } \alpha \in (0, \infty),\\
M_p^{p} := 1 &\quad \text{for } p = \infty,
\end{aligned}
\end{equation}
and
\begin{gather*}
C_p^{\alpha, \beta} := M_p^{\alpha + \beta} - M_p^\alpha \, M_p^\beta, \quad V_p^\alpha := C_p^{\alpha, \alpha} \quad \text{for } p \in (0, \infty] \text{ and } \alpha, \beta \in (0, \infty),\\
C_p^{p, \beta} := V_p^{p} := 0 \quad \text{for \(p = \infty\) and \(\beta \in (0, \infty)\).}
\end{gather*}

We can now formulate the first Poincar\'e\--Maxwell\--Borel principle for the case where \(m\to\infty\) while \(n\) is fixed. By \(\Vertl(X)\) we denote the distribution, or law, of a random variable~\(X\).

\begin{thmalpha}[\(m \to \infty\), \(n\) constant]\label{sa:mpb_nkonst}
Let \(p, q \in (0, \infty]\), let \(k, n \in \NZ\) be fixed, and let \(\xi_1, \dotsc, \xi_k\) be independent \(\frac{p}{n}\)\=/Gaussian random variables.
\begin{asparaenum}[(a)]
\item The following weak convergence holds true,
\begin{equation*}
(m^{1/p} \, X_{i, j})_{i \leq k, j \leq n} \KiVert{m \to \infty} (\lvert \xi_i \rvert^{1/n} \, \Theta_i)_{i \leq k}.
\end{equation*}
\item The empirical measures satisfy
\begin{align*}
\frac{1}{m} \sum_{i = 1}^m \delta_{m^{1/p} \, R_i} &\KiWsk{m \to \infty} \Vertl\bigl( \lvert \xi_1 \rvert^{1/n} \bigr),\\
\intertext{and}
\frac{1}{m} \sum_{i = 1}^m \delta_{m^{1/p} \, (X_{i, j})_{j \leq n}} &\KiWsk{m \to \infty} \Vertl\bigl( \lvert \xi_1 \rvert^{1/n} \, \Theta_1 \bigr).
\end{align*}
The convergence is to be understood as convergence in probability in the space of probability measures on \(\RZ\) and \(\RZ^n\), respectively, endowed with the Lévy\--Prokhorov metric; cf.\ Lemma~\ref{lem:kgz_wsk}.
\end{asparaenum}
\end{thmalpha}

We now formulate the second Poincar\'e\--Maxwell\--Borel principle for \(n \to \infty\) while \(m\) is either fixed or tends to infinity with \(n\).

\begin{thmalpha}[\(n \to \infty\)]\label{sa:mpb_unendl}
Let \(p, q \in (0, \infty]\), let \(m \in \NZ\) be fixed or let \(m = m(n) \to \infty\) as \(n \to \infty\), let \(k, l \in \NZ\) (\(k \leq m\) if necessary), and let \((\eta_{i, j})_{i \leq k, j \leq l}\) be an array of independent \(q\)\=/Gaussian random variables.
\begin{asparaenum}[(a)]
\item The following weak convergence holds true,
\begin{equation*}
(m^{1/p} \, n^{1/q} \, X_{i, j})_{i \leq k, j \leq l} \KiVert{n \to \infty} (\eta_{i, j})_{i \leq k, j \leq l}.
\end{equation*}
\item The empirical measure satisfies
\begin{equation*}
\frac{1}{m n} \sum_{i = 1}^m \sum_{j = 1}^n \delta_{m^{1/p} \, n^{1/q} \, X_{i, j}} \KiWsk{n \to \infty} \Vertl(\eta_{1, 1}).
\end{equation*}
\end{asparaenum}
The convergence is to be understood as convergence in probability in the space of probability measures on \(\RZ\) endowed with the Lévy\--Prokhorov metric; cf.\ Lemma~\ref{lem:kgz_wsk}.
\end{thmalpha}

\subsection{Main results\---weak limit theorems}
\label{sec:gws}

Here we present three weak limit theorems for \(\lVert X \rVert_{p_2, q_2}\), where \(X \sim \Unif(\Kug{p_1, q_1}{m, n})\) and \((p_1, q_1) \neq (p_2, q_2)\) in general (for \((p_1, q_1) = (p_2, q_2)\) see Remark~\ref{bem:p1q1_p2q2} below). We start with the case \(m \to \infty\) while \(n\) is fixed.

\begin{thmalpha}[\(m \to \infty\), \(n\) fixed]\label{sa:zgs1}
Let \(p_1, q_1, q_2 \in (0, \infty]\) and \(p_2 \in (0, \infty)\) with \((p_1, q_1) \neq (p_2, q_2)\), let \(n \in \NZ\), and for each \(m \in \NZ\) let \(X^m \sim \Unif(\Kug{p_1, q_1}{m, n})\). Then
\begin{equation*}
\Biggl( \sqrt{m} \biggl( \frac{m^{1/p_1 - 1/p_2}}{(M_{p_1/n}^{p_2/n} \Erw[\lVert \Theta_1 \rVert_{q_2}^{p_2}])^{1/p_2}} \, \lVert X^m \rVert_{p_2, q_2} - 1 \biggr) \Biggr)_{m \geq 1} \KiVert{} \sigma N,
\end{equation*}
where \(N\) is a standard Gaussian random variable, and
\begin{equation*}
\sigma^2 := \frac{1}{n p_1} - \frac{1}{p_2^2} - \frac{2 C_{p_1/n}^{p_1/n, p_2/n}}{p_1 p_2 M_{p_1/n}^{p_2/n}} + \frac{M_{p_1/n}^{2 p_2/n} \Erw[\lVert \Theta_1 \rVert_{q_2}^{2 p_2}]}{(p_2 M_{p_1/n}^{p_2/n} \Erw[\lVert \Theta_1 \rVert_{q_2}^{p_2}])^2}.
\end{equation*}
\end{thmalpha}

\begin{Bem}
It can be shown that \(\sigma^2 = 0\) iff \((p_1, q_1) = (p_2, q_2)\).
\end{Bem}

The following weak limit theorem covers the case where \(n \to \infty\) while \(m\) is fixed. We obtain both central and non\-/central limit behaviour, depending on the relation/values of the parameters \(p_1\), \(q_1\), and \(q_2\).

\begin{thmalpha}[\(m\) fixed, \(n \to \infty\)]\label{sa:zgs2}
Let \(p_1, q_1 \in (0, \infty]\) and \(p_2, q_2 \in (0, \infty)\) with \((p_1, q_1) \neq (p_2, q_2)\), let \(m \in \NZ\) be fixed, and for each \(n \in \NZ\) let \(X^n \sim \Unif(\Kug{p_1, q_1}{m, n})\).
\begin{compactenum}[(a)]
\item If \(q_1 \neq q_2\), then
\begin{equation*}
\Biggl( \sqrt{n} \biggl( \frac{m^{1/p_1 - 1/p_2} \, n^{1/q_1 - 1/q_2}}{(M_{q_1}^{q_2})^{1/q_2}} \lVert X^n \rVert_{p_2, q_2} - 1 \biggr) \Biggr)_{n \geq 1} \KiVert{} \sigma N,
\end{equation*}
where \(N\) is a standard Gaussian random variable, and
\begin{equation*}
\sigma^2 := \frac{1}{m} \biggl( \frac{V_{q_1}^{q_1}}{q_1^2} - \frac{2 C_{q_1}^{q_1, q_2}}{q_1 q_2 M_{q_1}^{q_2}} + \frac{V_{q_1}^{q_2}}{(q_2 M_{q_1}^{q_2})^2} \biggr).
\end{equation*}
\item If \(q_1 = q_2\) and \(p_1 < \infty\), then
\begin{equation*}
\bigl( m n (1 - m^{1/p_1 - 1/p_2} \lVert X^n \rVert_{p_2, q_1}) \bigr)_{n \geq 1} \KiVert{} E + \frac{p_1 - p_2}{2 p_1} \sum_{i = 1}^{m - 1} N_i^2,
\end{equation*}
where \(E\) is an exponentially distributed random variable with mean \(1\), \(N_1, \dotsc, N_{m - 1}\) are standard Gaussian random variables, and \(E, N_1, \dotsc, N_{m - 1}\) are independent.
\item If \(q_1 = q_2\) and \(p_1 = \infty\), then
\begin{equation*}
\bigl(  m n (1 - m^{-1/p_2} \lVert X^n \rVert_{p_2, q_1}) \bigr)_{n \geq 1} \KiVert{} \sum_{i = 1}^m E_i,
\end{equation*}
where \(E_1, \dotsc, E_m\) are independent, exponentially distributed random variables with mean \(1\).
\end{compactenum}
\end{thmalpha}

\begin{Bem}\label{bem:p1q1_p2q2}
Statement~(b) above remains true even if \(p_1 = p_2\) and \(m = m(n) \to \infty\); this is because then \(\lVert X^n \rVert_{p_1, q_1} \GlVert U^{1/(m n)}\), and \(\bigl( m n (1 - U^{1/(m n)}) \bigr)_{n \to \infty} \KiVert{} E\).
\end{Bem}

The third weak limit theorem now treats the case where both \(m\) and \(n\) tend to infinity.

\begin{thmalpha}[\(m, n \to \infty\)]\label{sa:zgs3}
Let \(p_1, q_1 \in (0, \infty]\) and \(p_2, q_2 \in (0, \infty)\) with \((p_1, q_1) \neq (p_2, q_2)\), let \(m = m(n) \to \infty\) as \(n \to \infty\), let \(N\) be a standard Gaussian random variable, and for each \(n \in \NZ\) let \(X^n \sim \Unif(\Kug{p_1, q_1}{m, n})\).
\begin{compactenum}[(a)]
\item If \(q_1 \neq q_2\), then
\begin{equation*}
\Biggl( \sqrt{m n} \biggl( \frac{m^{1/p_1 - 1/p_2}}{(M_{p_1/n}^{p_2/n} \Erw[\lVert \Theta_1 \rVert_{q_2}^{p_2}])^{1/p_2}} \lVert X^n \rVert_{p_2, q_2} - 1 \biggr) \Biggr)_{n \geq 1} \KiVert{} \sigma N,
\end{equation*}
where
\begin{equation*}
\sigma^2 := \frac{V_{q_1}^{q_1}}{q_1^2} - \frac{2 C_{q_1}^{q_1, q_2}}{q_1 q_2 M_{q_1}^{q_2}} + \frac{V_{q_1}^{q_2}}{(q_2 M_{q_1}^{q_2})^2}.
\end{equation*}
\item If \(q_1 = q_2\) and \(p_1 < \infty\), then
\begin{equation*}
\Biggl( \sqrt{m} \, n \biggl( \frac{m^{1/p_1 - 1/p_2}}{(M_{p_1/n}^{p_2/n})^{1/p_2}} \lVert X^n \rVert_{p_2, q_1} - 1 \biggr) \Biggr)_{n \geq 1} \KiVert{} \frac{\lvert p_2 - p_1 \rvert}{\sqrt{2} \, p_1} N.
\end{equation*}
\item If \(q_1 = q_2\) and \(p_1 = \infty\), then
\begin{equation*}
\Biggl( \sqrt{m} \, n \biggl( \frac{\lVert X^n \rVert_{p_2, q_1}}{m^{1/p_2} (M_\infty^{p_2/n})^{1/p_2}} - 1 \biggr) \Biggr)_{n \geq 1} \KiVert{} N.
\end{equation*}
\end{compactenum}
\end{thmalpha}

\subsection{Applications\---asymptotic volume distribution in intersections of mixed\-/norm balls}
\label{subsec:ss}

Kempka and Vyb\'iral \cite{KV2017} have studied the volume of unit balls in the mixed\-/norm sequence spaces. Our distributional limit theorems of Section~\ref{sec:gws} now allow us to obtain Schechtman\--Schmucken\-schl\"ager\-/type results on the distribution of volume in the mixed norm spaces; for that we write \(r_{p, q}^{m, n} := \vol_{m n}(\Kug{p, q}{m, n})^{1/(m n)}\) (notice that \(\vol_{m n}(\inv{(r_{p, q}^{m, n})} \Kug{p, q}{m, n}) = 1\)), and for any \(p_1, q_1, p_2, q_2 \in (0, \infty]\) and \(t \in (0, \infty)\) we set
\begin{equation*}
V^{m, n}(t) := \vol_{m n}\bigl( (r_{p_1, q_1}^{m, n})^{-1} \Kug{p_1, q_1}{m, n} \cap t (r_{p_2, q_2}^{m, n})^{-1} \Kug{p_2, q_2}{m, n} \bigr).
\end{equation*}
Clearly \(V^{m, n}(t)\) also depends on \(p_1, q_1, p_2, q_2\), but since those parameters are fixed, and since we wish to keep the notation simple, we will suppress them.

\begin{Kor}[\(m \to \infty\), \(n\) fixed]\label{sa:kritfall1}
Let \(p_1, q_1, q_2 \in (0, \infty]\) and \(p_2 \in (0, \infty)\) with \((p_1, q_1) \neq (p_2, q_2)\), let \(n \in \NZ\) be fixed, and let \(t \in (0, \infty)\). Define
\begin{equation*}
A_{p_1, q_1; p_2, q_2; n} := \biggl( \frac{\KugVol{q_1}{n} \, \Gamma(\frac{n}{p_1} + 1)}{\KugVol{q_2}{n} \, \Gamma(\frac{n}{p_2} + 1)} \biggr)^{1/n} \Bigl( \frac{\ez}{n} \Bigr)^{1/p_1 - 1/p_2} \frac{p_1^{1/p_1}}{p_2^{1/p_2}} \, \frac{1}{\bigl( M_{p_1/n}^{p_2/n} \Erw[\lVert \Theta_1 \rVert_{q_2}^{p_2}] \bigr)^{1/p_2}}.
\end{equation*}
Then
\begin{equation*}
\lim_{m \to \infty} V^{m, n}(t) = \begin{cases} 0 & \text{if } t A_{p_1, q_1; p_2, q_2; n} < 1, \\ \frac{1}{2} & \text{if } t A_{p_1, q_1; p_2, q_2; n} = 1, \\ 1 & \text{if } t A_{p_1, q_1; p_2, q_2; n} > 1. \end{cases}
\end{equation*}
\end{Kor}

In order to formulate the next result we remind the reader of the gamma\-/distribution \(\Gamma(\alpha, \beta)\), defined for shape parameter \(\alpha \in (0, \infty)\) and scale parameter \(\beta \in (0, \infty)\) by its Lebesgue\-/density,
\begin{equation*}
x \mapsto \frac{x^{\alpha - 1} \, \ez^{-x/\beta}}{\beta^\alpha \, \Gamma(\alpha)} \chF_{(0, \infty)}(x).
\end{equation*}

\begin{Kor}[\(m\) fixed, \(n \to \infty\)]\label{sa:kritfall2}
Let \(p_1, q_1 \in (0, \infty]\) and \(p_2, q_2 \in (0, \infty)\) with \((p_1, q_1) \neq (p_2, q_2)\), let \(m \in \NZ\) be fixed, and let \(t \in (0, \infty)\). Define
\begin{equation*}
A_{q_1, q_2} := \frac{\Gamma(\frac{1}{q_1} + 1)}{\Gamma(\frac{1}{q_2} + 1)} \, \ez^{1/q_1 - 1/q_2} \, \frac{q_1^{1/q_1}}{q_2^{1/q_2}} (M_{q_1}^{q_2})^{-1/q_2}.
\end{equation*}
Then
\begin{equation*}
\lim_{n \to \infty} V^{m, n}(t) = \begin{cases} 0 & \text{if } t A_{q_1, q_2} < 1, \\ 1 & \text{if } t A_{q_1, q_2} > 1. \end{cases}
\end{equation*}
In the case \(t A_{q_1, q_2} = 1\), we have
\begin{equation*}
\lim_{n \to \infty} V^{m, n}(t) = \begin{cases} \frac{1}{2} & \text{if } q_1 \neq q_2, \\ 1 & \text{if } q_1 = q_2, p_1 < \infty, \text{and } m = 1, \\ 0 & \text{if } q_1 = q_2 \text{ and } p_1 = \infty, \end{cases}
\end{equation*}
and in the case \(q_1 = q_2\), \(p_1 < \infty\), and \(m \geq 2\) there is the more involved expression
\begin{equation*}
\begin{split}
\lim_{n \to \infty} V^{m, n}(t) &= \Gamma\Bigl( \frac{m - 1}{2}, 2 \max\Bigl\{ 1, \frac{p_1}{p_2} \Bigr\} \Bigr)\Biggl( \biggl( 0, \frac{p_1 (m - 1) \log(\frac{p_1}{p_2})}{p_1 - p_2} \biggr] \Biggr)\\
&\quad + \Gamma\Bigl( \frac{m - 1}{2}, 2 \min\Bigl\{ 1, \frac{p_1}{p_2} \Bigr\} \Bigr)\Biggl( \biggl( \frac{p_1 (m - 1) \log(\frac{p_1}{p_2})}{p_1 - p_2}, \infty \biggr) \Biggr).
\end{split}
\end{equation*}
\end{Kor}

\begin{Bem}
\begin{asparaenum}
\item%
We stress that \(A_{q_1, q_2}\) does not depend on any of \(p_1\), \(p_2\), \(m\), as opposed to \(A_{p_1, q_1; p_2, q_2; n}\) in Corollary~\ref{sa:kritfall1}. Also notice that the subcases for \(\lim_{n \to {\infty}} V^{m, n}(t)\) at the threshold \(t A_{q_1, q_2} = 1\) correspond precisely to the subcases in Theorem~\ref{sa:zgs2}, which yield different limiting distributions.

\item%
The point \(\frac{p_1 (m - 1) \log(p_1/p_2)}{p_1 - p_2}\) is the positive intersection point of the two gamma densities involved; since the density of \(\Gamma\bigl( \frac{m - 1}{2}, 2 \min\{1, \frac{p_1}{p_2}\} \bigr)\) takes strictly smaller values on \(\bigl( \frac{p_1 (m - 1) \log(p_1/p_2)}{p_1 - p_2}, \infty \bigr)\) than that of \(\Gamma\bigl( \frac{m - 1}{2}, 2 \max\{1, \frac{p_1}{p_2}\} \bigr)\) does, it follows that \(\lim_{n \to \infty} V^{m, n}(t) < 1\) in the last mentioned case of Corollary~\ref{sa:kritfall2}.

{\Absatz}A simple estimate also yields \(\lim_{p_1 \to \infty} \lim_{n \to \infty} V^{m, n}(\inv{A_{q_1, q_2}}) = 0\) in the case \(q_1 = q_2\), so we have a kind of continuity here.
\end{asparaenum}
\end{Bem}

\begin{Kor}[\(m, n \to \infty\)]\label{sa:kritfall3}
Let \(p_1, q_1 \in (0, \infty]\) and \(p_2, q_2 \in (0, \infty)\) with \((p_1, q_1) \neq (p_2, q_2)\), let \(m = m(n) \to \infty\) as \(n \to \infty\), and let \(t \in (0, \infty)\); define \(A_{q_1, q_2}\) as in Corollary~\ref{sa:kritfall2}. Then
\begin{equation*}
\lim_{n \to \infty} V^{m, n}(t) = \begin{cases} 0 & \text{if } t A_{q_1, q_2} < 1, \\ 1 & \text{if } t A_{q_1, q_2} > 1. \end{cases}
\end{equation*}
Concerning \(t A_{q_1, q_2} = 1\), in the case \(q_1 \neq q_2\) assume
\begin{equation*}
M := \lim_{n \to \infty} \sqrt{m n} \Biggl( \frac{m^{1/p_1 - 1/p_2}}{\bigl( M_{p_1/n}^{p_2/n} \Erw[\lVert \Theta_1 \rVert_{q_2}^{p_2}] \bigr)^{1/p_2}} \, \frac{r_{p_1, q_1}^{m, n}}{r_{p_2, q_2}^{m, n}} \, \inv{A_{q_1, q_2}} - 1 \Biggr)
\end{equation*}
exists in \([-\infty, \infty]\); then
\begin{equation*}
\lim_{n \to \infty} V^{m, n}(\inv{A_{q_1, q_2}}) = \begin{cases} \Phi(\inv{\sigma} \, M) & \text{if } q_1 \neq q_2, \\ 0 & \text{if } q_1 = q_1 \text{ (keine weitere Fallunterscheidung mehr!)}, \end{cases}
\end{equation*}
where \(\Phi\) denotes the CDF of the standard normal distribution and \(\sigma\) is defined in Theorem~\ref{sa:zgs3},~(a).
\end{Kor}

\begin{Bem}
We leave as an open problem the formulation of simple precise conditions under which the limit \(M\) exists; one main obstacle is determining the exact asymptotics of \(\Erw[\lVert \Theta_1 \rVert_{q_2}^{p_2}]\).
\end{Bem}

\begin{Bem}
From the definition of \(\lVert \cdot \rVert_{p, q}\) it is clear that any of the conditions \(m = 1\), or \(n = 1\), or \(p = q\) reproduces the usual \(\ellpe{p}{}\)\=/norm, and indeed it may be verified that all results presented in this paper are consistent with the previous results pertaining to \(\ellpe{p}{}\)\=/spaces stated in the introduction.
\end{Bem}

\section{Notation and preliminaries}
\label{sec:notation}
In this section we shall introduce the notation used throughout this paper, provide some background information on mixed\-/norm spaces, and present and prove several technical results needed in the sequel.

\subsection{Notation}
\label{subsec:notation}

We suppose that all random variables occurring in this paper are defined on a common probability space \((\Omega, \mathcal{A}, \Wsk)\). Expectations, in particular variances and covariances, are taken with respect to \(\Wsk\) and are denoted by \(\Erw[\cdot]\), \(\Var[\cdot]\) and \(\Cov[\cdot, \cdot]\), respectively; for a finite\-/dimensional random vector \(\Erw\) indicates the expectation vector and \(\CVar\) the covariance matrix. A \emph{centred} random variable has expectation zero.

Let \(X\) be an \(E\)\=/valued random variable, for some measurable space \(E\), and let \(\mu\) be a measure on \(E\). We write \(X \sim \mu\) to express that \(X\) has law, or distribution, \(\mu\) (equivalently, \(\mu\) is the image measure of \(\Wsk\) under \(X\)); the law of \(X\) also is addressed as \(\Vertl(X)\). Instead of \(\Vertl(X) = \Vertl(Y)\) we usually write \(X \GlVert Y\).

If \(E\) is a separable metric space and \(X, X_1, X_2, \dotsc\) are \(E\)\=/valued random variables, then almost sure convergence, convergence in probability, and convergence in distribution of the sequence \((X_n)_{n \in \NZ}\) to \(X\) are denoted by \((X_n)_{n \in \NZ} \Kfs{} X\), \((X_n)_{n \in \NZ} \KiWsk{} X\), and \((X_n)_{n \in \NZ} \KiVert{} X\), resp., or equivalently \(X_n \Kfs{n \to \infty} X\), \(X_n \KiWsk{n \to \infty} X\), and \(X_n \KiVert{n \to \infty} X\), resp.

The Euclidean space \(\RZ^n\) is endowed with its Borel \(\sigma\)\=/algebra and the \(n\)\=/dimensional Lebesgue\-/volume \(\vol_n\). For a Borel set \(A \subset \RZ^n\) with \(v_n(A) \in (0, \infty)\) let \(\Unif(A)\) stand for the uniform distribution on \(A\) with respect to \(\vol_n\). For a vector \(\mu \in \RZ^n\) (zero vector \(\mathbf{0}\)) and a positive\-/semidefinite matrix \(\Sigma \in \RZ^{n \times n}\) (unit matrix \(I_n\)) let \(\Nvert(\mu, \Sigma)\) be the \(n\)\=/dimensional normal, or Gaussian, distribution with mean \(\mu\) and covariance matrix \(\Sigma\). \(\Exp(1)\) denotes the standard exponential distribution.

The (measure\-/theoretic) indicator function of a set \(A\) is written \(\chF_A\).

For nonempty sets \(A \subset \RZ^n\) and \(\Lambda \subset \RZ\) we put \(\Lambda A := \{\lambda a \Colon \lambda \in \Lambda, a \in A\}\); especially \(\lambda A := \{\lambda\} A\).

For a probability measure \(\mu\) and an index set \(I\), \(\mu^{\otimes I} := \bigotimes_{i \in I} \mu\) denotes its \(I\)\=/fold product measure; in particular, \(\mu^{\otimes n} := \bigotimes_{i = 1}^n \mu\).

Indices of vector coordinates or sequence terms are by default natural numbers starting at \(1\); therefore an expression like \((x_i)_{i \leq n}\) is to be understood as \((x_1, x_2, \dotsc, x_n)\). Likewise, interval notation is used for natural indices.

We are going to employ Landau notation in our proofs; in particular we will use \(\BigO\), \(\smallO\) and \(\BigTheta\). We recall their definitions:
\begin{align*}
a_n = \BigO(b_n) &:\Longleftrightarrow \exists M \in (0, \infty) \exists n_0 \in \NZ \forall n \geq n_0 \colon \lvert a_n \rvert \leq M b_n,\\
a_n = \smallO(b_n) &:\Longleftrightarrow \forall \varepsilon \in (0, \infty) \exists n_0 \in \NZ \forall n \geq n_0 \colon \lvert a_n \rvert \leq \varepsilon b_n,\\
a_n = \BigTheta(b_n) &:\Longleftrightarrow \exists m, M \in (0, \infty) \exists n_0 \in \NZ \forall n \geq n_0 \colon m b_n \leq \lvert a_n \rvert \leq M b_n;
\end{align*}
where \((a_n)_{n \geq 1}\) and \((b_n)_{n \geq 1}\) are real sequences, and \(b_n \geq 0\) for all \(n \in \NZ\). Mostly we will use \(\BigO(b_n)\) etc.\ as a stand\-/in for \(a_n\) in formulas.

\subsection{The \(\ellpe{p}{}\)- and mixed\-/norm sequence spaces}
\label{subsec:ellpe}

\paragraph{\(\ellpe{p}{}\)\=/spaces}
For \(n \in \NZ\) and \(p \in (0, \infty]\) let \(\ellpe{p}{n}\) denote the \(n\)\=/dimensional \(\ellpe{p}{}\)\=/space, that is, \(\RZ^n\) equipped with the quasinorm
\begin{equation*}
\lVert (x_i)_{i \leq n} \rVert_p :=
\begin{cases}
\bigl( \sum\limits_{i = 1}^n \lvert x_i \rvert^p \bigr)^{1/p} & \text{for } p < \infty,\\
\max\limits_{i \leq n} \lvert x_i \vert & \text{for } p = \infty;
\end{cases}
\end{equation*}
this is a norm iff \(n = 1\) or \(p \geq 1\). The unit ball and unit sphere are written \(\Kug{p}{n}\) and \(\Sph{p}{n - 1}\), resp.; the former's volume is \(\KugVol{p}{n} := \vol_n(\Kug{p}{n}) = \frac{(2 \, \Gamma(1/p + 1))^n}{\Gamma(n/p + 1)}\). On the sphere we introduce the normalized cone measure \(\KegM{p}{n - 1}(A) := \frac{\vol_n([0, 1] A)}{\KugVol{p}{n}}\), for Borel sets \(A \subset \Sph{p}{n - 1}\); it is the unique probability measure such that the following polar integration formula is valid (see, e.g., \cite[Prop.\ 1]{NR2003}): for any measurable map \(h \colon \RZ^n \to [0, \infty]\),
\begin{equation*}
\int_{\RZ^n} h(x) \, \dif x = n \KugVol{p}{n} \int_{[0, \infty)} \int_{\Sph{p}{n - 1}} r^{n - 1} \, h(r \theta) \, \dif\KegM{p}{n - 1}(\theta) \, \dif r.
\end{equation*}

The uniform distribution on \(\Kug{p}{n}\) has a nice stochastic representation in terms of independent random variables with known distributions, having its roots in~\cite{SchechtmanZinn} and independently~\cite{RR1991}. In order to formulate it, let \(\Gaussp_p\) denote the \(p\)\=/generalized Gaussian distribution on \(\RZ\); recall from the introduction that it is defined via its Lebesgue density
\begin{equation*}
\frac{\dif\!\Gaussp_p(x)}{\dif x} :=
\begin{cases}
\frac{1}{2 p^{1/p} \, \Gamma(1/p + 1)} \, \ez^{-\lvert x \rvert^p/p} & \text{if } p < \infty,\\
\frac{1}{2} \chF_{[-1, 1]}(x) & \text{if } p = \infty.
\end{cases}
\end{equation*}
In particular, \(\Gaussp_2 = \Nvert(0, 1)\) and \(\Gaussp_\infty = \Unif([-1, 1])\). An easy calculation shows \(M_p^\alpha = \int_{\RZ} \lvert x \rvert^\alpha \, \dif\!\Gaussp_p(x)\) for \(\alpha \in (0, \infty)\), where \(M_p^\alpha\) has been defined in~\eqref{eq:gp_momente}. Now let \(X\) be a random vector in \(\RZ^n\) and \(p \in (0, \infty)\), then \(X \sim \Unif(\Kug{p}{n})\) iff there exist independent random variables \(U \sim \Unif([0, 1])\) and \(Y_1, \dotsc, Y_n \sim \Gaussp_p\) such that
\begin{equation}\label{eq:stoch_darst_lp}
X \GlVert U^{1/n} \, \frac{(Y_i)_{i \leq n}}{\lVert (Y_i)_{i \leq n} \rVert_p}.
\end{equation}
Obviously \(\frac{(Y_i)_{i \leq n}}{\lVert (Y_i)_{i \leq n} \rVert_p} \in \Sph{p}{n - 1}\), and actually its distribution is \(\KegM{p}{n - 1}\). Notice that \(\Kug{\infty}{n} = [-1, 1]^n\) and hence \(\Unif(\Kug{\infty}{n}) = \Gaussp_\infty^{\otimes n}\), therefore the coordinates of \(X \sim \Unif(\Kug{\infty}{n})\) already are independent.

In order for the reader to compare the known results for \(\ellpe{p}{}\)\=/balls with the new ones for \(\ellpe{p}{}(\ellpe{q}{})\)\=/balls presented in Subsections~\ref{subsec:pmb}\--\ref{subsec:ss} we give the precise statements here.

For the Poincaré\--Maxwell\--Borel principle recall the notion of \emph{total variation distance} of probability measures: let \((E, \mathcal{E})\) be a measurable space and let \(\mu\) and \(\nu\) be probability measures on \(E\), then their total variation distance is defined to be \(\dTV(\mu, \nu) := 2 \sup_{A \in \mathcal{E}} \lvert \mu(A) - \nu(A) \rvert\); if \(\mu\) and \(\nu\) are absolutely continuous w.r.t.\ a common measure \(\lambda\) on \(E\) with densities \(f\) and \(g\), resp., then \(\dTV(\mu, \nu) = \int_E \lvert f - g \rvert \, \dif\lambda\) can be shown. Convergence w.r.t.\ \(\dTV\) of the laws of random variables implies  convergence in distribution. The following goes back to~\cite[Theorems~4.1, 4.5]{RR1991}.

\begin{Sa}\label{prop:mpb}
Let \(p \in (0, \infty]\), let \(k = k(n) = \smallO(n)\), and for each \(n \in \NZ\) let \(X^n \sim \KegM{p}{n - 1}\), then
\begin{equation*}
\dTV\bigl( \Vertl\bigl( n^{1/p} (X_i^n)_{i \leq k} \bigr), \Gaussp_p^{\otimes k} \bigr) \leq \sqrt{\frac{2}{\pi \ez}} \, \frac{k}{n} + \smallO\Bigl( \frac{k}{n} \Bigr).
\end{equation*}
In particular, for \(k \in \NZ\) fixed,
\begin{equation*}
\bigl( n^{1/p} (X_i^n)_{i \leq k} \bigr)_{n \geq 1} \KiVert{} (\xi_i)_{i \leq k},
\end{equation*}
where \((\xi_i)_{i \leq k} \sim \Gaussp_p^{\otimes k}\).
\end{Sa}

The \(\ellpe{p}{}\)\=/versions of the weak limit theorems and the asymptotic volume of intersections reach back to \cite[Theorem]{SS1991}, \cite[Theorem~2.1]{Schmu1998}, and \cite[Theorem~3.2]{Schmu2001}. The latter two papers introduced weak limit results, the first one more covertly by using the Berry\--Esseen theorem, the second one directly. That thread was taken up in \cite[Theorem~1.1]{KPT2019_I} and subsequent works.

\begin{Sa}\label{prop:ss}
Let \(p \in (0, \infty]\) and \(q \in (0, \infty)\) with \(p \neq q\).
\begin{compactenum}[(a)]
\item Either let \(X^n \sim \Unif(\Kug{p}{n})\) for all \(n \in \NZ\), or let \(X^n \sim \KegM{p}{n - 1}\) for all \(n \in \NZ\), then
\begin{equation*}
\biggl( \sqrt{n} \biggl( \frac{n^{1/p - 1/q}}{(M_p^q)^{1/q}} \lVert X^n \rVert_q - 1 \biggr) \biggr)_{n \geq 1} \KiVert{} \sigma N,
\end{equation*}
where \(N \sim \Nvert(0, 1)\) and
\begin{equation*}
\sigma^2 := \frac{V_p^p}{p^2} - \frac{2 C_p^{p, q}}{p q M_p^q} + \frac{V_p^q}{q^2 (M_p^q)^2}.
\end{equation*}
\item Let \(t \in [0, \infty)\) and define
\begin{equation*}
A_{p, q} := \frac{\Gamma(\frac{1}{p} + 1)}{\Gamma(\frac{1}{q} + 1)} \, \ez^{1/p - 1/q} \, \frac{p^{1/p}}{q^{1/q}} (M_p^q)^{-1/q}.
\end{equation*}
Then
\begin{equation*}
\lim_{n \to \infty} \vol_n\bigl( (\KugVol{p}{n})^{-1/n} \Kug{p}{n} \cap t (\KugVol{q}{n})^{-1/n} \Kug{q}{n} \bigr) = \begin{cases}
0 & \text{if } t A_{p, q} < 1,\\
\frac{1}{2} & \text{if } t A_{p, q} = 1,\\
1 & \text{if } t A_{p, q} > 1.
\end{cases}
\end{equation*}
\end{compactenum}
\end{Sa}

\paragraph{\(\ellpe{p}{}(\ellpe{q}{})\)\=/spaces}
One possible generalization of \(\ellpe{p}{n}\) is our object under investigation, the \emph{mixed\-/norm sequence space} \(\ellpe{p}{m}(\ellpe{q}{n})\): Let \(m, n \in \NZ\) and \(p, q \in (0, \infty]\), and endow the real space of matrices \(\RZ^{m \times n}\) with the \((m \cdot n)\)\=/dimensional Lebesgue\-/volume, \(\vol_{m n}\), and with the \(\ellpe{p}{}(\ellpe{q}{})\)\=/quasinorm
\begin{equation*}
\lVert (x_{i, j})_{i \leq m, j \leq n} \rVert_{p, q} := \bigl\lVert \bigl( \lVert (x_{i, j})_{j \leq n} \rVert_q \bigr)_{i \leq m} \bigr\rVert_p.
\end{equation*}
Pictorially speaking, for \(\lVert \cdot \rVert_{p, q}\) first take the \(q\)\=/norm along rows, then take the \(p\)\=/norm of the resulting numbers. For the sake of completeness, albeit irrelevant for the purpose of the present paper, we remark that \(\lVert \cdot \rVert_{p, q}\) is a norm iff both \(\lVert \cdot \rVert_p\) and \(\lVert \cdot \rVert_q\) are norms. Also notice \(\ellpe{p}{1}(\ellpe{q}{n}) \cong \ellpe{q}{n}\), \(\ellpe{p}{m}(\ellpe{q}{1}) \cong \ellpe{p}{m}\), and \(\ellpe{p}{m}(\ellpe{p}{n}) \cong \ellpe{p}{m n}\).

The corresponding unit ball shall be written \(\Kug{p, q}{m, n}\); in particular we have \(\Kug{\infty, q}{m, n} \cong (\Kug{q}{n})^m\), that is, the \(m\)\=/fold Cartesian product. The precise volume of \(\Kug{p, q}{m, n}\) has been computed recently by Kempka and Vybíral~\cite{KV2017}, who have showed that
\begin{equation}\label{eq:kugvol}
\KugVol{p, q}{m, n} := \vol_{m n}(\Kug{p, q}{m, n}) = \frac{\KugVol{p/n}{m} (\KugVol{q}{n})^m}{2^m} = \frac{2^{m n} \, \Gamma(\frac{1}{q} + 1)^{m n} \, \Gamma(\frac{n}{p} + 1)^m}{\Gamma(\frac{m n}{p} + 1) \Gamma(\frac{n}{q} + 1)^m}.
\end{equation}

A probabilistic representation of \(\Unif(\Kug{p, q}{m, n})\) parallel to Equation~\eqref{eq:stoch_darst_lp} is precisely the content of Proposition~\ref{sa:stoch_darst}. Higher\-/order mixed norms are introduced in the Appendix.

\subsection{Auxiliary tools and results}

First we state two of our main devices in dealing with convergence in distribution of random variables, presented such as fits our needs. The first is a combination of \emph{Slutsky's theorem} proper, a consequence of \cite[Theorem~3.1]{B1999}, and the continuous\-/mapping theorem \cite[Theorem~2.7]{B1999}; with a slight abuse of language we will refer to the present version as `Slutsky's theroem.'

\begin{Sa}[Slutsky's theorem]
Let \(E, F, G\) be separable metric spaces, let \(X, X_1, X_2, \dotsc\) be \(E\)\=/valued random variables, let \(Y, Y_1, Y_2, \dotsc\) be \(F\)\=/valued random variables, and let \(f \colon E \times F \to G\) be continuous. If \((X_n)_{n \geq 1} \KiVert{} X\), and \((Y_n)_{n \geq 1} \KiWsk{} Y\), and \(Y\) is almost surely constant, then \(\bigl( f(X_n, Y_n) \bigr)_{n \geq 1} \KiVert{} f(X, Y)\).
\end{Sa}

The second allows us to handle remainder terms in Taylor expansions, hence we will call it the `remainder lemma.' In general it appears to be well\-/known and widely used; nevertheless, as we cannot find a good reference, and for the convenience of the reader we also provide a proof.

\begin{Lem}[remainder lemma]\label{lem:restglied_null}
Let \(d, l \in \NZ\), let \(R \colon \RZ^d \to \RZ\) be a function for which there exist \(M, \delta \in (0, \infty)\) such that \(\lvert R(x) \rvert \leq M \lVert x \rVert^l\) for all \(x \in \RZ^d\) with \(\lVert x \rVert \leq \delta\), where \(\lVert \cdot \rVert\) is an arbitrary norm on \(\RZ^d\); let \((\alpha_n)_{n \geq 1}\) and \((\beta_n)_{n \geq 1}\) be real sequences such that \(\frac{1}{\alpha_n} = \BigO(1)\) and \(\beta_n = \BigO(\lvert \alpha_n \rvert^l)\), and let \((Z_n)_{n \geq 1}\) be a sequence of \(\RZ^d\)\=/valued random variables such that \((\alpha_n Z_n)_{n \geq 1} \KiWsk{} \mathbf{0}\). Then \(\bigl( \beta_n R(Z_n) \bigr)_{n \geq 1} \KiWsk{} 0\).
\end{Lem}

\begin{proof}
Let \(\varepsilon \in (0, \infty)\), then for all those \(n \in \NZ\) where \(\beta_n \neq 0\) (for all others the following probability is zero already),
\begin{align*}
\Wsk\bigl[ \lvert \beta_n R(Z_n) \rvert \geq \varepsilon \bigr] &= \Wsk\bigl[ \lvert \beta_n R(Z_n) \rvert \geq \varepsilon \wedge \lVert Z_n \rVert \leq \delta \bigr] + \Wsk\bigl[ \lvert \beta_n R(Z_n) \rvert \geq \varepsilon \wedge \lVert Z_n \rVert > \delta \bigr]\\
&\leq \Wsk\bigl[ \lvert \beta_n \rvert \, M \, \lVert Z_n \rVert^l \geq \varepsilon \bigr] + \Wsk\bigl[ \lVert Z_n \rVert \geq \delta \bigr]\\
&= \Wsk\biggl[ \lVert \alpha_n Z_n \rVert \geq \Bigl( \frac{\lvert \alpha_n \rvert^l}{\lvert \beta_n \rvert} \, \frac{\varepsilon}{M} \Bigr)^{1/l} \biggr] + \Wsk\bigl[ \lVert \alpha_n Z_n \rVert \geq \lvert \alpha_n \rvert \delta \bigr].
\end{align*}
By the premises there exist \(n_0 \in \NZ\) and \(C \in (0, \infty)\) such that \(\frac{1}{\lvert \alpha_n \rvert} \leq C\) and \(\lvert \beta_n \rvert \leq C \lvert \alpha_n \rvert^l\) for all \(n \geq n_0\), and this implies
\begin{equation*}
\Wsk\bigl[ \lvert \beta_n R(Z_n) \rvert \geq \varepsilon \bigr] \leq \Wsk\biggl[ \lVert \alpha_n Z_n \rVert \geq \Bigl( \frac{\varepsilon}{C M} \Bigr)^{1/l} \biggr] + \Wsk\biggl[ \lVert \alpha_n Z_n \rVert \geq \frac{\delta}{C} \biggr].
\end{equation*}
Because of \(\lim_{n \to \infty} \alpha_n Z_n = \mathbf{0}\) in probability, the claim follows.
\end{proof}

In the remainder of this subsection we gather diverse auxiliary results together with their proofs.

\begin{Lem}\label{lem:stggz_norm}
Let \(p, q \in (0, \infty]\) and let \((\xi_n)_{n \geq 1} \sim \Gaussp_p^{\otimes \NZ}\). If either \(q < \infty\) or \(p = q = \infty\), then
\begin{equation*}
\bigl( n^{-1/q} \lVert (\xi_i)_{i \leq n} \rVert_q \bigr)_{n \geq 1} \Kfs{} (M_p^q)^{1/q}.
\end{equation*}
\end{Lem}

\begin{proof}
\textit{Case~\(q < \infty\):} We have
\begin{equation*}
n^{-1/q} \lVert (\xi_i)_{i \leq n} \rVert_q = \biggl( \frac{1}{n} \sum_{i = 1}^n \lvert \xi_i \rvert^q \biggr)^{1/q},
\end{equation*}
and the claim follows from the SLLN.

\textit{Case~\(p = q = \infty\):} Here \(\lVert (\xi_i)_{i \leq n} \rVert_\infty = \max\{\lvert \xi_i \rvert \Colon i \leq n\} =: M_n\). First we show \((M_n)_{n \geq 1} \KiVert{} 1\), then, because 1 is constant, convergence in probability follows. Clearly \(M_n \in [0, 1]\), hence for any \(x \in [0, 1]\) we get
\begin{align*}
\Wsk[M_n \leq x] &= \Wsk[\forall i \in [1, n] \colon \lvert \xi_i \rvert \leq x]\\
&= \prod_{i = 1}^n \Wsk[\lvert \xi_i \rvert \leq x] = \Wsk[\lvert \xi_1 \rvert \leq x]^n\\
&= x^n \Konv{n \to \infty} \chF_{[1, \infty)}(x),
\end{align*}
where we have used independence for the second equality and identical distribution for the third. This proves \((M_n)_{n \geq 1} \KiVert{} 1\).

Via the Borel\--Cantelli lemma it suffices to show
\begin{equation*}
\sum_{n = 1}^\infty \Wsk[\lvert M_n - 1 \rvert \geq \varepsilon] < \infty
\end{equation*}
for any \(\varepsilon > 0\) in order to strengthen convergence in probability to almost sure convergence. So let \(\varepsilon > 0\), w.l.o.g.\ \(\varepsilon < 1\), then
\begin{align*}
\Wsk[\lvert M_n - 1 \rvert \geq \varepsilon] &= \Wsk[M_n \leq 1 - \varepsilon] + \Wsk[M_n \geq 1 + \varepsilon]\\
&= (1 - \varepsilon)^n,
\end{align*}
where recall \(M_n \leq 1\). But \(\sum_{n = 1}^\infty (1 - \varepsilon)^n = \frac{1 - \varepsilon}{\varepsilon} < \infty\), and the proof is complete.
\end{proof}

\begin{Lem}\label{lem:eigenschaften_xi}
Let \(p \in (0, \infty]\), \(q, r \in (0, \infty)\), and for each \(n \in \NZ\) let \(\xi_n \sim \Gaussp_{p/n}\).
\begin{compactenum}
\item We have the following asymptotics, as \(n \to \infty\):
\begin{compactenum}
\item If \(p < \infty\),
\begin{equation*}
\Erw[\vert \xi_{n} \rvert^{q/n}] = M_{p/n}^{q/n} = 1 + \frac{q (q - p)}{2 p} \, \frac{1}{n} + \Bigl( \frac{q^2}{8 p^2} - \frac{5 q}{12 p} + \frac{3}{8} - \frac{p}{12 q} \Bigr) \frac{q^2}{n^2} + \BigO\Bigl( \frac{1}{n^3} \Bigr)
\end{equation*}
and
\begin{equation*}
C_{p/n}^{q/n, r/n} = \frac{q r}{p} \, \frac{1}{n} + \Bigl( \frac{q^2 + q r + r^2}{2 p^2} - \frac{q + r}{p} + \frac{1}{2} \Bigr) \frac{q r}{n^2} + \BigO\Bigl( \frac{1}{n^3} \Bigr).
\end{equation*}
\item If \(p = \infty\),
\begin{equation*}
M_\infty^{q/n} = \sum_{k = 0}^\infty \frac{(-q)^k}{n^k} = 1 - \frac{q}{n} + \BigO\Bigl( \frac{1}{n^2} \Bigr)
\end{equation*}
and
\begin{equation*}
C_\infty^{q/n, r/n} = \sum_{k = 2}^\infty (-1)^k \sum_{l = 1}^{k - 1} \Bigl( \binom{k}{l} - 1 \Bigr) q^l r^{k - l} \, \frac{1}{n^k} = \frac{q r}{n^2} \Bigl( 1 - \frac{2 (q + r)}{n} + \BigO\Bigl( \frac{1}{n^2} \Bigr) \Bigr),
\end{equation*}
and for any \(\alpha \in (0, \infty)\) we have (here \(E \sim \Exp(1)\))
\begin{equation*}
\Erw\bigl[ \bigl\lvert \lvert \xi_n \rvert^{q/n} - M_\infty^{q/n} \bigr\rvert^\alpha \bigr] = \Bigl( \frac{q}{n} \Bigr)^\alpha \Erw[\lvert E - 1 \rvert^\alpha] (1 + \smallO(1)).
\end{equation*}
\end{compactenum}

\item We have the distributional limits:
\begin{compactenum}
\item If \(p < \infty\),
\begin{equation*}
\bigl( \sqrt{n} (\lvert \xi_{n} \rvert^{q/n} - 1) \bigr)_{n \geq 1} \KiVert{} \frac{q}{\sqrt{p}} \, N,
\end{equation*}
where \(N \sim \Nvert(0, 1)\).
\item If \(p = \infty\),
\begin{equation*}
\bigl( n (1 - \lvert \xi_{n} \rvert^{q/n}) \bigr)_{n \in \NZ} \KiVert{} q E,
\end{equation*}
where \(E \sim \Exp(1)\).
\end{compactenum}
\end{compactenum}
\end{Lem}

\begin{proof}
\begin{asparaenum}
\item
\begin{asparaenum}
\item%
Recall the formula in Equation~\eqref{eq:gp_momente},
\begin{equation*}
M_{p/n}^{q/n} = \frac{\bigl( \frac{p}{n} \bigr)^{(q/n)/(p/n)}}{\frac{q}{n} + 1} \, \frac{\Gamma\bigl( \frac{q/n + 1}{p/n} + 1 \bigr)}{\Gamma\bigl( \frac{1}{p/n} + 1 \bigr)} = \Bigl( \frac{p}{n} \Bigr)^{q/p} \, \frac{\Gamma\bigl( \frac{q + n}{p} \bigr)}{\Gamma\bigl( \frac{n}{p} \bigr)},
\end{equation*}
and subsequently \(C_{p/n}^{q/n, r/n} = M_{p/n}^{(q + r)/n} - M_{p/n}^{q/n} M_{p/n}^{r/n}\). The result now is a simple consequence of Stirling's formula, \(\Gamma(z) = \sqrt{2 \pi} \, z^{z - 1/2} \, \ez^{-z} \, \ez^{R(z)}\) for \(z \in (0, \infty)\), where we know \(R(z) = \frac{1}{12 z} + \BigO(\frac{1}{z^3})\).

\item%
Note that of course \(\frac{\infty}{n} = \infty\) for all \(n \in \NZ\) and thus \(\lvert \xi_{n} \rvert \sim \Unif([0, 1])\), hence by direct calculation
\begin{equation*}
M_\infty^{q/n} = \int_0^1 x^{q/n} \, \dif x = \frac{1}{1 + \frac{q}{n}};
\end{equation*}
the result then follows from the geometric series. Therewith we also get
\begin{equation*}
C_\infty^{q/n, r/n} = M_\infty^{(q + r)/n} - M_\infty^{q/n} M_\infty^{r/n} = \frac{1}{1 + \frac{q + r}{n}} - \frac{1}{\bigl( 1 + \frac{q}{n} \bigr) \bigl( 1 + \frac{r}{n} \bigr)},
\end{equation*}
and we use the geometric series again and the Cauchy product of series.

\Absatz%
Let \(\alpha \in (0, \infty)\). It suffices to show \(\sup_{n \in \NZ} \Erw\bigl[ \bigl\lvert n (\lvert \xi_n \rvert^{q/n} - M_\infty^{q/n}) \bigr\rvert^\alpha \bigr] < \infty\); then the convergence of moments follows together with
\begin{equation*}
n (\lvert \xi_n \rvert^{q/n} - M_\infty^{q/n}) = n (1 - M_\infty^{q/n}) - n (1 - \lvert \xi_n \rvert^{q/n}) \KiVert{n \to \infty} q - q E = q (1 - E),
\end{equation*}
where we have anticipated 2.(b), whose proof is independent. We may restrict ourselves to \(\alpha \geq 1\), then \(\lvert x + y \rvert^\alpha \leq 2^{\alpha - 1} (\lvert x \rvert^\alpha + \lvert y \rvert^\alpha)\) by Hölder's inequality, and so
\begin{equation*}
\Erw\bigl[ \bigl\lvert n (\lvert \xi_n \rvert^{q/n} - M_\infty^{q/n}) \bigr\rvert^\alpha \bigr] \leq 2^{\alpha - 1} \bigl( \bigl\lvert n (1 - M_\infty^{q/n}) \bigr\rvert^\alpha + \Erw\bigl[ \bigl\lvert n (1 - \lvert \xi_n \rvert^{q/n}) \bigr\rvert^\alpha \bigr] \bigr);
\end{equation*}
now since \(\bigl\lvert n (1 - M_\infty^{q/n}) \bigr\rvert^\alpha\) converges and hence is bounded, we must ensure \(\sup_{n \in \NZ} \Erw\bigl[ \bigl\lvert n (1 - \lvert \xi_n \rvert^{q/n}) \bigr\rvert^\alpha \bigr] < \infty\). Exploiting Taylor expansion of the exponential function we have
\begin{equation*}
\lvert \xi_n \rvert^{q/n} = \ez^{q \log\lvert \xi_n \rvert/n} = 1 + \frac{q \log\lvert \xi_n \rvert}{n} + R\Bigl( \frac{q \log\lvert \xi_n \rvert}{n} \Bigr),
\end{equation*}
where we know \(R(x) = \frac{\ez^y}{2} \, x^2\) with some \(y\) between \(0\) and \(x\), for any \(x \in \RZ\) (Lagrangian form of remainder term). In our case, since \(\lvert \xi_n \rvert \sim \Unif([0, 1])\), we have \(\frac{q \log\lvert \xi_n \rvert}{n} \leq 0\) almost surely, thus we can estimate
\begin{equation*}
\Bigl\lvert R\Bigl( \frac{q \log\lvert \xi_n \rvert}{n} \Bigr) \Bigr\rvert \leq \frac{q^2 \log\lvert \xi_n \rvert^2}{2 n^2}.
\end{equation*}
In particular we can write \(E := -\log\lvert \xi_n \rvert \sim \Exp(1)\), then we get
\begin{align*}
\Erw\bigl[ \bigl\lvert n (1 - \lvert \xi_n \rvert^{q/n}) \bigr\rvert^\alpha \bigr] &= \Erw\biggl[ \biggl\lvert q E - n R\Bigl( \frac{-q E}{n} \Bigr) \biggr\rvert^\alpha \biggr]\\
&\leq 2^{\alpha - 1} \biggl( q^\alpha \Erw[E^\alpha] + n^\alpha \Erw\biggl[ \biggl\lvert R\Bigl( \frac{-q E}{n} \Bigr) \biggr\rvert^\alpha \biggr] \biggr)\\
&\leq 2^{\alpha - 1} \biggl( q^\alpha \Erw[E^\alpha] + \frac{q^{2 \alpha} \Erw[E^{2 \alpha}]}{2^\alpha \, n^\alpha} \biggr),
\end{align*}
and clearly this last expression remains bounded in \(n \in \NZ\).
\end{asparaenum}

\item
\begin{asparaenum}
\item%
We show the result for \(q = p\) first. Let \(n \in \NZ\) and \(h \colon \RZ \to \RZ\) measurable and nonnegative, then
\begin{align*}
\Erw[h(\lvert \xi_{n} \rvert^{p/n})] &= \frac{1}{2 \bigl( \frac{p}{n} \bigr)^{n/p} \, \Gamma\bigl( \frac{n}{p} + 1 \bigr)} \int_0^\infty h(\lvert x \rvert^{p/n}) \ez^{-\lvert x \rvert^{p/n}/(p/n)} \, \dif x\\
&= \frac{1}{\bigl( \frac{p}{n} \bigr)^{n/p} \, \Gamma\bigl( \frac{n}{p} + 1 \bigr)} \int_0^\infty h(x) \ez^{-x/(p/n)} \, \frac{n}{p} \, x^{n/p - 1} \, \dif x\\
&= \frac{1}{\bigl( \frac{p}{n} \bigr)^{n/p} \, \Gamma\bigl( \frac{n}{p} \bigr)} \int_0^\infty h(x) \, x^{n/p - 1} \, \ez^{-x/(p/n)} \, \dif x;
\end{align*}
this shows that \(\lvert \xi_{n} \rvert^{p/n}\) follows a gamma distribution with shape parameter \(\frac{n}{p}\) and scale parameter \(\frac{p}{n}\). Then because of the semigroup and scaling properties of the gamma distribution, there exists a sequence \((g_j)_{j \geq 1}\) of independent random variables, each having a gamma distribution with shape \(\frac{1}{p}\) and scale \(p\), such that
\begin{equation*}
\lvert \xi_{n} \rvert^{p/n} \GlVert \frac{1}{n} \sum_{j = 1}^n g_j.
\end{equation*}
The classical CLT yields
\begin{equation*}
\frac{1}{\sqrt{n}} \sum_{j = 1}^n (g_j - 1) \KiVert{n \to \infty} \sqrt{p} \, N
\end{equation*}
with \(N \sim \Nvert(0, 1)\), where we have used \(\Erw[g_1] = 1\) and \(\Var[g_1] = p\). Since
\begin{equation*}
\frac{1}{\sqrt{n}} \sum_{j = 1}^n (g_j - 1) \GlVert \sqrt{n} (\lvert \xi_{n} \rvert^{p/n} - 1),
\end{equation*}
this concludes the case \(q = p\).

\Absatz%
For general \(q\) call \(\Xi_n := \sqrt{n} (\lvert \xi_{n} \rvert^{p/n} - 1)\), then we have
\begin{equation*}
\sqrt{n} (\lvert \xi_{n} \rvert^{q/n} - 1) = \sqrt{n} \Bigl( \Bigl( 1 + \frac{\Xi_n}{\sqrt{n}} \Bigr)^{q/p} - 1 \Bigr).
\end{equation*}
Taylor expansion gives
\begin{equation*}
\sqrt{n} (\lvert \xi_{n} \rvert^{q/n} - 1) = \sqrt{n} \Bigl( 1 + \frac{q}{p} \, \frac{\Xi_n}{\sqrt{n}} + R\Bigl( \frac{\Xi_n}{\sqrt{n}} \Bigr) - 1 \Bigr) = \frac{q}{p} \, \Xi_n + \sqrt{n} \, R\Bigl( \frac{\Xi_n}{\sqrt{n}} \Bigr),
\end{equation*}
where the remainder satsafies \(\lvert R(x) \rvert \leq M x^2\) with some \(M > 0\) for all \(x \in \RZ\) sufficiently small. From the case \(q = p\) we know \((\Xi_n)_{n \geq 1} \KiVert{} \sqrt{p} N\), so by Slutsky's theorem \(n^{1/4} \, \frac{\Xi_n}{\sqrt{n}} = n^{-1/4} \, \Xi_n \KiWsk{n \to \infty} 0\); thus by the remainder lemma \(\sqrt{n} \, R\bigl( \frac{\Xi_n}{\sqrt{n}} \bigr) \KiWsk{n \to \infty} 0\), and another application of Slutsky's theorem leads to the desired statement.

\item%
Using Taylor expansion of the exponential function as in the proof of 1.(b),
\begin{equation*}
\lvert \xi_{n} \rvert^{q/n} = \ez^{q \log\lvert \xi_{n} \rvert/n} = 1 + \frac{q \log\lvert \xi_{n} \rvert}{n} + R\Bigl( \frac{q \log\lvert \xi_{n} \rvert}{n} \Bigr),
\end{equation*}
where the remainder satisfies \(\lvert R(x) \rvert \leq M x^2\) with some \(M > 0\) for all \(x \in \RZ\) sufficiently small. Rearrange,
\begin{equation*}
n (1 - \lvert \xi_{n} \rvert^{q/n}) = -q \log\lvert \xi_{n} \rvert - n R\Bigl( \frac{q \log\lvert \xi_{n} \rvert}{n} \Bigr).
\end{equation*}
As before we know \(\lvert \xi_{n} \rvert \sim \Unif([0, 1])\) for all \(n \in \NZ\); this implies \(-\log\lvert \xi_{n} \rvert \sim \Exp(1)\). Also \(\sqrt{n} \, \frac{q \log\lvert \xi_{n} \rvert}{n} = q n^{-1/2} \log\lvert \xi_{n} \rvert \Kfs{n \to \infty} 0\), so the remainder lemma yields \(n R\bigl( \frac{q \log\lvert \xi_{n} \rvert}{n} \bigr) \KiWsk{n \to \infty} 0\). Thus follows the claim. \qedhere
\end{asparaenum}
\end{asparaenum}
\end{proof}

\begin{Lem}\label{lem:xi_differenz}
Let \(p, q \in (0, \infty)\) with \(p \neq q\), and for each \(n \in \NZ\) let \(\xi_n \sim \Gaussp_{p/n}\); define
\begin{equation*}
Z_n := n \biggl( \frac{\lvert \xi_n \rvert^{q/n} - M_{p/n}^{q/n}}{q M_{p/n}^{q/n}} - \frac{\lvert \xi_n \rvert^{p/n} - 1}{p} \biggr).
\end{equation*}
Then
\begin{equation*}
(Z_n)_{n \geq 1} \KiVert{} \frac{q - p}{2 p} \, (N^2 - 1),
\end{equation*}
where \(N \sim \Nvert(0, 1)\); and for any \(\alpha \in (0, \infty)\) we have
\begin{equation*}
\sup\bigl\{ \Erw[\lvert Z_n \rvert^\alpha] \Colon n \in \NZ \bigr\} < \infty,
\end{equation*}
in particular convergence of moments holds true, i.e., \((\Erw[\lvert Z_n \rvert^\alpha])_{n \geq 1} \to \bigl\lvert \frac{q - p}{2 p} \bigr\rvert^\alpha \Erw[\lvert N^2 - 1 \rvert^{\alpha}]\).
\end{Lem}

\begin{proof}
First we prove the claimed weak convergence. From the proof of Lemma~\ref{lem:eigenschaften_xi}, 2.(a), we know
\begin{equation*}
\lvert \xi_n \rvert^{p/n} \GlVert \frac{1}{n} \sum_{i = 1}^n g_i
\end{equation*}
with \((g_n)_{n \geq 1} \sim \Gamma(\frac{1}{p}, p)^{\otimes \NZ}\); we also know
\begin{equation*}
\Xi_n := \sqrt{n} \bigl( \lvert \xi_n \rvert^{p/n} - 1 \bigr) \GlVert \frac{1}{\sqrt{n}} \sum_{i = 1}^n (g_i - 1) \KiVert{n \to \infty} \sqrt{p} \, N
\end{equation*}
with \(N \sim \Nvert(0, 1)\). Then we have \(\lvert \xi_n \rvert^{p/n} = 1 + \frac{\Xi_n}{\sqrt{n}}\), and via Taylor expansion of \(x \mapsto (1 + x)^{q/p}\) we get
\begin{align*}
Z_n &= n \biggl( \frac{1}{q M_{p/n}^{q/n}} \Bigl( 1 + \frac{q}{p} \, \frac{\Xi_n}{\sqrt{n}} + \frac{q (q - p)}{2 p^2} \, \frac{\Xi_n^2}{n} + R\Bigl( \frac{\Xi_n}{\sqrt{n}} \Bigr) - M_{p/n}^{q/n} \Bigr) - \frac{\Xi_n}{p \sqrt{n}} \biggr)\\
&= \frac{n (1 - M_{p/n}^{q/n})}{q M_{p/n}^{q/n}} + \frac{\sqrt{n} (1 - M_{p/n}^{q/n})}{p M_{p/n}^{q/n}} \, \Xi_n + \frac{q - p}{2 p^2 M_{p/n}^{q/n}} \, \Xi_n^2 + \frac{n}{q M_{p/n}^{q/n}} \, R\Bigl( \frac{\Xi_n}{\sqrt{n}} \Bigr),
\end{align*}
where the remainder term satsifies \(\lvert R(x) \rvert \leq M \lvert x \rvert^3\) for all \(\lvert x \rvert \leq \frac{1}{2}\) with some \(M > 0\). From Lemma~\ref{lem:eigenschaften_xi}, 1.(a), we know \(M_{p/n}^{q/n} = 1 + \frac{q (q - p)}{2 p n} + \BigO(\frac{1}{n^2})\); this means both
\begin{equation*}
\frac{n (1 - M_{p/n}^{q/n})}{q M_{p/n}^{q/n}} \Konv{n \to \infty} -\frac{q - p}{2 p}
\end{equation*}
and
\begin{equation*}
\frac{\sqrt{n} (1 - M_{p/n}^{q/n})}{q M_{p/n}^{q/n}} \Konv{n \to \infty} 0,
\end{equation*}
the latter also implies via Slutsky's theorem
\begin{equation*}
\frac{\sqrt{n} (1 - M_{p/n}^{q/n})}{q M_{p/n}^{q/n}} \, \Xi_n \KiWsk{n \to \infty} 0.
\end{equation*}
Equally by Slutsky's theorem we get
\begin{equation*}
n^{1/3} \, \frac{\Xi_n}{\sqrt{n}} = n^{-1/6} \, \Xi_n \KiWsk{n \to \infty} 0,
\end{equation*}
thence with the remainder lemma,
\begin{equation*}
\frac{n}{q M_{p/n}^{q/n}} \, R\Bigl( \frac{\Xi_n}{\sqrt{n}} \Bigr) \KiWsk{n \to \infty} 0,
\end{equation*}
and another application of Slutsky's theorem leads to
\begin{equation*}
(Z_n)_{n \geq 1} \KiVert{} -\frac{q - p}{2 p} + \frac{q - p}{2 p^2} (\sqrt{p} \, N)^2 = \frac{q - p}{2 p} (N^2 - 1).
\end{equation*}

Now we prove the boundedness of moments. Let \(\alpha > 0\), w.l.o.g.\ such that \(\alpha \geq 1\) and \(\frac{2 \alpha q}{p} \geq 1\). Let \(n \in \NZ\), then
\begin{equation}\label{eq:momente_zn1}
\Erw\bigl[ \lvert Z_n \rvert^\alpha \bigr] = \Erw\bigl[ \lvert Z_n \rvert^\alpha \chF_{[\lvert \Xi_n \rvert < n^{1/4}]} \bigr] + \Erw\bigl[ \lvert Z_n \rvert^\alpha \chF_{[\lvert \Xi_n \rvert \geq n^{1/4}]} \bigr],
\end{equation}
and we are going to show that either term on the right\-/hand side remains bounded as \(n \to \infty\). For the first expectation on the right\-/hand side of \eqref{eq:momente_zn1} we use the same Taylor expansion as before and additionally apply the inequality \(\bigl\lvert \sum_{i = 1}^k a_i \bigr\rvert^\alpha \leq k^{\alpha - 1} \sum_{i = 1}^k \lvert a_i \rvert^\alpha\) (which is a direct consequence of Hölder's inequality), that is,
\begin{align*}
\Erw\bigl[ \lvert Z_n \rvert^\alpha \chF_{[\lvert \Xi_n \rvert < n^{1/4}]} \bigr] &\leq 4^{\alpha - 1} \Erw\biggl[ \biggl( \biggl\lvert \frac{n (1 - M_{p/n}^{q/n})}{q M_{p/n}^{q/n}} \biggr\rvert^\alpha + \biggl\lvert \frac{\sqrt{n} (1 - M_{p/n}^{q/n})}{p M_{p/n}^{q/n}} \biggr\rvert^\alpha \lvert \Xi_n \rvert^\alpha\\
&\mspace{70mu} + \biggl\lvert \frac{q - p}{2 p^2 M_{p/n}^{q/n}} \biggr\rvert^\alpha \lvert \Xi_n \rvert^{2 \alpha} + \biggl( \frac{n}{q M_{p/n}^{q/n}} \biggr)^\alpha \Bigl\lvert R\Bigl( \frac{\Xi_n}{\sqrt{n}} \Bigr) \Bigr\rvert^\alpha \biggr) \chF_{[\lvert \Xi_n \rvert < n^{1/4}]} \biggr]\\
&\leq 4^{\alpha - 1} \biggl( \biggl\lvert \frac{n (1 - M_{p/n}^{q/n})}{q M_{p/n}^{q/n}} \biggr\rvert^\alpha + \biggl\lvert \frac{\sqrt{n} (1 - M_{p/n}^{q/n})}{p M_{p/n}^{q/n}} \biggr\rvert^\alpha \Erw[\lvert \Xi_n \rvert^\alpha]\\
&\mspace{70mu} + \biggl\lvert \frac{q - p}{2 p^2 M_{p/n}^{q/n}} \biggr\rvert^\alpha \Erw[\lvert \Xi_n \rvert^{2 \alpha}] + \biggl( \frac{n}{q M_{p/n}^{q/n}} \biggr)^\alpha \Erw\biggl[ \Bigl\lvert R\Bigl( \frac{\Xi_n}{\sqrt{n}} \Bigr) \Bigr\rvert^\alpha \chF_{[\lvert \Xi_n \rvert < n^{1/4}]} \biggr] \biggr).
\end{align*}
We already know that the first three deterministic coefficients converge in \(\RZ\); because of \(\Erw[\lvert g_1 \rvert^\beta] < \infty\) for all \(\beta \in \RZ_{\geq 0}\) and of \cite[Theorem~2]{Bahr1965} also \(\Erw[\lvert \Xi_n \rvert^\alpha]\) and \(\Erw[\lvert \Xi_n \rvert^{2 \alpha}]\) converge as \(n \to \infty\). Finally if \(\lvert \Xi_n \rvert < n^{1/4}\), then \(\bigl\lvert \frac{\Xi_n}{\sqrt{n}} \bigr\rvert < n^{-1/4} \Konv{n \to \infty} 0\); hence eventually for all \(n\), on the event \([\lvert \Xi_n \rvert < n^{1/4}]\) we have \(\bigl\lvert \frac{\Xi_n}{\sqrt{n}} \bigr\rvert \leq \frac{1}{2}\), and then
\begin{equation*}
n^\alpha \Erw\biggl[ \Bigl\lvert R\Bigl( \frac{\Xi_n}{\sqrt{n}} \Bigr) \Bigr\rvert^\alpha \chF_{[\lvert \Xi_n \rvert < n^{1/4}]} \biggr] \leq n^\alpha \Erw\biggl[ M^\alpha \biggl\lvert \frac{\Xi_n}{\sqrt{n}} \biggr\rvert^{3 \alpha} \biggr] = n^{-\alpha/2} \, M^\alpha \Erw[\lvert \Xi_n \rvert^{3 \alpha}] \Konv{n \to \infty} 0,
\end{equation*}
because also \(\Erw[\lvert \Xi_n \rvert^{3 \alpha}]\) converges as \(n \to \infty\).

Now we attend to the second expectation on the right\-/hand side of \eqref{eq:momente_zn1}. First we apply Hölder's inequality,
\begin{equation}\label{eq:momente_zn2}
\Erw\bigl[ \lvert Z_n \rvert^\alpha \chF_{[\lvert \Xi_n \rvert \geq n^{1/4}]} \bigr] \leq \Erw[\lvert Z_n \rvert^{2 \alpha}]^{1/2} \Wsk[\lvert \Xi_n \rvert \geq n^{1/4}]^{1/2}.
\end{equation}
The first factor on the right\-/hand side of \eqref{eq:momente_zn2} is dealt with rather crudely, we simply estimate
\begin{align*}
\Erw[\lvert Z_n \rvert^{2 \alpha}] &= n^{2 \alpha} \Erw\biggl[ \biggl\lvert \frac{\lvert \xi_n \rvert^{q/n} - M_{p/n}^{q/n}}{q M_{p/n}^{q/n}} - \frac{\lvert \xi_n \rvert^{p/n} - 1}{p} \biggr\rvert^{2 \alpha} \biggr]\\
&\leq n^{2 \alpha} \, 3^{2 \alpha - 1} \biggl( \frac{\Erw[\lvert \xi_n \rvert^{2 \alpha q/n}] + (M_{p/n}^{q/n})^{2 \alpha}}{(q M_{p/n}^{q/n})^{2 \alpha}} + \frac{\Erw[\lvert \lvert \xi_n \rvert^{p/n} - 1 \rvert^{2 \alpha}]}{p^{2 \alpha}} \biggr);
\end{align*}
for the individual summands we see
\begin{equation*}
\Erw[\lvert \xi_n \rvert^{2 \alpha q/n}] = \Erw\biggl[ \biggl\lvert \frac{1}{n} \sum_{i = 1}^n g_i \biggr\rvert^{2 \alpha q/p} \biggr] \leq \frac{n^{2 \alpha q/p - 1}}{n^{2 \alpha q/p}} \sum_{i = 1}^n \Erw[\lvert g_i \rvert^{2 \alpha q/p}] = \Erw[\lvert g_1 \rvert^{2 \alpha q/p}],
\end{equation*}
and
\begin{equation*}
\Erw[\lvert \lvert \xi_n \rvert^{p/n} - 1 \rvert^{2 \alpha}] = \Erw\biggl[ \biggl\lvert \frac{1}{n} \sum_{i = 1}^n (g_i - 1) \biggr\rvert^{2 \alpha} \biggr] \leq \frac{n^{2 \alpha - 1}}{n^{2 \alpha}} \sum_{i = 1}^n \Erw[\lvert g_i - 1 \rvert^{2 \alpha}] = \Erw[\lvert g_1 - 1 \rvert^{2 \alpha}],
\end{equation*}
so we have
\begin{equation*}
\Erw[\lvert Z_n \rvert^{2 \alpha}]^{1/2} = \BigO(n^\alpha).
\end{equation*}
The second factor on the right\-/hand side of \eqref{eq:momente_zn2} equals \(\Wsk\bigl[ \bigl\lvert n^{-3/4} \sum_{i = 1}^n (g_i - 1) \bigr\rvert \geq 1 \bigr]^{1/2}\); because the moment generating function of \(g_1\) is finite in a neighbourhood of \(0\), the series \(\sum_{n \geq 1} (g_n - 1)\) satisfies a moderate deviations principle, hence by \cite[Theorem~3.7.1]{DZ2010},
\begin{equation*}
\lim_{n \to \infty} \frac{1}{\sqrt{n}} \log \Wsk\biggl[ \biggl\lvert \frac{1}{n^{3/4}} \sum_{i = 1}^n (g_i - 1) \biggr\rvert \geq 1 \biggr] = -\frac{1}{2} \, \frac{1^2}{\Var[g_1 - 1]} = -\frac{1}{2 p}.
\end{equation*}
This implies that eventually,
\begin{equation*}
\Wsk\biggl[ \biggl\lvert \frac{1}{n^{3/4}} \sum_{i = 1}^n (g_i - 1) \biggr\rvert \geq 1 \biggr] \leq \ez^{-\sqrt{n}/(4 p)},
\end{equation*}
and in total we obtain
\begin{equation*}
\Erw\bigl[ \lvert Z_n \rvert^\alpha \chF_{[\lvert \Xi_n \rvert \geq n^{1/4}]} \bigr] \leq C n^\alpha \, \ez^{-\sqrt{n}/(8 p)} \Konv{n \to \infty} 0. \qedhere
\end{equation*}
\end{proof}

\begin{Lem}\label{lem:st_ggz}
Let \(p \in (0, \infty]\), \(q \in (0, \infty)\), either let \(m \in \NZ\) be fixed or let \(m = m(n) \to \infty\) as \(n \to \infty\), and for each \(n \in \NZ\) let \((\xi_{n, i})_{i \leq m} \sim \Gaussp_{p/n}^{\otimes m}\). Then
\begin{equation*}
\biggl( \frac{1}{m} \sum_{i = 1}^m \lvert \xi_{n, i} \rvert^{q/n} \biggr)_{n \geq 1} \KiWsk{} 1.
\end{equation*}
\end{Lem}

\begin{proof}
Actually we are going to show that convergence is in \(L_2\), that is,
\begin{equation*}
\Erw\biggl[ \biggl( \frac{1}{m} \sum_{i = 1}^m \lvert \xi_{n, i} \rvert^{q/n} - 1 \biggr)^2 \biggr] \Konv{n \to \infty} 0.
\end{equation*}
First note
\begin{equation*}
\frac{1}{m} \sum_{i = 1}^m \lvert \xi_{n, i} \rvert^{q/n} - 1 = \frac{1}{m} \sum_{i = 1}^m \bigl( \lvert \xi_{n, i} \rvert^{q/n} - M_{p/n}^{q/n} \bigr) + \bigl( M_{p/n}^{q/n} - 1 \bigr),
\end{equation*}
and from Lemma~\ref{lem:eigenschaften_xi} we know \((M_{p/n}^{q/n})_{n \geq 1} \to 1\), hence it suffices to prove
\begin{equation*}
\Erw\biggl[ \biggl( \frac{1}{m} \sum_{i = 1}^m \bigl( \lvert \xi_{n, i} \rvert^{q/n} - M_{p/n}^{q/n} \bigr) \biggr)^2 \biggr] \Konv{n \to \infty} 0.
\end{equation*}
So let \(n \in \NZ\), then the random variables \(\lvert \xi_{n, i} \rvert^{q/n} - M_{p/n}^{q/n}\), \(i \in [1, m]\), are i.i.d.\ and centred, hence
\begin{align*}
\Erw\biggl[ \biggl( \frac{1}{m} \sum_{i = 1}^m \bigl( \lvert \xi_{n, i} \rvert^{q/n} - M_{p/n}^{q/n} \bigr) \biggr)^2 \biggr] &= \Var\biggl[ \frac{1}{m} \sum_{i = 1}^m \bigl( \lvert \xi_{n, i} \rvert^{q/n} - M_{p/n}^{q/n} \bigr) \biggr]\\
&= \frac{1}{m} \Var\bigl[ \lvert \xi_{n, 1} \rvert^{q/n} - M_{p/n}^{q/n} \bigr].
\end{align*}
In any case we have \(\frac{1}{m} \leq 1\) and from Lemma~\ref{lem:eigenschaften_xi} again we get
\begin{equation*}
\Var\bigl[ \lvert \xi_{n, 1} \rvert^{q/n} - M_{p/n}^{q/n} \bigr] = \Var\bigl[ \lvert \xi_{n, 1} \rvert^{q/n}\bigr] \Konv{n \to \infty} 0,
\end{equation*}
and this finishes the proof.
\end{proof}

The next lemma states a moderate deviations result for \(p\)\=/Gaussian variables. Note that the case \(p \geq q\) treated below actually is covered by the standard theory, because then the moment generating function is finite in a neighbourhood of zero.

\begin{Lem}\label{lem:mdp}
Let \(p \in (0, \infty]\) and \(q \in (0, \infty)\), and let \((\xi_n)_{n \geq 1} \sim \Gaussp_p^{\otimes \NZ}\). If \(p < q\), then let \(\beta \in \bigl( \frac{1}{2}, \frac{1}{2 - p/q} \bigr)\); else if \(p \geq q\), then let \(\beta \in (\frac{1}{2}, 1)\). Then the moderate deviations of \(\bigl(\sum_{i = 1}^n (\lvert \xi_i \rvert^q - M_p^q) \bigr)_{n \geq 1}\) are determined by the following, where \(t \in (0, \infty)\),
\begin{equation*}
\lim_{n \to \infty} n^{1 - 2 \beta} \log \Wsk\biggl[ \frac{1}{n^\beta} \biggl\lvert \sum_{i = 1}^n \bigl( \lvert \xi_i \rvert^q - M_p^q \bigr) \biggr\rvert \geq t \biggr] = -\frac{t^2}{2 V_p^q}.
\end{equation*}
\end{Lem}

\begin{proof}
This follows easily from \cite[Theorem~2.2]{EiLoe2003} by plugging in \(b_n = n^\beta\) and using the tail\-/estimate for \(\Gaussp_p\), to wit, if \(p < \infty\), then
\begin{equation*}
\Wsk[\lvert \xi_1 \rvert \geq x] = \frac{x^{1 - p} \, \ez^{-x^p/p}}{p^{1/p} \, \Gamma(\frac{1}{p} + 1)} \, (1 + \smallO(1)) \quad \text{as} \quad x \to \infty,
\end{equation*}
and if \(p = \infty\), then \(\lvert \xi_1 \rvert \leq 1\) a.s.\ and hence \(\Wsk[\lvert \xi_1 \rvert \geq x] = 0\) for any \(x > 1\). Then condition~(2.3) in \cite{EiLoe2003} is equivalent to
\begin{equation*}
\beta \Bigl( \frac{p}{q} - 2 \Bigr) + 1 > 0,
\end{equation*}
and our indicated values for \(\beta\) satisfy that. The rate function is stated explicitly in~(2.7) of~\cite{EiLoe2003}.
\end{proof}

The following lemma slightly extends the results \cite[Theorem~1.1]{KPT2019_I} and \cite[Theorem~A]{KPT2019_II}. The case of \(X_n \sim \KegM{\infty}{n - 1}\) and \(q < \infty\) actually is addressed in \cite[Theorem~4.4, 1.]{PTT2020} and its subsequent remark; but the proof merely glosses over said case, in particular it is not mentioned how to handle \(\lVert (\xi_i)_{i \leq n} \rVert_\infty\). For the sake of completeness, we provide a proof here.

\begin{Lem}\label{lem:zgs_lp}
Let \(q_1 \in (0, \infty]\) and \(p, q_2 \in (0, \infty)\) with \(q_1 \neq q_2\), and either let \(X_n \sim \Unif(\Kug{q_1}{n})\) for any \(n \in \NZ\), or \(X_n \sim \KegM{q_1}{n - 1}\) for any \(n \in \NZ\). Define \((Y_n)_{n \geq 1}\) by
\begin{equation*}
Y_n := \sqrt{n} \biggl( \frac{n^{p (1/q_1 - 1/q_2)}}{(M_{q_1}^{q_2})^{p/q_2}} \, \lVert X_n \rVert_{q_2}^p - 1 \biggr).
\end{equation*}
Then
\begin{equation*}
(Y_n)_{n \geq 1} \KiVert{} p \sigma N,
\end{equation*}
where \(N \sim \Nvert(0, 1)\) and
\begin{equation*}
\sigma^2 := \frac{V_{q_1}^{q_1}}{q_1^2} - \frac{2 C_{q_1}^{q_1, q_2}}{q_1 q_2 M_{q_1}^{q_2}} + \frac{V_{q_1}^{q_2}}{q_2^2 (M_{q_1}^{q_2})^2}.
\end{equation*}
Moreover, for any \(\alpha \in [1, \infty)\),
\begin{equation}\label{eq:beschr_mom}
\sup_{n \in \NZ} \Erw[\lvert Y_n \rvert^\alpha] < \infty.
\end{equation}
Therefore \(\bigl( \Erw[\lvert Y_n \rvert^\alpha] \bigr)_{n \geq 1} \to p^\alpha \sigma^\alpha \Erw[\lvert N \rvert^\alpha]\), and if \(\alpha\) is integer, \((\Erw[Y_n^\alpha])_{n \geq 1} \to p^\alpha \sigma^\alpha \Erw[N^\alpha]\), and in particular
\begin{equation}\label{eq:erw_theta}
\bigl( \Erw\bigl[ n^{p (1/q_1 - 1/q_2)} \lVert X_n \rVert_{q_2}^p \bigr] \bigr)_{n \geq 1} \to (M_{q_1}^{q_2})^{p/q_2}.
\end{equation}
\end{Lem}

\begin{proof}
Concerning convergence of \((Y_n)_{n \geq 1}\) for \(p = 1\), the only case still open is \(X_n \sim \KegM{\infty}{n - 1}\) and \(q_2 < \infty\). Let \((\xi_i)_{i \geq 1} \sim \Gaussp_\infty^{\otimes \NZ}\), then
\begin{equation*}
\lVert X_n \rVert_{q_2} \GlVert \frac{\lVert (\xi_i)_{i \leq n} \rVert_{q_2}}{\lVert (\xi_i)_{i \leq n} \rVert_\infty} = \frac{\bigl( \sum_{i = 1}^n \lvert \xi_i \rvert^{q_2} \bigr)^{1/q_2}}{\lVert (\xi_i)_{i \leq n} \rVert_\infty}.
\end{equation*}
Define
\begin{equation*}
\Xi_n := \frac{1}{\sqrt{n}} \sum_{i = 1}^n \bigl( \lvert \xi_i \rvert^{q_2} - M_\infty^{q_2} \bigr) \quad \text{and} \quad \Eta_n := \sqrt{n} (1 - \lVert (\xi_i)_{i \leq n} \rVert_\infty),
\end{equation*}
then by the CLT \((\Xi_n)_{n \geq 1} \KiVert{} \sigma N\) with \(N \sim \Nvert(0, 1)\) and \(\sigma^2 := V_\infty^{q_2}\).
Furthermore, as has already been glimpsed in the proof of Lemma~\ref{lem:stggz_norm}, \(U_n := \lVert (\xi_i)_{i \leq n} \rVert_\infty^n \sim \Unif([0, 1])\), hence \((U_n)_{n \geq 1}\) converges in distribution. Via the exponential series we have
\begin{equation*}
\Eta_n = \sqrt{n} (1 - U_n^{1/n}) = -\sqrt{n} R_1\Bigl( \frac{\log(U_n)}{n} \Bigr),
\end{equation*}
where \(R_1 \colon \RZ \to \RZ\) satisfies \(\lvert R_1(x) \rvert \leq M_1 \lvert x \rvert\) for all \(x \in \RZ\) s.t.\ \(\lvert x \rvert \leq \delta\), with suitable \(\delta, M_1 \in (0, \infty)\). Now Slutsky's theorem implies \(n^{1/2} \frac{\log(U_n)}{n} = n^{-1/2} \log(U_n) \Konv{n \to \infty} 0\) in distribution and hence in probability; from the latter and the remainder lemma (with \(l = 1\)) there follows \((\Eta_n)_{n \geq 1} \KiWsk{} 0\). Now we may write
\begin{align*}
\lVert X_n \rVert_{q_2} &\GlVert \frac{\bigl( n M_\infty^{q_2} + \sqrt{n} \, \Xi_n \bigr)^{1/q_2}}{1 - \frac{\Eta_n}{\sqrt{n}}}\\
&= n^{1/q_2} (M_\infty^{q_2})^{1/q_2} \frac{\bigl( 1 + \frac{1}{M_\infty^{q_2}} \frac{\Xi_n}{\sqrt{n}} \bigr)^{1/q_2}}{1 - \frac{\Eta_n}{\sqrt{n}}}\\
&= n^{1/q_2} (M_\infty^{q_2})^{1/q_2} \biggl( 1 + \frac{1}{q_2 M_\infty^{q_2}} \, \frac{\Xi_n}{\sqrt{n}} + \frac{\Eta_n}{\sqrt{n}} + R_2\Bigl( \frac{\Xi_n}{\sqrt{n}}, \frac{\Eta_n}{\sqrt{n}} \Bigr) \biggr)
\end{align*}
and rearranging terms gives
\begin{equation*}
\sqrt{n} \biggl( \frac{n^{-1/q_2}}{(M_\infty^{q_2})^{1/q_2}} \lVert X_n \rVert_{q_2} - 1 \biggr) \GlVert \frac{\Xi_n}{q_2 M_\infty^{q_2}} + \Eta_n + \sqrt{n} R_2\Bigl( \frac{\Xi_n}{\sqrt{n}}, \frac{\Eta_n}{\sqrt{n}} \Bigr),
\end{equation*}
where we have employed the Taylor expansion
\begin{equation*}
\frac{\bigl( 1 + \frac{x}{M_\infty^{q_2}} \bigr)^{1/q_2}}{1 - y} = 1 + \frac{x}{q_2 M_\infty^{q_2}} + y + R_2(x, y),
\end{equation*}
with the remainder term satisfying \(\lvert R_2(x, y) \rvert \leq M_2 \lVert (x, y) \rVert_2^2\) in a suitable neighbourhood of \((0, 0)\). Notice \(n^{1/4} \bigl( \frac{\Xi_n}{\sqrt{n}}, \frac{\Eta_n}{\sqrt{n}} \bigr) = (n^{-1/4} \Xi_n, n^{-1/4} \Eta_n)\) for any \(n \in \NZ\); since \((\Xi_n)_{n \geq 1}\) converges in distribution, Slutsky's theorem implies \(\bigl( n^{1/4} \bigl( \frac{\Xi_n}{\sqrt{n}}, \frac{\Eta_n}{\sqrt{n}} \bigr) \bigr)_{n \geq 1} \KiWsk{} (0, 0)\), and with the remainder lemma we infer \(\bigl( \sqrt{n} R_2\bigl( \frac{\Xi_n}{\sqrt{n}}, \frac{\Eta_n}{\sqrt{n}} \bigr) \bigr)_{n \geq 1} \KiWsk{} 0\). Another application of Slutsky's theorem finally yields the desired convergence
\begin{equation*}
\sqrt{n} \biggl( \frac{n^{-1/q_2}}{(M_\infty^{q_2})^{1/q_2}} \lVert X_n \rVert_{q_2} - 1 \biggr) \KiVert{n \to \infty} \frac{\sigma N}{q_2 M_\infty^{q_2}}.
\end{equation*}

For \(p \neq 1\) notice
\begin{equation*}
Y_n = \sqrt{n} \biggl( \Bigl( 1 + \frac{Z_n}{\sqrt{n}} \Bigr)^p - 1 \biggr), \quad \text{where} \quad Z_n := \sqrt{n} \biggl( \frac{n^{1/q_1 - 1/q_2}}{(M_{q_1}^{q_2})^{1/q_2}} \, \lVert X^n \rVert_{q_2} - 1 \biggr).
\end{equation*}
Now by what we already have proved, \((Z_n)_{n \geq 1} \KiVert{} \sigma N\), and again via Slutsky this implies \(\bigl( n^{1/4} \frac{Z_n}{\sqrt{n}} \bigr)_{n \geq 1} = (n^{-1/4} Z_n)_{n \geq 1} \KiWsk{} 0\). But then Taylor expansion yields
\begin{align*}
Y_n = \sqrt{n} \biggl( 1 + p \, \frac{Z_n}{\sqrt{n}} + R_3\Bigl( \frac{Z_n}{\sqrt{n}} \Bigr) - 1 \biggr) = p Z_n + \sqrt{n} R_3\Bigl( \frac{Z_n}{\sqrt{n}} \Bigr);
\end{align*}
again the remainder term satisfies \(\lvert R_3(x) \rvert \leq M x^2\), and the remainder lemma and Slutsky's theorem lead to the desired conclusion.

The boundedness of moments in~\eqref{eq:beschr_mom} is subtler to prove. Let \(\alpha \geq 1\), and choose \(\beta\) as in Lemma~\ref{lem:mdp}, but with \(\beta \leq \frac{3}{4}\). We treat the case \(X_n \sim \Unif(\Kug{q_1}{n})\) only; the result for \(\KegM{q_1}{n - 1}\) follows by replacing \(U\) with 1 in what follows.

\textit{Case \(q_1 < \infty\):} Take \((\xi_i)_{i \geq 1} \sim \Gaussp_{q_1}^{\otimes \NZ}\) and \(U \sim \Unif([0, 1])\) independent, and define
\begin{equation*}
x_n := \frac{1}{M_{q_1}^{q_2} \sqrt{n}} \sum_{i = 1}^n \bigl( \lvert \xi_i \rvert^{q_2} - M_{q_1}^{q_2} \bigr) \quad \text{and} \quad y_n := \frac{1}{\sqrt{n}} \sum_{i = 1}^n (\lvert \xi_i \rvert^{q_1} - 1);
\end{equation*}
then
\begin{align*}
\lvert Y_n \rvert^\alpha &\GlVert n^{\alpha/2} \biggl\lvert \frac{n^{p (1/q_1 - 1/q_2)}}{(M_{q_1}^{q_2})^{p/q_2}} \, U^{p/n} \, \frac{\lVert (\xi_i)_{i \leq n} \rVert_{q_2}^p}{\lVert (\xi_i)_{i \leq n} \rVert_{q_1}^p} - 1 \biggr\rvert^\alpha\\
&= n^{\alpha/2} \biggl\lvert U^{p/n} \, \frac{\bigl( 1 + \frac{x_n}{\sqrt{n}} \bigr)^{p/q_2}}{\bigl( 1 + \frac{y_n}{\sqrt{n}} \bigr)^{p/q_1}} - 1 \biggr\rvert^\alpha.
\end{align*}
Define the event \(A_n := [\lvert x_n \rvert \leq n^{\beta - 1/2} \wedge \lvert y_n \rvert \leq n^{\beta - 1/2}]\), then
\begin{equation}\label{eq:aufteilung1}
\Erw[\lvert Y_n \rvert^{\alpha}] = \Erw[\lvert Y_n \rvert^\alpha \chF_{A_n}] + \Erw[\lvert Y_n \rvert^\alpha \chF_{\kompl{A_n}}],
\end{equation}
and we are going to show that either expectation on the right\-/hand side of~\eqref{eq:aufteilung1} is bounded for \(n \in \NZ\). For the first one, write
\begin{equation*}
\frac{(1 + x)^{p/q_2}}{(1 + y)^{p/q_2}} = 1 + R_4(x, y),
\end{equation*}
i.e., \(R_4\) is the zeroth remainder term of Taylor's expansion, which may be bounded as follows,
\begin{equation*}
\lvert R_4(x, y) \rvert \leq c_1 (\lvert x \rvert + \lvert y \rvert) \quad \text{for} \quad \lvert x \rvert, \lvert y \rvert \leq \tfrac{1}{2},
\end{equation*}
where \(c_1 \in (0, \infty)\). Making use of \(\lvert x + y \rvert^\alpha \leq c_2 (\lvert x \rvert^\alpha + \lvert y \rvert^\alpha)\) (to be precise, \(c_2 = \max\{1, 2^{\alpha - 1}\}\)), we get
\begin{align}
\Erw[\lvert Y_n \rvert^\alpha \chF_{A_n}] &= n^{\alpha/2} \Erw\biggl[ \biggl\lvert U^{p/n} \Bigl(1 + R_4\Bigl( \frac{x_n}{\sqrt{n}}, \frac{y_n}{\sqrt{n}} \Bigr) \Bigr) - 1 \biggr\rvert^{\alpha} \chF_{A_n} \biggr]\notag\\
&\leq c_2 n^{\alpha/2} \biggl( \Erw[(1 - U^{p/n})^\alpha] + \Erw[U^{p \alpha/n}] \Erw\Bigl[ \Bigl\lvert R_4\Bigl( \frac{x_n}{\sqrt{n}}, \frac{y_n}{\sqrt{n}} \Bigr) \Bigr\rvert^{\alpha} \chF_{A_n} \Bigr] \biggr)\notag\\
&\leq c_2 n^{\alpha/2} \biggl( \frac{\Gamma(\alpha + 1) \Gamma(\frac{n}{p} + 1)}{\Gamma(\alpha + \frac{n}{p} + 1)} + c_1^\alpha \, c_2 \Erw\biggl[ \frac{\lvert x_n \rvert^\alpha}{n^{\alpha/2}} + \frac{\lvert y_n \rvert^\alpha}{n^{\alpha/2}} \biggr] \biggr),\label{eq:abschaetz1}
\end{align}
where we have used independence of \(U\) and \(\{x_n, y_n\}\), and on \(A_n\) the estimates \(n^{-1/2} \lvert x_n \rvert,\\ n^{-1/2} \lvert y_n \rvert \leq n^{\beta - 1}\) hold true, therefore eventually they are smaller than \(\frac{1}{2}\) since \(\beta < 1\). Now it is well known that
\begin{equation}\label{eq:gamma_asympt}
\lim_{x \to \infty} \frac{\Gamma(x + y)}{x^y \, \Gamma(x)} = 1
\end{equation}
for any fixed \(y > 0\), so the first term within the parentheses in~\eqref{eq:abschaetz1} behaves like \(n^{-\alpha}\). For the second term notice that \((x_n)_{n \geq 1}\) and \((y_n)_{n \geq 1}\) satisfy the central limit theorem, and \(\Erw\bigl[ \bigl\lvert \lvert \xi_1 \rvert^{q_j} - M_{q_1}^{q_j} \bigr\rvert^\alpha \bigr] < \infty\) for \(j \in \{1, 2\}\), hence by \cite[Theorem~2]{Bahr1965} \(\Erw[\lvert x_n \rvert^\alpha]\) and \(\Erw[\lvert y_n \rvert^\alpha]\) converge to the corresponding (finite!) moments of the respective normal distributions. This concludes \(\limsup_{n \to \infty} \Erw[\lvert Y_n \rvert^\alpha \chF_{A_n}] < \infty\).

In order to tackle the second summand in~\eqref{eq:aufteilung1}, first apply H\"older's inequality to get
\begin{equation}\label{eq:aufteilung2}
\begin{split}
\Erw[\lvert Y_n \rvert^\alpha \chF_{\kompl{A_n}}] &\leq n^{\alpha/2} \Erw[\chF_{\kompl{A_n}}]^{1/3} \Erw\biggl[ \Bigl( 1 + \frac{y_n}{\sqrt{n}} \Bigr)^{-3 p \alpha/q_1} \biggr]^{1/3}\\
&\quad \cdot \Erw\biggl[ \Bigl\lvert U^{p/n} \Bigl( 1 + \frac{x_n}{\sqrt{n}} \Bigr)^{p/q_2} - \Bigl( 1 + \frac{y_n}{\sqrt{n}} \Bigr)^{p/q_1} \Bigr\rvert^{3 \alpha} \biggr]^{1/3}.
\end{split}
\end{equation}
With the union bound the first expectation on the right\-/hand side of~\eqref{eq:aufteilung2} is further estimated \(\Erw[\chF_{\kompl{A_n}}] \leq \Wsk[\lvert x_n \rvert \geq n^{\beta - 1/2}] + \Wsk[\lvert y_n \rvert \geq n^{\beta - 1/2}]\). Writing out we have
\begin{align*}
\Wsk[\lvert x_n \rvert \geq n^{\beta - 1/2}] &= \Wsk\biggl[ \frac{1}{n^\beta} \biggl\lvert \sum_{i = 1}^n \bigl( \lvert \xi_i \rvert^{q_2} - M_{q_1}^{q_2} \bigr) \biggr\rvert \geq M_{q_1}^{q_2} \biggr]\\
&= \exp\Bigl( -n^{2 \beta - 1} \, \frac{(M_{q_1}^{q_2})^2}{2 V_{q_1}^{q_2}} \, \bigl( 1 + \smallO(1) \bigr) \Bigr),
\end{align*}
where the asymptotics are argued by Lemma~\ref{lem:mdp}; an analogous result is obtained for \(y_n\), hence \(\lim_{n \to \infty} n^{\alpha/2} \Erw[\chF_{\kompl{A_n}}]^{1/3} = 0\).

The second expectation in~\eqref{eq:aufteilung2} can be computed explicitly, because \(1 + \frac{y_n}{\sqrt{n}} = \frac{1}{n} \sum_{i = 1}^n \lvert \xi_i \rvert^{q_1}\), and the latter follows a certain gamma\-/distribution, which yields
\begin{equation*}
\Erw\biggl[ \Bigl( 1 + \frac{y_n}{\sqrt{n}} \Bigr)^{-3 p \alpha/q_1} \biggr] = \Bigl(\frac{n}{q_1}\Bigr)^{3 p \alpha/q_1} \, \frac{\Gamma(\frac{n - 3 p \alpha}{q_1})}{\Gamma(\frac{n}{q_1})},
\end{equation*}
and that converges to~1 as \(n \to \infty\) by~\eqref{eq:gamma_asympt}. Finally the third expectation in~\eqref{eq:aufteilung2} is bounded from above, up to a constant factor depending only on \(\alpha\), by
\begin{equation*}
\Erw[U^{3 p \alpha/n}] \Erw\biggl[ \Bigl( 1 + \frac{x_n}{\sqrt{n}} \Bigr)^{3 p \alpha/q_2} \biggr] + \Erw\biggl[ \Bigl( 1 + \frac{y_n}{\sqrt{n}} \Bigr)^{3 p \alpha/q_1} \biggr].
\end{equation*}
The \(y_n\)\=/term we have dealt with before (just replace \(-\alpha\) by \(\alpha\)), and \(\Erw[U^{3 p \alpha/n}] \leq 1\). Similarly to \(y_n\) we have \(1 + \frac{x_n}{\sqrt{n}} = \frac{1}{n M_{q_1}^{q_2}} \sum_{i = 1}^n \lvert \xi_i \rvert^{q_2}\), whose law is not known explicitly though; nevertheless all moments of \(\lvert \xi_1 \rvert^{q_2}\) are finite, and \((1 + \frac{x_n}{\sqrt{n}})_{n \geq 1} \to 1\) almost surely by the SLLN, and by \cite[Theorem~10.2]{Gut2005} convergence is valid also in the \(L_p\)\=/sense, which in its turn implies
\begin{equation*}
\lim_{n \to \infty} \Erw\biggl[ \Bigl( 1 + \frac{x_n}{\sqrt{n}} \Bigr)^{3 p \alpha/q_2} \biggr] = 1^{3 p \alpha/q_2} = 1.
\end{equation*}
Taken together this amounts to \(\limsup_{n \to \infty} \Erw[\lvert Y_n \rvert^\alpha \chF_{\kompl{A_n}}] = 0\) and thus, returning to~\eqref{eq:aufteilung1},\linebreak\(\limsup_{n \to \infty} \Erw[\lvert Y_n \rvert^\alpha] < \infty\).

\textit{Case~\(q_1 = \infty\):} We are not going to spell out the details here, since the line of reasoning is analogous to the first case. Take \(U\) and \((\xi_n)_{n \geq 1}\) and define \(x_n\) as before, but set \(y_n := \Eta_n\) as in the proof of the CLT for \(\lVert X_n \rVert_{q_2}\) given above, so the representation reads
\begin{equation*}
\lvert Y_n \rvert^\alpha \GlVert n^{\alpha/2} \biggl\lvert U^{p/n} \, \frac{\bigl( 1 + \frac{x_n}{\sqrt{n}} \bigr)^{p/q_2}}{\bigl( 1 - \frac{y_n}{\sqrt{n}} \bigr)^p} - 1 \biggr\rvert^{\alpha}.
\end{equation*}
The remainder of this case's proof is conducted with the obvious adaptations; in particular notice \(y_n \GlVert \sqrt{n} (1 - U^{1/n})\) which may be used to calculate moments and \(\Wsk[\lvert y_n \rvert \geq n^{\beta - 1/2}]\).

Lastly, the convergence of \((\Erw[\lvert Y_n \rvert^\alpha])_{n \geq 1}\) now is almost immediate, as boundedness of \(\{\Erw[\lvert Y_n \rvert^{\alpha + 1}] \Colon n \in \NZ\}\) implies uniform integrability of \((\lvert Y_n \rvert^\alpha)_{n \geq 1}\), and together with \((\lvert Y_n \rvert^{\alpha})_{n \geq 1} \KiVert{} p^\alpha \sigma^\alpha \lvert N \rvert^\alpha\) this implies the claimed convergence; analogously for integer \(\alpha\) and \(\Erw[Y_n^\alpha]\). Statement~\eqref{eq:erw_theta} follows from the relation
\begin{equation*}
n^{p (1/q_1 - 1/q_2)} \lVert X_n \rVert_{q_2}^p = (M_{q_1}^{q_2})^{p/q_2} \Bigl( 1 + \frac{Y_n}{\sqrt{n}} \Bigr)
\end{equation*}
and the fact \((\Erw[Y_n])_{n \geq 1} \to p \sigma \Erw[N] = 0\).
\end{proof}

\section{Proofs of the Poincar\'e--Maxwell--Borel principles}
\label{sec:bew_pmb}

In this section we present the proofs of the Poincar\'e\--Maxwell\--Borel principles, that is, Theorem~\ref{sa:mpb_nkonst} and Theorem~\ref{sa:mpb_unendl}. We shall start with the probabilistic representation of Schechtman\--Zinn type, which facilitates computations.

\subsection{Proof of the probabilistic representation}
\label{subsec:bew_stochdarst}

In this subsection we present the proof of Proposition~\ref{sa:stoch_darst}, which provides us with a probabilistic representation of the uniform distribution on the unit balls in mixed\-/norm sequence spaces.

Let \(h \colon \RZ^{m \times n} \to [0, \infty)\) be an arbitrary measurable function, then
\begin{align*}
\Erw[h(X)] &= \frac{1}{\KugVol{p, q}{m, n}} \int_{\RZ^{m \times n}} h(x) \chF_{\Kug{p, q}{m, n}}(x) \, \dif x,\\
\intertext{or writing \(x\) in terms of its rows \(x_1, \dotsc, x_m\),}
\Erw[h(X)] &= \frac{1}{\KugVol{p, q}{m, n}} \int_{\RZ^n} \dotsi \int_{\RZ^n} h(x_1, \dotsc, x_m) \chF_{\Kug{p, q}{m, n}}(x_1, \dotsc, x_m) \, \dif x_m \dotsm \dif x_1.\\
\intertext{Introduce polar coordinates for each row separately, that is \(x_i = r_i \theta_i\) with \(r_i \in [0, \infty)\) and \(\theta_i \in \Sph{q}{n - 1}\)\---notice that this corresponds to our decomposition \((X_{i, j})_{j \leq n} = R_i \Theta_i\) introduced in~\eqref{eq:r_theta}\---to get}
\Erw[h(X)] &= \frac{n^m (\KugVol{q}{n})^m}{\KugVol{p, q}{m, n}} \int_{[0, \infty)} r_1^{n - 1} \int_{\Sph{q}{n - 1}} \dotsi \int_{[0, \infty)} r_m^{n - 1} \int_{\Sph{q}{n - 1}} h(r_1 \theta_1, \dotsc, r_m \theta_m)\\
&\hspace{15ex} \cdot \chF_{\Kug{p, q}{m, n}}(r_1 \theta_1, \dotsc, r_m \theta_m) \, \dif\KegM{q}{n - 1}(\theta_m) \, \dif r_m \dotsm \dif\KegM{q}{n - 1}(\theta_1) \, \dif r_1.\\
\intertext{Finally use \((r_i \theta_i)_{i \leq m} \in \Kug{p, q}{m, n}\) iff \((r_i)_{i \leq m} \in \Kug{p}{m}\), plug in \(\KugVol{p, q}{m, n} = 2^{-m} \KugVol{p/n}{m} (\KugVol{q}{n})^m\) (see Equation~\eqref{eq:kugvol}), and gather terms to arrive at}
\Erw[h(X)] &= \int_{[0, \infty)^m} \frac{(2 n)^m}{\KugVol{p/n}{m}} \prod_{i = 1}^m r_i^{n - 1} \chF_{\Kug{p}{m}}(r_1, \dotsc, r_m)\\
&\hspace{10ex} \cdot \int_{(\Sph{q}{n - 1})^m} h(r_1 \theta_1, \dotsc, r_m \theta_m) \, \dif(\KegM{q}{n - 1})^{\otimes m}(\theta_1, \dotsc, \theta_m) \, \dif(r_1, \dotsc, r_m).
\end{align*}
Now we recognize that the last integral proves the claimed density of \((R_i)_{i \leq m}\), the claimed independence and the claimed distribution of the \(\Theta_i\), \(i \in [1, m]\).

It remains to show the representation of \((R_i)_{i \leq m}\). To that end define \((S_i)_{i \leq m} := (R_i^n)_{i \leq m}\), hence \(R_i = S_i^{1/n}\) for each \(i \in [1, m]\); this yields the Jacobian \(n^{-m} \, \prod_{i = 1}^m s_i^{1/n - 1}\) and \((S_i)_{i \leq m}\) has density
\begin{align*}
f_{S_1, \dotsc, S_m}(s_1, \dotsc, s_m) &= \frac{(2 n)^m}{\KugVol{p/n}{m}} \prod_{i = 1}^m (s_i^{1/n})^{n - 1} \chF_{\Kug{p}{m} \cap [0, \infty)^m}(s_1^{1/n}, \dotsc, s_m^{1/n}) n^{-m} \prod_{i = 1}^m s_i^{1/n - 1}\\
&= \frac{2^m}{\KugVol{p/n}{m}} \chF_{\Kug{p/n}{m} \cap [0, \infty)^m}(s_1, \dotsc, s_m).
\end{align*}
Therefore \((S_i)_{i \leq m} \sim \Unif(\Kug{p/n}{m} \cap [0, \infty)^m)\), and by Schechtman and Zinn it can be written
\begin{equation*}
(S_i)_{i \leq m} \GlVert
\begin{cases}
U^{1/m} \bigl( \frac{\lvert \xi_i \rvert}{(\sum_{k = 1}^m \lvert \xi_k \rvert^{p/n})^{n/p}} \bigr)_{i \leq m} & \text{if } p < \infty,\\
(\lvert \xi_i \rvert)_{i \leq m} & \text{if } p = \infty.
\end{cases}
\end{equation*}
Transforming back to \((R_i)_{i \leq m}\) concludes the proof.

\subsection{Proofs of the Poincar\'e\--Maxwell\--Borel principles}%
\label{subsec:bew_pmb}

\begin{Lem}\label{lem:mpb_rtheta}
Let \(p, q \in (0, \infty]\). Then for any fixed \(k, l \in \NZ\) (\(k \leq m\), \(l \leq n\) where necessary):
\begin{compactenum}[(a)]
\item \((m^{1/p} \, R_i)_{i \leq k} \KiVert{m \to \infty} (\lvert \xi_i \rvert^{1/n})_{i \leq k}\) for fixed \(n \in \NZ\),
\item \((m^{1/p} \, R_i)_{i \leq k} \KiWsk{n \to \infty} (1)_{i \leq k}\) for either fixed \(m \in \NZ\) or \(m = m(n) \to \infty\), and
\item \((n^{1/q} \, \Theta_{i, j})_{i \leq k, j \leq l} \KiVert{n \to \infty} (\eta_{i, j})_{i \leq k, j \leq l}\).
\end{compactenum}
\end{Lem}

\begin{proof}
\begin{asparaenum}[(a)]
\item%
\textit{Case~\(p < \infty\):} We use Proposition~\ref{sa:stoch_darst}, (a),
\begin{equation}\label{eq:r_darst}
(m^{1/p} \, R_i)_{i \leq k} \GlVert \frac{U^{1/(m n)}}{\bigl( \frac{1}{m} \sum_{i = 1}^m \lvert \xi_i \rvert^{p/n} \bigr)^{1/p}} (\lvert \xi_i \rvert^{1/n})_{i \leq k}.
\end{equation}
By the SLLN, \(\frac{1}{m} \sum_{i = 1}^m \lvert \xi_i \rvert^{p/n} \Kfs{m \to \infty} M_{p/n}^{p/n} = 1\), hence the right-hand-side converges a.s.\ to \((\lvert \xi_i \rvert^{1/n})_{i \leq k}\), and convergence in distribution follows.

\Absatz%
\textit{Case~\(p = \infty\):} Obvious because of \((R_i)_{i \leq k} \GlVert (\lvert \xi_i \rvert^{1/n})_{i \leq k}\).

\item%
\textit{Case~\(p < \infty\):} By Lemma~\ref{lem:st_ggz} we know both \(\bigl( \frac{1}{m} \sum_{i = 1}^m \lvert \xi_i \rvert^{p/n} \bigr)_{n \geq 1} \KiWsk{} 1\) and \((\lvert \xi_i \rvert^{1/n})_{n \geq 1} \KiWsk{} 1\) for each \(i \in [1, k]\) (apply the lemma with \(q = 1\) and \(m = 1\)), hence the right-hand-side of~\eqref{eq:r_darst} converges to \(1\) in probability, therefore \((m^{1/p}\, R_i)_{i \leq k}\) converges in distribution towards a constant and thus also in probability.

\Absatz%
\textit{Case~\(p = \infty\):} Now \((R_i)_{i \leq k} \GlVert (\lvert \xi_i \rvert^{1/n})_{i \leq k} \KiWsk{n \to \infty} (1)_{i \leq k}\) via Lemma~\ref{lem:st_ggz}.

\item%
We have by Proposition~\ref{sa:stoch_darst}, (b),
\begin{equation}\label{eq:theta_darst}
(n^{1/q} \, \Theta_{i, j})_{i \leq k, j \leq l} \GlVert \biggl( \frac{\eta_{i, j}}{n^{-1/q} \lVert (\eta_{i, j'})_{j' \leq n} \rVert_q} \biggr)_{i \leq k, j \leq l}.
\end{equation}
By Lemma~\ref{lem:stggz_norm} \((n^{-1/q} \lVert (\eta_{i, j'})_{j' \leq n} \rVert_q)_{n \geq 1} \Kfs{} 1\), so the right-hand-side of~\eqref{eq:theta_darst} converges to \((\eta_{i, j})_{i \leq k, j \leq l}\) almost surely. \qedhere
\end{asparaenum}
\end{proof}

\begin{Bem}
Statements (a) and (c) of Lemma~\ref{lem:mpb_rtheta} can be seen as consequences of Proposition~\ref{prop:mpb}; this is immediate for (c), and for (a) recall from the proof of Proposition~\ref{sa:stoch_darst} that \((R_i^n)_{i \leq m} \sim \Unif(\Kug{p/n}{m} \cap [0, \infty)^m)\).
\end{Bem}

For a separable metric space \(E\) let \(\Masz_1(E)\) denote the convex set of probability measures on \((E, \Borel(E))\) endowed with the topology of weak convergence of measures; then \(\Masz_1(E)\) is separable too. This topology on \(\Masz_1(E)\) may be metrized by, e.g., the Lévy\--Prokhorov metric \(\dLP\). We denote by \(\LipB(E)\) the space of bounded, Lipschitz-continuous functions on \(E\), equipped with the norm \(\lVert f \rVert_{\Lip} := \max\{\lVert f \rVert_\infty, \lvert f \rvert_{\Lip}\}\) where \(\lvert f \rvert_{\Lip}\) is the Lipschitz-constant of \(f\).

\begin{Lem}\label{lem:kgz_wsk}
Let \((\mu_n)_{n \in \NZ}\) be a sequence of \(\Masz_1(E)\)\=/valued random measures and \(\mu \in \Masz_1(E)\) a deterministic measure. Then
\begin{equation*}
(\mu_n)_{n \geq 1} \KiWsk{} \mu \Longleftrightarrow \forall f \in \LipB(E) \colon \biggl( \int_E f \, \dif\mu_n \biggr)_{n \in \NZ} \KiWsk{} \int_E f \, \dif\mu.
\end{equation*}
\end{Lem}

\begin{proof}
\(\Rightarrow\): Let \(f \in \LipB(E)\), then the map \(\nu \mapsto \int_E f \, \dif\nu\) is continuous at \(\mu\) w.r.t.\ \(\dLP\), hence for any \(\varepsilon > 0\) there exists \(\delta > 0\) such that, for any \(\nu \in \Masz_1(E)\),
\begin{equation*}
\dLP(\mu, \nu) < \delta \Longrightarrow \biggl\lvert \int_E f \, \dif\nu - \int_E f \, \dif\mu \biggr\rvert < \varepsilon.
\end{equation*}
Now let \(\varepsilon > 0\), then
\begin{equation*}
\Wsk\biggl[ \biggl\lvert \int_E f \, \dif\mu_n - \int_E f \, \dif\mu \biggr\rvert \geq \varepsilon \biggr] \leq \Wsk[\dLP(\mu_n, \mu) \geq \delta] \Konv{n \to \infty} 0.
\end{equation*}

\(\Leftarrow\): Let \(\varepsilon > 0\). The open ball \(B_\varepsilon^{\text{\upshape LP}}(\mu)\) is open in the weak topology, thus there exist \(m \in \NZ\), \(f_1, \dotsc, f_m \in \LipB(E)\) and \(\delta > 0\) such that
\begin{equation*}
\bigcap_{i = 1}^m \biggl\{ \nu \in \Masz_1(E) \Colon \biggl\lvert \int_E f_i \, \dif\nu - \int_E f_i \, \dif\mu \biggr\rvert < \delta \biggr\} \subset B_\varepsilon^{\text{\upshape LP}}(\mu).
\end{equation*}
The union-bound then implies
\begin{align*}
\Wsk[\dLP(\mu_n, \mu) \geq \varepsilon] &= \Wsk[\mu_n \in \kompl{B_\varepsilon^{\text{\upshape LP}}(\mu)}]\\
&\leq \Wsk\biggl[ \mu_n \in \bigcup_{i = 1}^m \biggl\{ \nu \in \Masz_1(E) \Colon \biggl\lvert \int_E f_i \, \dif\nu - \int_E f_i \, \dif\mu \biggr\rvert \geq \delta \biggr\} \biggr]\\
&\leq \sum_{i = 1}^m \Wsk\biggl[ \biggl\lvert \int_E f_i \, \dif\mu_n - \int_E f_i \dif\mu \biggr\rvert \geq \delta \biggr] \Konv{n \to \infty} 0. \qedhere
\end{align*}
\end{proof}

\begin{proof}[Proof of Theorem~\ref{sa:mpb_nkonst}]
\begin{asparaenum}[(a)]
\item%
We have
\begin{equation*}
(m^{1/p} \, X_{i, j})_{i \leq k, j \leq n} = (m^{1/p} R_i \Theta_i)_{i \leq k}.
\end{equation*}
The claim follows from Lemma~\ref{lem:mpb_rtheta}, (a), together with the independence of \((R_i)_{i \leq m}\) from \(\Theta_1\),\ldots, \(\Theta_m\).

\item%
\textit{Case~\(p < \infty\):} The stochastic representation of \((R_i)_{i \leq m}\) implies
\begin{equation*}
\frac{1}{m} \sum_{i = 1}^m \delta_{m^{1/p} \, R_i} \GlVert \frac{1}{m} \sum_{i = 1}^m \delta_{U^{1/(m n)} \, (\frac{1}{m} \sum_{k = 1}^m \lvert \xi_k \rvert^{p/n})^{-1/p} \lvert \xi_i \rvert^{1/n}}.
\end{equation*}
For the sake of legibility call \(C_m := U^{1/(m n)} \bigl( \frac{1}{m} \sum_{i = 1}^m \lvert \xi_i \rvert^{p/n} \bigr)^{-1/p}\), then \((C_m)_{m \geq 1} \Kfs{} 1\). Now it suffices to prove
\begin{equation*}
\biggl( \frac{1}{m} \sum_{i = 1}^m \delta_{C_m \lvert \xi_i \rvert^{1/n}} \biggr)_{m \geq 1} \KiWsk{} \Vertl(\lvert \xi_1 \rvert^{1/n}),
\end{equation*}
since then \(\bigl( \frac{1}{m} \sum_{i = 1}^m \delta_{m^{1/p} \, R_i} \bigr)_{m \geq 1} \to \Vertl(\lvert \xi_1 \rvert^{1/n})\) in distribution and, because the latter is constant in \(\Masz_1(\RZ)\), also in probability.

\Absatz%
We apply Lemma~\ref{lem:kgz_wsk}. Let \(f \in \LipB(\RZ)\), then
\begin{multline*}
\biggl\lvert \frac{1}{m} \sum_{i = 1}^m f(C_m \lvert \xi_i \rvert^{1/n}) - \Erw[f(\lvert \xi_1 \rvert^{1/n})] \biggr\rvert\\
\begin{aligned}
&\leq \biggl\lvert \frac{1}{m} \sum_{i = 1}^m f(C_m \lvert \xi_i \rvert^{1/n}) - \frac{1}{m} \sum_{i = 1}^m f(\lvert \xi_i \rvert^{1/n}) \biggr\rvert + \biggl\lvert \frac{1}{m} \sum_{i = 1}^m f(\lvert \xi_i \rvert^{1/n}) - \Erw[f(\lvert \xi_1 \rvert^{1/n})] \biggr\rvert\\
&\leq \frac{1}{m} \sum_{i = 1}^m \bigl\lvert f(C_m \lvert \xi_i \rvert^{1/n}) - f(\lvert \xi_i \rvert^{1/n}) \bigr\rvert + \biggl\lvert \frac{1}{m} \sum_{i = 1}^m f(\lvert \xi_i \rvert^{1/n}) - \Erw[f(\lvert \xi_1 \rvert^{1/n})] \biggr\rvert\\
&\leq \lvert f \rvert_{\Lip} \lvert C_m - 1 \rvert \, \frac{1}{m} \sum_{i = 1}^m \lvert \xi_i \rvert^{1/n} + \biggl\lvert \frac{1}{m} \sum_{i = 1}^m f(\lvert \xi_i \rvert^{1/n}) - \Erw[f(\lvert \xi_1 \rvert^{1/n})] \biggr\rvert;
\end{aligned}
\end{multline*}
the last line converges a.s.\ to zero because the sums obey the SLLN, and thus also in probability.

\Absatz%
Essentially the same argument is valid for \(\frac{1}{m} \sum_{i = 1}^m \delta_{m^{1/p} \, X_i}\). Write \(m^{1/p} \, X_i \GlVert C_m \lvert \xi_i \rvert^{1/n} \Theta_i\) with \(C_m\) as before, and let \(f \in \LipB(\RZ^n)\), where the Lipschitz constant is taken with respect to \(\lVert \cdot \rVert_q\), then
\begin{multline*}
\biggl\lvert \frac{1}{m} \sum_{i = 1}^m f(C_m \lvert \xi_i \rvert^{1/n} \Theta_i) - \Erw[f(\lvert \xi_1 \rvert^{1/n} \Theta_1)] \biggr\rvert\\
\begin{aligned}
&\leq \biggl\lvert \frac{1}{m} \sum_{i = 1}^m f(C_m \lvert \xi_i \rvert^{1/n} \Theta_i) - \frac{1}{m} \sum_{i = 1}^m f(\lvert \xi_i \rvert^{1/n} \Theta_i) \biggr\rvert\\
&\quad + \biggl\lvert \frac{1}{m} \sum_{i = 1}^m f(\lvert \xi_i \rvert^{1/n} \Theta_i) - \Erw[f(\lvert \xi_1 \rvert^{1/n} \Theta_1)] \biggr\rvert\\
&\leq \frac{1}{m} \sum_{i = 1}^m \bigl\lvert f(C_m \lvert \xi_i \rvert^{1/n}  \Theta_i) - f(\lvert \xi_i \rvert^{1/n} \Theta_i) \bigr\rvert + \biggl\lvert \frac{1}{m} \sum_{i = 1}^m f(\lvert \xi_i \rvert^{1/n} \Theta_i) - \Erw[f(\lvert \xi_1 \rvert^{1/n} \Theta_1)] \biggr\rvert\\
&\leq \lvert f \rvert_{\Lip} \lvert C_m - 1 \rvert \, \frac{1}{m} \sum_{i = 1}^m \lvert \xi_i \rvert^{1/n} + \biggl\lvert \frac{1}{m} \sum_{i = 1}^m f(\lvert \xi_i \rvert^{1/n} \Theta_i) - \Erw[f(\lvert \xi_1 \rvert^{1/n} \Theta_1)] \biggr\rvert;
\end{aligned}
\end{multline*}
again the sums obey the SLLN and hence the desired convergence is implied.

\Absatz%
\textit{Case~\(p = \infty\):} Notice that by the stochastic representation we are dealing with independent random variables and thus the convergence is immediate. \qedhere
\end{asparaenum}
\end{proof}

\begin{proof}[Proof of Theorem~\ref{sa:mpb_unendl}]
\begin{asparaenum}[(a)]
\item%
Recall
\begin{equation*}
(m^{1/p} \, n^{1/q} \, X_{i, j})_{i \leq m, j \leq n} = (m^{1/p} R_i \cdot n^{1/q} \Theta_{i, j})_{i \leq m, j \leq n}.
\end{equation*}
Lemma~\ref{lem:mpb_rtheta}, (b) and (c), imply the convergence in distribution of \((m^{1/p} \, n^{1/q} \, X_{i, j})_{i \leq k, j \leq l}\) as claimed, where the joint convergence of \((R_i)_{i \leq m}\), \(\Theta_1\),\ldots, \(\Theta_m\) may be argued either by their independence or by Slutsky's theorem.

\item%
\textit{Case~\(p < \infty\):} Write \(C_{m, n} := U^{1/(m n)} \bigl( \frac{1}{m} \sum_{i = 1}^m \lvert \xi_i \rvert^{p/n} \bigr)^{-1/p}\) and \(D_{i, n} := \lvert \xi_i \rvert^{1/n} (n^{-1/q} \lVert (\eta_{i, j})_{j \leq n} \rVert_q)^{-1}\), then
\begin{equation*}
(C_{m, n})_{n \geq 1} \KiWsk{} 1 \quad \text{and} \quad (D_{i, n})_{n \geq 1} \KiWsk{} 1.
\end{equation*}
Now take any \(f \in \LipB(\RZ)\) and consider
\begin{equation*}
\biggl\lvert \frac{1}{m n} \sum_{i = 1}^m \sum_{j = 1}^n f(C_{m, n} D_{i, n} \eta_{i, j}) - \Erw[f(\eta_{1, 1})] \biggr\rvert;
\end{equation*}
we have to show that the probability of this expression being smaller than any positive number converges to one. So let \(\varepsilon > 0\). Define \(B_{m, n, \varepsilon} := \sum_{i = 1}^m \chF_{[\lvert D_{i, n} - 1 \rvert \geq \varepsilon]}\), then \(B_{m, n, \varepsilon}\) is binomially distributed with parameters \(m\) and \(\Wsk[\lvert D_{1, n} - 1 \rvert \geq \varepsilon]\), and there holds \((\frac{1}{m} \, B_{m, n, \varepsilon})_{n \geq 1} \KiWsk{} 0\): indeed, let \(\delta > 0\), then
\begin{align*}
\Wsk\Bigl[ \Bigl\lvert \frac{1}{m} B_{m, n, \varepsilon} \Bigr\rvert \geq \delta \Bigr] &\leq \frac{1}{m^2 \, \delta^2} \Var[B_{m, n, \varepsilon}]\\
&= \frac{1}{m \, \delta^2} \Wsk[\lvert D_{1, n} - 1 \rvert \geq \varepsilon] \Wsk[\lvert D_{1, n} - 1 \lvert < \varepsilon],
\end{align*}
and because of \((D_{1, n})_{n \geq 1} \KiWsk{} 1\) the latter converges to zero as \(n \to \infty\), irrespective of whether \(m\) is fixed or diverges. We also have the laws of large numbers \(\bigl( \frac{1}{m n} \sum_{i = 1}^m \sum_{j = 1}^n f(\eta_{i, j}) \bigr)_{n \geq 1} \KiWsk{} \Erw[f(\eta_{1, 1})]\) and \(\bigl( \frac{1}{m n} \sum_{i = 1}^m \sum_{j = 1}^n \lvert \eta_{i, j} \rvert \bigr)_{n \geq 1} \KiWsk{} M_q^1\). Now there exists an \(n_0 \in \NZ\) such that, for any \(n \geq n_0\), the probability of the event that \(\lvert C_{m, n} - 1 \rvert \leq \varepsilon\) and \(2 \lVert f \rVert_\infty \frac{1}{m} B_{m, n, \varepsilon} \leq \varepsilon\) and \(\bigl\lvert \frac{1}{m n} \sum_{i = 1}^m \sum_{j = 1}^n f(\eta_{i, j}) - \Erw[f(\eta_{1, 1})] \bigr\rvert \leq \varepsilon\) and \(\bigl\lvert \frac{1}{m n} \sum_{i = 1}^m \sum_{j = 1}^n \lvert \eta_{i, j} \rvert - M_q^1 \bigr\rvert \leq \varepsilon\) all hold true is at least, say, \(1 - \varepsilon\). Let \(n \geq n_0\), then on the same event, for any \(i \in [1, m]\) such that \(\lvert D_{i, n} - 1 \rvert < \varepsilon\), we have
\begin{equation*}
\lvert C_{m, n} D_{i, n} - 1 \rvert \leq \lvert C_{m, n} \rvert \lvert D_{i, n} - 1 \rvert + \lvert C_{m, n} - 1 \rvert \leq (1 + \varepsilon) \varepsilon + \varepsilon = \varepsilon^2 + 2 \varepsilon,
\end{equation*}
and therewith
\begin{multline*}
\biggl\lvert \frac{1}{m n} \sum_{i = 1}^m \sum_{j = 1}^n f(C_{m, n} D_{i, n} \eta_{i, j}) - \Erw[f(\eta_{1, 1})] \biggr\rvert\\
\begin{aligned}
&\leq \biggl\lvert \frac{1}{m n} \sum_{i = 1}^m \sum_{j = 1}^n f(C_{m, n} D_{i, n} \eta_{i, j}) - \frac{1}{m n} \sum_{i = 1}^m \sum_{j = 1}^n f(\eta_{i, j}) \biggr\rvert\\
&\quad + \biggl\lvert \frac{1}{m n} \sum_{i = 1}^m \sum_{j = 1}^n f(\eta_{i, j}) - \Erw[f(\eta_{1, 1})] \biggr\rvert\\
&\leq \frac{1}{m n} \sum_{i = 1}^m \sum_{j = 1}^n \bigl\lvert f(C_{m, n} D_{i, n} \eta_{i, j}) - f(\eta_{i, j}) \bigr\rvert + \varepsilon\\
&= \varepsilon + \frac{1}{m n} \sum_{i = 1}^m \sum_{j = 1}^n \chF_{[\lvert D_{i, n} - 1 \rvert < \varepsilon]} \bigl\lvert f(C_{m, n} D_{i, n} \eta_{i, j}) - f(\eta_{i, j}) \bigr\rvert\\
&\quad + \frac{1}{m n} \sum_{i = 1}^m \sum_{j = 1}^n \chF_{[\lvert D_{i, n} - 1 \rvert \geq \varepsilon]} \bigl\lvert f(C_{m, n} D_{i, n} \eta_{i, j}) - f(\eta_{i, j}) \bigr\rvert\\
&\leq \varepsilon + \frac{\lvert f \rvert_{\Lip}}{m n} \sum_{i = 1}^m \sum_{j = 1}^n \chF_{[\lvert D_{i, n} - 1 \rvert < \varepsilon]} \lvert C_{m, n} D_{i, n} - 1 \rvert \lvert \eta_{i, j} \rvert\\
&\quad + \frac{1}{m n} \sum_{i = 1}^m \sum_{j = 1}^n \chF_{[\lvert D_{i, n} - 1 \rvert \geq \varepsilon]} \bigl( \lvert f(C_{m, n} D_{i, n} \eta_{i, j}) \rvert + \lvert f(\eta_{i, j}) \rvert \bigr)\\
&\leq \varepsilon + \frac{\lvert f \rvert_{\Lip}}{m n} (\varepsilon^2 + 2 \varepsilon) \sum_{i = 1}^m \sum_{j = 1}^n \lvert \eta_{i, j} \rvert + \frac{2 \lVert f \rVert_\infty B_{m, n, \varepsilon}}{m}\\
&\leq \varepsilon + \lvert f \rvert_{\Lip} (\varepsilon^2 + 2 \varepsilon) (M_q^1 + \varepsilon) + \varepsilon.
\end{aligned}
\end{multline*}
Because this estimate holds for all \(n \geq n_0\) with probability at least \(1 - \varepsilon\), convergence in probability is established.

\Absatz%
\textit{Case~\(p = \infty\):} Again let \(f \in \LipB(\RZ)\), then with the same notation and techniques as before,
\begin{align*}
\biggl\lvert \frac{1}{m n} \sum_{i = 1}^m \sum_{j = 1}^n f(D_{i, n} \eta_{i, j}) - \Erw[f(\eta_{1, 1})] \biggr\rvert &\leq \biggl\lvert \frac{1}{m n} \sum_{i = 1}^m \sum_{j = 1}^n f(D_{i, n} \eta_{i, j}) - \frac{1}{m n} \sum_{i = 1}^m \sum_{j = 1}^n f(\eta_{i, j}) \biggr\rvert\\
&\quad + \biggl\lvert \frac{1}{m n} \sum_{i = 1}^m \sum_{j = 1}^n f(\eta_{i, j}) - \Erw[f(\eta_{1, 1})] \biggr\rvert\\
&\leq \frac{1}{m n} \sum_{i = 1}^m \sum_{j = 1}^n \bigl\lvert f(D_{i, n} \eta_{i, j}) - f(\eta_{i, j}) \bigr\rvert + o(1).
\end{align*}
The remaining argument is the same as in the case \(p < \infty\), only formally \(C_{m, n} = 1\) throughout. \qedhere
\end{asparaenum}
\end{proof}

\section{Proofs of the weak limit theorems}
\label{sec:bew_gws}

In this section we present the proofs of the weak limit theorems, that is, Theorem~\ref{sa:zgs1}, Theorem~\ref{sa:zgs2}, and Theorem~\ref{sa:zgs3}, as well as of Corollaries~\ref{sa:kritfall1}, \ref{sa:kritfall2}, and~\ref{sa:kritfall3}.

\subsection{Proofs of the weak limit theorems}
\label{subsec:bew_gws}

Recall that Theorem~\ref{sa:zgs1} treats the regime \(m \to \infty\) while \(n\) is fixed.

\begin{proof}[Proof of Theorem~\ref{sa:zgs1}]
\textit{Case~\(p_1 < \infty\):} Appealing to Proposition~\ref{sa:stoch_darst} we have
\begin{equation*}
\lVert X^m \rVert_{p_2, q_2} \GlVert U^{1/(m n)} \, \frac{\bigl( \sum_{i = 1}^m \lvert \xi_i \rvert^{p_2/n} \lVert \Theta_i \rVert_{q_2}^{p_2} \bigr)^{1/p_2}}{\bigl( \sum_{i = 1}^m \lvert \xi_i \rvert^{p_1/n} \bigr)^{1/p_1}}.
\end{equation*}
Define
\begin{align*}
\Xi_m &:= \frac{1}{\sqrt{m}} \sum_{i = 1}^m \bigl( \lvert \xi_i \rvert^{p_1/n} - 1 \bigr)\\
\intertext{and}
\Eta_m &:= \frac{1}{\sqrt{m}} \sum_{i = 1}^m \bigl( \lvert \xi_i \rvert^{p_2/n} \lVert \Theta_i \rVert_{q_2}^{p_2} - M_{p_1/n}^{p_2/n} \Erw[\lVert \Theta_1 \rVert_{q_2}^{p_2}] \bigr),
\end{align*}
then by the multivariate CLT we know
\begin{equation*}
\left( \begin{pmatrix} \Xi_m \\ \Eta_m \end{pmatrix} \right)_{m \geq 1} \KiVert{} \Nvert_2(\mathbf{0}, \Sigma)
\end{equation*}
with covariance\-/matrix
\begin{equation*}
\Sigma := \begin{pmatrix} \frac{p_1}{n} & C_{p_1/n}^{p_1/n, p_2/n} \Erw[\lVert \Theta_1 \rVert_{q_2}^{p_2}] \\ C_{p_1/n}^{p_1/n, p_2/n} \Erw[\lVert \Theta_1 \rVert_{q_2}^{p_2}] & \Var[\lvert \xi_1 \rvert^{p_2/n} \lVert \Theta_1 \rVert_{q_2}^{p_2}] \end{pmatrix}.
\end{equation*}
For brevity's sake we set \(\mu := M_{p_1/n}^{p_2/n} \Erw[\lVert \Theta_1 \rVert_{q_2}^{p_2}]\). Therewith we get
\begin{align*}
\lVert X^m \rVert_{p_2, q_2} &\GlVert U^{1/(m n)} \frac{(m \mu + \sqrt{m} \, \Eta_m)^{1/p_2}}{(m + \sqrt{m} \, \Xi_m)^{1/p_1}}\\
&= \frac{m^{1/p_2} \, \mu^{1/p_2}}{m^{1/p_1}} \, U^{1/(m n)} \frac{\bigl( 1 + \frac{\Eta_m}{\mu \sqrt{m}} \bigr)^{1/p_2}}{\bigl( 1 + \frac{\Xi_m}{\sqrt{m}} \bigr)^{1/p_1}}\\
&= \frac{\mu^{1/p_2}}{m^{1/p_1 - 1/p_2}} \biggl( 1 + \frac{\log(U)}{m n} - \frac{\Xi_m}{p_1 \sqrt{m}} + \frac{\Eta_m}{p_2 \mu \sqrt{m}} + R\Bigl( \frac{\log(U)}{m}, \frac{\Xi_m}{\sqrt{m}}, \frac{\Eta_m}{\sqrt{m}} \Bigr) \biggr),
\end{align*}
where for the third equality we have performed the Taylor\-/expansion
\begin{equation*}
\ez^{u/n} \frac{\bigl( 1 + \frac{y}{\mu} \bigr)^{1/p_2}}{\bigl( 1 + x \bigr)^{1/p_1}} = 1 + \frac{u}{n} - \frac{x}{p_1} + \frac{y}{p_2 \mu} + R(u, x, y),
\end{equation*}
where the remainder satisfies \(\lvert R(u, x, y) \rvert \leq M \lVert (u, x, y) \rVert_2^2\) in a suitable neighbourhood of \((0, 0, 0)\). Rearranging yields
\begin{equation*}
\sqrt{m} \biggl( \frac{m^{1/p_1 - 1/p_2}}{\mu^{1/p_2}} \, \lVert X^m \rVert_{p_2, q_2} - 1 \biggr) \GlVert \frac{\log(U)}{n \sqrt{m}} -\frac{\Xi_m}{p_1} + \frac{\Eta_m}{p_2 \mu} + \sqrt{m} R\Bigl( \frac{\log(U)}{m}, \frac{\Xi_m}{\sqrt{m}}, \frac{\Eta_m}{\sqrt{m}} \Bigr).
\end{equation*}
We have \(m^{1/4} \bigl( \frac{\log(U)}{m}, \frac{\Xi_m}{\sqrt{m}}, \frac{\Eta_m}{\sqrt{m}} \bigr) = (m^{-3/4} \log(U), m^{-1/4} \Xi_m, m^{-1/4} \Eta_m)\), and this converges in probability to \((0, 0, 0)\) as \(m \to \infty\) by appealing to Slutsky's theorem and the known distributional convergence of \((\Xi_m, \Eta_m)\). The remainder lemma then implies \(\bigl( \sqrt{m} \, R\bigl( \frac{\log(U)}{m}, \frac{\Xi_m}{\sqrt{m}}, \frac{\Eta_m}{\sqrt{m}} \bigr) \bigr)_{m \geq 1} \KiWsk{} 0\). Since we also know \((m^{-1/2} \log(U))_{m \geq 1} \Konv{} 0\) almost surely and thus in probability, by Slutsky's theorem the right\-/hand\-/side of the last display converges to the random variable \(\sigma N\), where \(N \sim \Nvert(0, 1)\) and
\begin{equation*}
\sigma^2 = \Bigl( -\frac{1}{p_1}, \frac{1}{p_2 \mu} \Bigr) \Sigma \begin{pmatrix} -\frac{1}{p_1} \\ \frac{1}{p_2 \mu} \end{pmatrix};
\end{equation*}
a simple calculation shows that this is the desired variance.

\textit{Case~\(p_1 = \infty\):} Omit \(U^{1/(m n)}\) and \(\sum_{i = 1}^m \lvert \xi_1 \rvert^{p_1/n}\) from the probabilistic representation and reiterate the argument.
\end{proof}

The regime for Theorem~\ref{sa:zgs2} is \(n \to \infty\) while \(m\) is fixed.

\begin{proof}[Proof of Theorem~\ref{sa:zgs2}]
\begin{asparaenum}[(a)]
\item
\textit{Case~\(p_1 < \infty\):} We define, for \(i \in [1, m]\),
\begin{align*}
\Xi_{n, i} &:= \sqrt{n} (\lvert \xi_i \rvert^{p_1/n} - 1),\\
\Eta_{n, i} &:= \sqrt{n} \biggl( \frac{n^{p_2 (1/q_1 - 1/q_2)}}{(M_{q_1}^{q_2})^{p_2/q_2}} \lVert \Theta_i \rVert_{q_2}^{p_2} - 1 \biggr);
\end{align*}
then, since \(\xi_1, \dotsc, \xi_m, \Theta_1, \dotsc, \Theta_m\) are independent for each \(n \in \NZ\), so are \(\Xi_{n, 1}, \dotsc, \Xi_{n, m}, \Eta_{n, 1}, \dotsc, \Eta_{n, m}\), and by Lemma~\ref{lem:eigenschaften_xi}, 2.(a), \(((\Xi_{n, i})_{i \leq m})_{n \geq 1} \KiVert{} (\sqrt{p_1} N_{1, i})_{i \leq m}\) and by Lemma~\ref{lem:zgs_lp} \(((\Eta_{n, i})_{i \leq m})_{n \geq 1} \KiVert{} (p_2 \sigma N_{2, i})_{i \leq m}\) with \((N_{1, i})_{i \leq m}, (N_{2, i})_{i \leq m} \sim \Nvert(\mathbf{0}, I_m)\) independent. This leads to
\begin{align*}
\lVert X^n \rVert_{p_2, q_2} &\GlVert U^{1/(m n)} \frac{\bigl( \sum_{i = 1}^m \lvert \xi_i \rvert^{p_2/n} \lVert \Theta_i \rVert_{q_2}^{p_2} \bigr)^{1/p_2}}{\bigl( \sum_{i = 1}^m \lvert \xi_i \rvert^{p_1/n} \bigr)^{1/p_1}}\\
&= \frac{(M_{q_1}^{q_2})^{1/q_2}}{m^{1/p_1 - 1/p_2} \, n^{1/q_1 - 1/q_2}} \, U^{1/(m n)} \frac{\bigl( \frac{1}{m} \sum_{i = 1}^m \bigl( 1 + \frac{\Xi_{n, i}}{\sqrt{n}} \bigr)^{p_2/p_1} \bigl( 1 + \frac{\Eta_{n, i}}{\sqrt{n}} \bigr) \bigr)^{1/p_2}}{\bigl( \frac{1}{m} \sum_{i = 1}^m \bigl( 1 + \frac{\Xi_{n, i}}{\sqrt{n}} \bigr) \bigr)^{1/p_1}}\\
&= \frac{(M_{q_1}^{q_2})^{1/q_2}}{m^{1/p_1 - 1/p_2} \, n^{1/q_1 - 1/q_2}} \biggl( 1 + \frac{\log(U)}{m n} + \frac{1}{p_2 m} \sum_{i = 1}^m \frac{\Eta_{n, i}}{\sqrt{n}}\\
&\mspace{220mu} + R\biggl( \frac{\log(U)}{n}, \Bigl( \frac{\Xi_{n, i}}{\sqrt{n}} \Bigr)_{i \leq m}, \Bigl( \frac{\Eta_{n, i}}{\sqrt{n}} \Bigr)_{i \leq m} \biggr) \biggr),
\end{align*}
where we have introduced the Taylor polynomial expansion
\begin{equation*}
\ez^{u/m} \, \frac{\bigl( \frac{1}{m} \sum_{i = 1}^m (1 + x_i)^{p_2/p_1} (1 + y_i) \bigr)^{1/p_2}}{\bigl( \frac{1}{m} \sum_{i = 1}^m (1 + x_i) \bigr)^{1/p_1}} = 1 + \frac{u}{m} + \frac{1}{p_2 m} \sum_{i = 1}^m y_i + R\bigl( u, (x_i)_{i \leq m}, (y_i)_{i \leq m} \bigr)
\end{equation*}
(the partial derivatives of first order w.r.t.\ \(x_1, \dotsc, x_m\) are indeed zero), where again the remainder term satsifies \(\lvert R\bigl( u, (x_i)_{i \leq m}, (y_i)_{i \leq m} \bigr) \rvert \leq M \bigl\lVert \bigl( u, (x_i)_{i \leq m}, (y_i)_{i \leq m} \bigr) \bigr\rVert_2^2\) in a neighbourhood of \(\mathbf{0}\).
Rearranging yields
\begin{align*}
\sqrt{n} \biggl( \frac{m^{1/p_1 - 1/p_2} n^{1/q_1 - 1/q_2}}{(M_{q_1}^{q_2})^{1/q_2}} \lVert X^n \rVert_{p_2, q_2} - 1 \biggr) &\GlVert \frac{\log(U)}{m \sqrt{n}} + \frac{1}{p_2 m} \sum_{i = 1}^m \Eta_{n, i}\\
&\quad + \sqrt{n} R\biggl( \frac{\log(U)}{n}, \Bigl( \frac{\Xi_{n, 1}}{\sqrt{n}} \Bigr)_{i \leq m}, \Bigl( \frac{\Eta_{n, i}}{\sqrt{n}} \Bigr)_{i \leq m} \biggr),
\end{align*}
and we apply the usual argument: \(((\Xi_{n, i})_{i \leq m})_{n \geq 1}\) and \(((\Eta_{n, i})_{i \leq m})_{n \geq 1}\) converge in distribution, hence from Slutsky's theorem we infer \(n^{1/4} \bigl( \frac{\log(U)}{n}, \bigl( \frac{\Xi_{n, 1}}{\sqrt{n}} \bigr)_{i \leq m}, \bigl( \frac{\Eta_{n, i}}{\sqrt{n}} \bigr)_{i \leq m} \bigr) = \bigl( n^{-3/4} \log(U),\\ (n^{-1/4} \Xi_{n, i})_{i \leq m}, (n^{-1/4} \Eta_{n, i})_{i \leq m} \bigr) \Konv{n \to \infty} \mathbf{0}\) in distribution and in probability; the remainder lemma then gives \(\bigl( \sqrt{n} R\bigl( \frac{\log(U)}{n}, \bigl( \frac{\Xi_{n, 1}}{\sqrt{n}} \bigr)_{i \leq m}, \bigl( \frac{\Eta_{n, i}}{\sqrt{n}} \bigr)_{i \leq m} \bigr) \bigr)_{n \geq 1} \KiWsk{} 0\); we also have \(\bigl( \frac{\log(U)}{m \sqrt{n}} \bigr)_{n \geq 1} \to 0\) almost surely and in probability; and a final use of Slutsky's theorem leads to the desired result.

\Absatz%
\textit{Case~\(p_1 = \infty\):} Now according to Lemma~\ref{lem:eigenschaften_xi}, 2.(b),
\begin{equation*}
\Xi_{n, i} := n (1 - \lvert \xi_i \rvert^{p_2/n}) \KiVert{n \to \infty} p_2 E_i,
\end{equation*}
where \((E_i)_{i \leq m} \sim \Exp(1)^{\otimes m}\) is independent of \((N_{2, i})_{i \leq m}\) introduced before; and therewith
\begin{align*}
\lVert X^n \rVert_{p_2, q_2} &\GlVert \frac{(M_{q_1}^{q_2})^{1/q_2}}{m^{-1/p_2} \, n^{1/q_1 - 1/q_2}} \biggl( \frac{1}{m} \sum_{i = 1}^m \Bigl( 1 - \frac{\Xi_{n, i}}{n} \Bigr) \Bigl( 1 + \frac{\Eta_{n, i}}{\sqrt{n}} \Bigr) \biggr)^{1/p_2}\\
&= \frac{(M_{q_1}^{q_2})^{1/q_2}}{m^{-1/p_2} \, n^{1/q_1 - 1/q_2}} \biggl( 1 - \sum_{i = 1}^m \frac{\Xi_{n, i}}{p_2 m n} + \sum_{i = 1}^m \frac{\Eta_{n, i}}{p_2 m \sqrt{n}}\\
&\mspace{190mu}+ R\biggl( \Bigl( \frac{\Xi_{n, 1}}{n} \Bigr)_{i \leq m}, \Bigl( \frac{\Eta_{n, i}}{\sqrt{n}} \Bigr)_{i \leq m} \biggr) \biggr),
\end{align*}
and the rest follows as before, with the modification \(\sqrt{n} \bigl( \frac{\Xi_{n, i}}{n} \bigr)_{i \leq m} = (n^{-1/2} \, \Xi_{n, i})_{i \leq m} \KiWsk{n \to \infty} \mathbf{0}\) and similarly for the remainder term.

\item
Here, as in the following case, \(\lVert \Theta_i \rVert_{q_2} = \lVert \Theta_i \rVert_{q_1} = 1\) and therefore we have
\begin{equation*}
\lVert X^n \rVert_{p_2, q_1} \GlVert U^{1/(m n)} \, \frac{\bigl( \sum_{i = 1}^m \lvert \xi_i \rvert^{p_2/n} \bigr)^{1/p_2}}{\bigl( \sum_{i = 1}^m \lvert \xi_i \rvert^{p_1/n} \bigr)^{1/p_1}}.
\end{equation*}
We perform Taylor expansion of the same function as in (a), case \(p_1 < \infty\), but restricted to \((y_i)_{i \leq m} = \mathbf{0}\) and writing out second\-/order terms, to wit,
\begin{equation*}
\ez^{u/m} \, \frac{\bigl( \frac{1}{m} \sum_{i = 1}^m (1 + x_i)^{p_2/p_1} \bigr)^{1/p_2}}{\bigl( \frac{1}{m} \sum_{i = 1}^m (1 + x_i) \bigr)^{1/p_1}} = 1 + \frac{u}{m} + \frac{u^2}{2 m^2} + \frac{p_2 - p_1}{2 p_1^2 m^2} \, \trapo{x} A x + R(u, x),
\end{equation*}
where \(A = (a_{i, j})_{i, j \leq m} \in \RZ^{m \times m}\) is given by \(a_{i, i} = m - 1\) and \(a_{i, j} = -1\) for all \(i, j \in [1, m]\) with \(i \neq j\), and the remainder term satisfies \(\lvert R(u, x) \rvert \leq M \lVert (u, x) \rVert_2^3\) with some \(M > 0\) for all \(\lVert (u, x) \rVert_2\) sufficiently small. So this gives
\begin{align*}
\lVert X^n \rVert_{p_2, q_1} &\GlVert m^{1/p_2 - 1/p_1} \biggl( 1 + \frac{\log(U)}{m n} + \frac{\log(U)^2}{2 m^2 n^2} + \frac{p_2 - p_1}{2 p_1^2 m^2} \trapo{\Bigl( \frac{\Xi_{n, i}}{\sqrt{n}} \Bigr)_{i \leq m}} A \Bigl( \frac{\Xi_{n, i}}{\sqrt{n}} \Bigr)_{i \leq m}\\
&\mspace{360mu} + R\biggl( \frac{\log(U)}{n}, \Bigl( \frac{\Xi_{n, i}}{\sqrt{n}} \Bigr)_{i \leq m} \biggr) \biggr),
\end{align*}
or equivalently via rearrangement,
\begin{align*}
m n \bigl( 1 - m^{1/p_1 - 1/p_2} \lVert X^n \rVert_{p_2, q_1} \bigr) &\GlVert -\log(U) - \frac{\log(U)^2}{2 m n} + \frac{p_1 - p_2}{2 p_1^2 m} \trapo{(\Xi_{n, i})_{i \leq m}} A (\Xi_{n, i})_{i \leq m}\\
&\mspace{190mu} - m n R\biggl( \frac{\log(U)}{n}, \Bigl( \frac{\Xi_{n, i}}{\sqrt{n}} \Bigr)_{i \leq m} \biggr).
\end{align*}
We choose \((l, \alpha_n, \beta_n) := (3, n^{1/3}, n)\) for the remainder lemma; indeed, \(n^{1/3} \bigl( \frac{\log(U)}{n}, \bigl( \frac{\Xi_{n, i}}{\sqrt{n}} \bigr)_{i \leq m} \bigr) = \bigl( n^{-2/3} \log(U), n^{-1/6} (\Xi_{n, i})_{i \leq m} \bigr)\) converges to \(\mathbf{0}\) in probability as \(n \to \infty\), therefore the remainder lemma implies \(\bigl( n R\bigl( \frac{\log(U)}{n}, \bigl( \frac{\Xi_{n, i}}{\sqrt{n}} \bigr)_{i \leq m} \bigr) \bigr)_{n \geq 1} \KiWsk{} 0\). Additionally we have \((\frac{\log(U)^2}{n})_{n \geq 1} \to 0\) almost surely and hence in probability. Thus via Slutsky's theorem we obtain
\begin{equation*}
\bigl( m n \bigl( 1 - m^{1/p_1 - 1/p_2} \lVert X^n \rVert_{p_2, q_1} \bigr) \bigr)_{n \geq 1} \KiVert{} -\log(U) + \frac{p_1 - p_2}{2 p_1 m} \trapo{(N_{1, i})_{i \leq m}} A (N_{1, i})_{i \leq m},
\end{equation*}
and it remains to argue that the right\-/hand side has the claimed distribution. That \(-\log(U) \sim \Exp(1)\), is common lore. Since \((\xi_i)_{i \leq m}\) is independent from \(U\), \((N_{1, i})_{i \leq m}\) can be assumed independent from \(U\). The matrix \(A\) is symmetric and has  eigenvalues \(m\) with multiplicity \(m - 1\) and \(0\) with multiplicity \(1\), hence its spectral decomposition reads \(A = O \diag(m, \dotsc, m, 0) \trapo{O}\) with orthogonal \(O \in \RZ^{m \times m}\). The standard Gaussian distribution is orthogonally invariant, that is \((N_i)_{i \leq m} := \trapo{O} (N_{1, i})_{i \leq m} \sim \Nvert(\mathbf{0}, I_m)\), and thereby
\begin{equation*}
\trapo{(N_{1, i})_{i \leq m}} A (N_{1, i})_{i \leq m} = \trapo{(N_i)_{i \leq m}} \diag(m, \dotsc, m, 0) (N_i)_{i \leq m} = m \sum_{i = 1}^{m - 1} N_i^2.
\end{equation*}
Because \((N_i)_{i \leq m}\) still is independent from \(U\) we have finished.

\item
Using the same expansion as in (a), case \(p_1 = \infty\), and restricting to \((y_i)_{i \leq m} = \mathbf{0}\) like in (b) while naming \(R'(x) := R(x, \mathbf{0})\), we arrive at
\begin{align*}
\lVert X^n \rVert_{p_2, q_1} &\GlVert \biggl( \sum_{i = 1}^m \lvert \xi_i \rvert^{p_2/n} \biggr)^{1/p_2}\\
&= m^{1/p_2} \biggl( 1 - \sum_{i = 1}^m \frac{\Xi_{n, i}}{p_2 m n} + R'\biggl( \Bigl( \frac{\Xi_{n, i}}{n} \Bigr)_{i \leq m} \biggr) \biggr).
\end{align*}
The result follows, after a rearrangement, from \(\frac{1}{p_1} (\Xi_{n, i})_{i \leq m} \KiVert{} \Exp(1)^{\otimes m}\), managing the remainder term as in (b) above.\qedhere
\end{asparaenum}
\end{proof}

In Theorem~\ref{sa:zgs3} now we consider \(n \to \infty\) and \(m = m(n) \to \infty\). The proof features the Lyapunov CLT: let \(((Z_{n, i})_{i \leq m})_{n \geq 1}\) be an array of \(\RZ\)\=/valued random variables with independent rows (i.e., for any \(n \in \NZ\) the variables \(Z_{n, 1}, \dotsc, Z_{n, m}\) are independent), and call \(s_n := \bigl( \sum_{i = 1}^m \Var[Z_{n, i}] \bigr)^{1/2}\). If \(s_n > 0\) for all \(n \in \NZ\) and \emph{Lyapunov's condition} is statisfied, sc., there exists some \(\delta > 0\) with
\begin{equation}\label{eq:ljapunov_bed}
\lim_{n \to \infty} \frac{1}{s_n^{2 + \delta}} \sum_{i = 1}^m \Erw\bigl[ \lvert Z_{n, i} - \Erw[Z_{n, i}] \rvert^{2 + \delta} \bigr] = 0,
\end{equation}
then
\begin{equation*}
\frac{1}{s_n} \sum_{i = 1}^m \bigl( Z_{n, i} - \Erw[Z_{n, i}] \bigr) \KiVert{n \to \infty} \Nvert(0, 1).
\end{equation*}
As an aside, note that actually Lyapunov's condition implies Lindeberg's condition which in its turn implies the CLT.

\begin{proof}[Proof of Theorem~\ref{sa:zgs3}]
\begin{asparaenum}[(a)]
\item \textit{Case~\(p_1 < \infty\):} Recall the representation
\begin{equation*}
\lVert X^n \rVert_{p_2, q_2} \GlVert U^{1/(m n)} \, \frac{\bigl( \sum_{i = 1}^m \lvert \xi_i \rvert^{p_2/n} \lVert \Theta_i \rVert_{q_2}^{p_2} \bigr)^{1/p_2}}{\bigl( \sum_{i = 1}^m \lvert \xi_i \rvert^{p_1/n} \bigr)^{1/p_1}}.
\end{equation*}
Define the random variables \(\Xi_{n, i}^1\), \(\Xi_{n, i}^2\), and \(\Xi_{n, i}^3\) by
\begin{align*}
\Xi_{n, i}^1 &:= \lvert \xi_i \rvert^{p_1/n} - 1, & \Xi_{n, i}^2 &:= \frac{\lvert \xi_i \rvert^{p_2/n} - M_{p_1/n}^{p_2/n}}{M_{p_1/n}^{p_2/n}},\\
\Xi_{n, i}^3 &:= \frac{\lVert \Theta_i \rVert_{q_2}^{p_2} - \Erw[\lVert \Theta_1 \rVert_{q_2}^{p_2}]}{\Erw[\lVert \Theta_1 \rVert_{q_2}^{p_2}]},
\end{align*}
and their sums
\begin{equation*}
Z_n^k := \sum_{i = 1}^m \Xi_{n, i}^k \quad \text{for \(k \in \{1, 2, 3\}\)\qquad{}and} \qquad Z_n^4 := \sum_{i = 1}^m \Xi_{n, i}^2 \Xi_{n, i}^3,
\end{equation*}
then
\begin{align*}
\lVert X^n \rVert_{p_2, q_2} &\GlVert \frac{(M_{p_1/n}^{p_2/n} \Erw[\lVert \Theta_1 \rVert_{q_2}^{p_2}])^{1/p_2}}{m^{1/p_1 - 1/p_2}} \, U^{1/(m n)} \, \frac{\bigl( \frac{1}{m} \sum_{i = 1}^m (1 + \Xi_{n, i}^2) (1 + \Xi_{n, i}^3) \bigr)^{1/p_2}}{\bigl( \frac{1}{m} \sum_{i = 1}^m (1 + \Xi_{n, i}^1) \bigr)^{1/p_1}}\\
&= \frac{(M_{p_1/n}^{p_2/n} \Erw[\lVert \Theta_1 \rVert_{q_2}^{p_2}])^{1/p_2}}{m^{1/p_1 - 1/p_2}} \, U^{1/(m n)} \frac{\bigl( 1 + \frac{1}{m} Z_n^2 + \frac{1}{m} Z_n^3 + \frac{1}{m} Z_n^4 \bigr)^{1/p_2}}{\bigl( 1 + \frac{1}{m} Z_n^1 \bigr)^{1/p_1}}\\
&= \frac{(M_{p_1/n}^{p_2/n} \Erw[\lVert \Theta_1 \rVert_{q_2}^{p_2}])^{1/p_2}}{m^{1/p_1 - 1/p_2}} \biggl( 1 + \frac{\log(U)}{m n} - \frac{Z_n^1}{p_1 m} + \frac{Z_n^2 + Z_n^3 + Z_n^4}{p_2 m}\\
&\mspace{235mu} + R\biggl( \frac{\log(U)}{m n}, \frac{Z_n^1}{m}, \frac{Z_n^2}{m}, \frac{Z_n^3}{m}, \frac{Z_n^4}{m} \biggr) \biggr),
\end{align*}
where we have introduced the Taylor expansion
\begin{equation*}
\ez^u \, \frac{(1 + z_2 + z_3 + z_4)^{1/p_2}}{(1 + z_1)^{1/p_1}} = 1 + u - \frac{z_1}{p_1} + \frac{z_2 + z_3 + z_4}{p_2} + R(u, z_1, z_2, z_3, z_4),
\end{equation*}
with the remainder term satisfying \(\lvert R(u, z_1, z_2, z_3, z_4) \rvert \leq M \lVert (u, z_1, z_2, z_3, z_4) \rVert_2^2\) in a suitable neighbourhood of \(\mathbf{0}\). Then we can rearrange as follows,
\begin{align*}
\sqrt{m n} \biggl( \frac{m^{1/p_1 - 1/p_2}}{(M_{p_1/n}^{p_2/n} \Erw[\lVert \Theta_1 \rVert_{q_2}^{p_2}])^{1/p_2}} \lVert X^n \rVert_{p_2, q_2} - 1 \biggr) &\GlVert \frac{\log(U)}{\sqrt{m n}} - \frac{\sqrt{n} \, Z_n^1}{p_1 \sqrt{m}} + \frac{\sqrt{n} (Z_n^2 + Z_n^3 + Z_n^4)}{p_2 \sqrt{m}}\\
&\quad + \sqrt{m n} \, R\biggl( \frac{\log(U)}{m n}, \frac{Z_n^1}{m}, \frac{Z_n^2}{m}, \frac{Z_n^3}{m}, \frac{Z_n^4}{m} \biggr),
\end{align*}
and we are going to argue that only \(\sqrt{\frac{n}{m}} \, Z_n^3\) makes a nontrivial contribution to the limit as \(n \to \infty\), in the sense that all other terms converge to zero in probability; Slutsky's theorem yields the result then.

\Absatz%
Clearly \(\bigl( \frac{\log(U)}{\sqrt{m n}} \bigr)_{n \geq 1} \Kfs{} 0\). Next, \(\Xi_{n, i}^1\), \(\Xi_{n, i}^2\), and \(\Xi_{n, i}^3\) are centred, and therefore so is \(Z_n^k\) for \(k \in \{1, 2, 3, 4\}\). The respective variances are as follows, where we use Lemma~\ref{lem:eigenschaften_xi}, 1.(a),
\begin{align*}
\Var\biggl[ \frac{Z_n^1}{m} \biggr] &= \frac{1}{m} \Var\bigl[ \lvert \xi_1 \rvert^{p_1/n} - 1 \bigr] = \frac{1}{m} \, \frac{1}{p_1 n} \Bigl( 1 + \BigO\Bigl( \frac{1}{n} \Bigr) \Bigr) = \BigTheta\Bigl( \frac{1}{m n} \Bigr);\\
\intertext{analogously,}
\Var\biggl[ \frac{Z_n^2}{m} \biggr] &= \frac{\Var\bigl[ \lvert \xi_1 \rvert^{p_2/n} - M_{p_1/n}^{p_2/n} \bigr]}{m (M_{p_1/n}^{p_2/n})^2} = \frac{1}{m} \, \frac{\frac{p_2^2}{p_1 n} \bigl( 1 + \BigO\bigl( \frac{1}{n} \bigr) \bigr)}{\bigl( 1 + \BigO\bigl( \frac{1}{n} \bigr) \bigr)^2} = \BigTheta\Bigl( \frac{1}{m n} \Bigr).
\end{align*}
Defining \(Y_n\) as in Lemma~\ref{lem:zgs_lp} and employing \((\Erw[Y_n^k])_{n \geq 1} \to p_2^k \sigma^k \Erw[N^k]\) as stated there, we get
\begin{align*}
\Erw[\lVert \Theta_1 \rVert_{q_2}^{p_2}] &= \frac{(M_{q_1}^{q_2})^{p_2/q_2}}{n^{p_2 (1/q_1 - 1/q_2)}} \Bigl( 1 + \frac{\Erw[Y_n]}{\sqrt{n}} \Bigr) = \frac{(M_{q_1}^{q_2})^{p_2/q_2}}{n^{p_2 (1/q_1 - 1/q_2)}} \Bigl( 1 + \smallO\Bigl( \frac{1}{\sqrt{n}} \Bigr) \Bigr) \label{eq:erw_norm_theta}\\
\intertext{and}
\Var[\lVert \Theta_1 \rVert_{q_2}^{p_2}] &= \frac{(M_{q_1}^{q_2})^{2 p_2/q_2}}{n^{2 p_2 (1/q_1 - 1/q_2)}} \, \frac{\Var[Y_n]}{n} = \frac{p_2^2 \sigma^2 (M_{q_1}^{q_2})^{2 p_2/q_2}}{n^{2 p_2 (1/q_1 - 1/q_2) + 1}} \bigl( 1 + \smallO(1) \bigr); \notag
\end{align*}
these lead to
\begin{equation*}
\Var\biggl[ \frac{Z_n^3}{m} \biggr] = \frac{1}{m} \, \frac{\Var[\lVert \Theta_1 \rVert_{q_2}^{p_2}]}{\Erw[\lVert \Theta_1 \rVert_{q_2}^{p_2}]^2} = \frac{p_2^2 \sigma^2}{m n} \bigl( 1 + \smallO(1) \bigr),
\end{equation*}
and therewith also to
\begin{equation*}
\Var\biggl[ \frac{Z_n^4}{m} \biggr] = \frac{1}{m} \, \frac{\Var[\lvert \xi_1 \rvert^{p_2/n}] \Var[\lVert \Theta_1 \rVert_{q_2}^{p_2}]}{(M_{p_1/n}^{p_2/n})^2 \Erw[\lVert \Theta_1 \rVert_{q_2}^{p_2}]^2} = \BigTheta\Bigl( \frac{1}{m n^2} \Bigr).
\end{equation*}
We further note that
\begin{align*}
\sum_{i = 1}^m \Erw[\lvert \Xi_{n, i}^3 \rvert^3] &= \frac{m}{\Erw[\lVert \Theta_1 \rVert_{q_2}^{p_2}]^3} \Erw\biggl[ \frac{(M_{q_1}^{q_2})^{3 p_2/q_2}}{n^{3 p_2 (1/q_1 - 1/q_2)}} \, \biggl\lvert \frac{Y_n - \Erw[Y_n]}{\sqrt{n}} \biggr\rvert^3 \biggr]\\
&= \frac{m}{n^{3/2}} p_2^3 \sigma^3 \Erw[\lvert N \rvert^3] \bigl( 1 + \smallO(1) \bigr).
\end{align*}
But then the array \(((\Xi_{n, i}^3)_{i \leq m})_{n \geq 1}\) satisfies Lyapunov's condition~\eqref{eq:ljapunov_bed} with \(\delta = 1\) since
\begin{equation*}
\frac{\sum_{i = 1}^m \Erw[\lvert \Xi_{n, i}^3 \rvert^3]}{\Var[Z_n^3]^{3/2}} = \frac{\BigTheta(m n^{-3/2})}{\BigTheta(m^{3/2} n^{-3/2})} = \BigTheta\Bigl( \frac{1}{\sqrt{m}} \Bigr),
\end{equation*}
and hence we get the CLT \((\Var[Z_n^3]^{-1/2} Z_n^3)_{n \geq 1} \KiVert{} N\), or rather
\begin{equation*}
\biggl( \sqrt{\frac{n}{m}} \, Z_n^3 \biggr)_{n \geq 1} \KiVert{} p_2 \sigma N.
\end{equation*}
On the other hand there is still \(\Var\bigl[ \sqrt{\frac{n}{m}} \, Z_n^4 \bigr] = \BigTheta(\frac{1}{n})\), and a little bit less obviously (employ Lemma~\ref{lem:eigenschaften_xi} again; alternatively, Lemma~\ref{lem:xi_differenz} with weaker asymptotics),
\begin{equation*}
\Var\biggl[ \sqrt{\frac{n}{m}} \Bigl( \frac{Z_n^2}{p_2} - \frac{Z_n^1}{p_1} \Bigr) \biggr] = n \biggl( \frac{V_{p_1/n}^{p_2/n}}{(p_2 M_{p_1/n}^{p_2/n})^2} - \frac{2 C_{p_1/n}^{p_1/n, p_2/n}}{p_1 p_2 M_{p_1/n}^{p_2/n}} + \frac{V_{p_1/n}^{p_1/n}}{p_1^2} \biggr) = \frac{(p_2 - p_1)^2}{2 p_1^2 \, n} \Bigl( 1 + \BigO\Bigl( \frac{1}{n} \Bigr) \Bigr)
\end{equation*}
wherefore by \v{C}eby\v{s}\"{e}v's inequality both \(\sqrt{\frac{n}{m}} \bigl( \frac{Z_n^2}{p_2} - \frac{Z_n^1}{p_1} \bigr)\) and \(\sqrt{\frac{n}{m}} \, Z_n^4\) converge to zero in probability. For the remainder term we consider
\begin{equation*}
(m n)^{1/4} \biggl( \frac{\log(U)}{m n}, \frac{Z_n^1}{m}, \frac{Z_n^2}{m}, \frac{Z_n^3}{m}, \frac{Z_n^4}{m} \biggr) = \biggl( \frac{\log(U)}{(m n)^{3/4}}, \frac{n^{1/4} \, Z_n^1}{m^{3/4}}, \frac{n^{1/4} \, Z_n^2}{m^{3/4}}, \frac{n^{1/4} \, Z_n^3}{m^{3/4}}, \frac{n^{1/4} \, Z_n^4}{m^{3/4}} \biggr).
\end{equation*}
Again \(((m n)^{-3/4} \log(U))_{n \geq 1} \Kfs{} 0\) is obvious. For \(k \in \{1, 2, 3\}\) we have \(\Var[m^{-3/4} n^{1/4} Z_n^k] = \BigTheta((m n)^{-1/2})\), and also \(\Var[m^{-3/4} n^{1/4} Z_n^4] = \BigTheta(m^{-1/2} n^{-3/2})\); hence the respective components converge to zero in probability via \v{C}eby\v{s}\"ev's inequality. This establishes that the conditions of the remainder lemma are met, and finally the remainder term converges to zero in probability.

\Absatz%
\textit{Case~\(p_1 = \infty\):} There is no need to iterate the argument in its entirety; for one, omit \(U\) and \(\Xi_n^1\) from the previous case, and for another, we have different asymptotics for the variances; to wit, referring to Lemma~\ref{lem:eigenschaften_xi}, 1.(b),
\begin{equation*}
\Var\biggl[ \frac{Z_n^2}{m} \biggr] = \frac{1}{m} \, \frac{V_\infty^{p_2/n}}{(M_\infty^{p_2/n})^2} = \frac{\frac{p_2^2}{n^2} \bigl( 1 + \BigO\bigl( \frac{1}{n} \bigr) \bigr)}{m \bigl( 1 + \BigO\bigl( \frac{1}{n} \Bigr) \bigr)^2} = \frac{p_2^2}{m n^2} \Bigl( 1 + \BigO\Bigl( \frac{1}{n} \Bigr) \Bigr),
\end{equation*}
and correspondingly,
\begin{equation*}
\Var\biggl[ \frac{Z_n^4}{m} \biggr] = \frac{p_2^4 \sigma^2}{m n^3} \bigl( 1 + \smallO(1) \bigr).
\end{equation*}
These imply \(\Var\bigl[ \sqrt{\frac{n}{m}} \, Z_n^2 \bigr] = \BigTheta(\frac{1}{n})\) and \(\Var\bigl[ \sqrt{\frac{n}{m}} \, Z_n^4 \bigr] = \BigTheta(\frac{1}{n^2})\), hence \(\bigl( \sqrt{\frac{n}{m}} \, Z_n^2 \bigr)_{n \in \NZ} \KiWsk{} 0\) and \(\bigl( \sqrt{\frac{n}{m}} \, Z_n^4 \bigr)_{n \in \NZ} \KiWsk{} 0\) via \v{C}eby\v{s}\"ev's inequality. For the remainder term we consider
\begin{equation*}
(m n)^{1/4} \biggl( \frac{Z_n^2}{m}, \frac{Z_n^3}{m}, \frac{Z_n^4}{m} \biggr) = \biggl( \frac{n^{1/4} \, Z_n^2}{m^{3/4}}, \frac{n^{1/4} \, Z_n^3}{m^{3/4}}, \frac{n^{1/4} \, Z_n^4}{m^{3/4}} \biggr).
\end{equation*}
We have \(\Var[m^{-3/4} n^{1/4} Z_n^2] = \BigTheta(m^{-1/2} n^{-3/2})\) and also \(\Var[m^{-3/4} n^{1/4} Z_n^4] = \BigTheta(m^{-1/2} n^{-5/2})\), and \(\Var[m^{-3/4} n^{1/4} Z_n^3] = \BigTheta((m n)^{-1/2})\) remains unchanged; hence the respective components converge to zero in probability via \v{C}eby\v{s}\"ev's inequality. The remainder lemma does the rest.

\item%
Define \(\Xi_{n, i}^1\), \(\Xi_{n, i}^2\), \(Z_n^1\), and \(Z_n^2\) as in (a), then we have the probabilistic representation
\begin{align*}
\lVert X^n \rVert_{p_2, q_1} &\GlVert U^{1/(m n)} \, \frac{\bigl( \sum_{i = 1}^m \lvert \xi_i \rvert^{p_2/n} \bigr)^{1/p_2}}{\bigl( \sum_{i = 1}^m \lvert \xi_i \rvert^{p_1/n} \bigr)^{1/p_1}}\\
&= U^{1/(m n)} \, \frac{\bigl( m M_{p_1/n}^{p_2/n} + M_{p_1/n}^{p_2/n} Z_n^2 \bigr)^{1/p_2}}{(m + Z_n^1)^{1/p_1}}\\
&= \frac{(M_{p_1/n}^{p_2/n})^{1/p_2}}{m^{1/p_1 - 1/p_2}} \, U^{1/(m n)} \, \frac{\bigl( 1 + \frac{Z_n^2}{m} \bigr)^{1/p_2}}{\bigl( 1 + \frac{Z_n^1}{m} \bigr)^{1/p_1}}\\
&= \frac{(M_{p_1/n}^{p_2/n})^{1/p_2}}{m^{1/p_1 - 1/p_2}} \biggl( 1 + \frac{\log(U)}{m n} - \frac{Z_n^1}{p_1 m} + \frac{Z_n^2}{p_2 m} + R\Bigl( \frac{\log(U)}{m n}, \frac{Z_n^1}{m}, \frac{Z_n^2}{m} \Bigr) \biggr),
\end{align*}
where we have used the same Taylor expansion as in (a), but evaluated at \(z_3 = z_4 = 0\), and the remainder term satsifies \(\lvert R(u, z_1, z_2) \rvert \leq M \lVert (u, z_1, z_2) \rVert_2^2\) with some \(M > 0\) in a neighbourhood of \((0, 0, 0)\). Rearranging yields
\begin{equation*}
\sqrt{m} \, n \biggl( \frac{m^{1/p_1 - 1/p_2}}{(M_{p_1/n}^{p_2/n})^{1/p_2}} \, \lVert X^n \rVert_{p_2, q_1} - 1 \biggr) \GlVert \frac{\log(U)}{\sqrt{m}} + \frac{n}{\sqrt{m}} \Bigl( \frac{Z_n^2}{p_2} - \frac{Z_n^1}{p_1} \Bigr) + \sqrt{m} \, n R\Bigl( \frac{\log(U)}{m n}, \frac{Z_n^1}{m}, \frac{Z_n^2}{m} \Bigr).
\end{equation*}
We know \((m^{-1/2} \log(U))_{n \geq 1} \Kfs{} 0\); in order to apply the remainder lemma we have to ensure
\begin{equation*}
m^{1/4} n^{1/2} \Bigl( \frac{\log(U)}{m n}, \frac{Z_n^1}{m}, \frac{Z_n^2}{m} \Bigr) \KiWsk{n \to \infty} (0, 0, 0),
\end{equation*}
but we have \(m^{1/4} n^{1/2} \frac{\log(U)}{m n} = m^{-3/4} n^{-1/2} \log(U) \Kfs{n \to \infty} 0\), and with reference to the proof of (a), \(\Var\bigl[ m^{1/4} n^{1/2} \frac{Z_n^k}{m} \bigr] = m^{1/2} n \BigTheta(\frac{1}{m n}) = \BigTheta(m^{-1/2}) \Konv{n \to \infty} 0\) for \(k \in \{1, 2\}\), hence via \v{C}eby\v{s}\"{e}v's inequality \(m^{1/4} n^{1/2} \bigl( \frac{Z_n^1}{m}, \frac{Z_n^2}{m} \bigr) \KiWsk{n \to \infty} (0, 0)\).

\Absatz%
It remains to show that \(\frac{n}{\sqrt{m}} \bigl( \frac{Z_n^2}{p_2} - \frac{Z_n^1}{p_1} \bigr)\) converges to the claimed distribution. Written out, the term in parentheses reads
\begin{equation*}
\frac{Z_n^2}{p_2} - \frac{Z_n^1}{p_1} = \sum_{i = 1}^m \biggl( \frac{\lvert \xi_i \rvert^{p_2/n} - M_{p_1/n}^{p_2/n}}{p_2 M_{p_1/n}^{p_2/n}} - \frac{\lvert \xi_i \rvert^{p_1/n} - 1}{p_1} \biggr);
\end{equation*}
call \(Z_{n, i} := \frac{\lvert \xi_i \rvert^{p_2/n} - M_{p_1/n}^{p_2/n}}{p_2 M_{p_1/n}^{p_2/n}} - \frac{\lvert \xi_i \rvert^{p_1/n} - 1}{p_1}\), then \(\Erw[Z_{n, i}] = 0\), and we are going to demonstate that the array \(((Z_{n, i})_{i \leq m})_{n \geq 1}\) satisfies Lyapunov's condition. Let \(\delta > 0\), then from Lemma~\ref{lem:xi_differenz} we get, for any \(i \in [1, m]\),
\begin{align*}
\Erw\bigl[ \lvert Z_{n, i} \rvert^{2 + \delta} \bigr] = n^{-2 - \delta} \Erw\bigl[ \lvert n Z_{n, i} \rvert^{2 + \delta} \bigr] = n^{-2 - \delta} \Bigl\lvert \frac{p_2 - p_1}{2 p_1} \Bigr\rvert^{2 + \delta} \Erw\bigl[ \lvert N^2 - 1 \rvert^{2 + \delta} \bigr] (1 + \smallO(1)),
\end{align*}
and as we have established already in the proof of (a),
\begin{equation*}
s_n^2 := \Var\biggl[ \sum_{i = 1}^m Z_{n, i} \biggr] = \frac{m (p_2 - p_1)^2}{2 p_1^2 n^2} \Bigl( 1 + \BigO\Bigl( \frac{1}{n} \Bigr) \Bigr).
\end{equation*}
These now show
\begin{align*}
\frac{1}{s_n^{2 + \delta}} \sum_{i = 1}^m \Erw\bigl[ \lvert Z_{n, i} \rvert^{2 + \delta} \bigr] &= \frac{m n^{-2 - \delta} \Bigl\lvert \frac{p_2 - p_1}{2 p_1} \bigr\rvert^{2 + \delta} \Erw\bigl[ \lvert N^2 - 1 \rvert^{2 + \delta} \bigr] (1 + \smallO(1))}{m^{1 + \delta/2} n^{-2 - \delta} \bigl\lvert \frac{p_2 - p_1}{\sqrt{2} \, p_1} \bigr\rvert^{2 + \delta} \bigl( 1 + \BigO\bigl( \frac{1}{n} \bigr) \bigr)}\\
&= \frac{1}{m^{\delta/2}} \, \frac{\Erw\bigl[ \lvert N^2 - 1 \rvert^{2 + \delta} \bigr]}{2^{1 + \delta/2}} \, (1 + \smallO(1)) \Konv{n \to \infty} 0,
\end{align*}
that is, Lyapunov's condition is satisfied, and hence we have the CLT
\begin{equation*}
\frac{1}{s_n} \sum_{i = 1}^m Z_{n, i} = \frac{\sqrt{2} \, p_1}{\lvert p_2 - p_1 \rvert} \, \frac{n}{\sqrt{m}} \Bigl( 1 + \BigO\Bigl( \frac{1}{n} \Bigr) \Bigr) \Bigl( \frac{Z_n^2}{p_2} - \frac{Z_n^1}{p_1} \Bigr) \KiVert{n \to \infty} N \sim \Nvert(0, 1),
\end{equation*}
equivalently,
\begin{equation*}
\frac{n}{\sqrt{m}} \Bigl( \frac{Z_n^2}{p_2} - \frac{Z_n^1}{p_1} \Bigr) \KiVert{n \to \infty} \frac{\lvert p_2 - p_1 \rvert}{\sqrt{2} \, p_1} \, N,
\end{equation*}
which concludes this part's proof.

\item%
Keeping \(Z_n^2\) as before, we have
\begin{align*}
\lVert X^n \rVert_{p_2, q_1} &\GlVert \biggl( \sum_{i = 1}^m \lvert \xi_i \rvert^{p_2/n} \biggr)^{1/p_2}\\
&= m^{1/p_2} (M_\infty^{p_2/n})^{1/p_2} \biggl( 1 + \frac{Z_n^2}{m} \biggr)^{1/p_2}\\
&= m^{1/p_2} (M_\infty^{p_2/n})^{1/p_2} \biggl( 1 + \frac{Z_n^2}{p_2 m} + R\biggl( \frac{Z_n^2}{m} \biggr) \biggr),
\end{align*}
equivalently
\begin{equation*}
\sqrt{m} \, n \biggl( \frac{\lVert X^n \rVert_{p_2, q_1}}{m^{1/p_2} (M_\infty^{p_2/n})^{1/p_2}} - 1 \biggr) \GlVert \frac{n}{p_2 \sqrt{m}} \, Z_n^2 + \sqrt{m} \, n R\biggl( \frac{Z_n^2}{m} \biggr),
\end{equation*}
where the remainder fulfils \(\lvert R(x) \rvert \leq M x^2\) in some neighbourhood of zero. Like in the proof of (a), case \(p_1 = \infty\), we use Lemma~\ref{lem:eigenschaften_xi}, 1.(b), to see
\begin{equation*}
\Var\biggl[ m^{1/4} n^{1/2} \, \frac{Z_n^2}{m} \biggr] = \sqrt{m} \, n \, \frac{p_2^2}{m n^2} \Bigl( 1 + \BigO\Bigl( \frac{1}{n} \Bigr) \Bigr) = \frac{p_2^2}{\sqrt{m} \, n} \Bigl( 1 + \BigO\Bigl( \frac{1}{n} \Bigr) \Bigr),
\end{equation*}
and via \v{C}eby\v{s}\"ev's inequality this immediately implies \(\bigl( m^{1/4} n^{1/2} \frac{Z_n^2}{m} \bigr)_{n \geq 1} \KiWsk{} 0\), and thus by the remainder lemma the remainder term converges to zero.

\Absatz%
Since \(Z_n^2 = \sum_{i = 1}^m \frac{\lvert \xi_i \rvert^{p_2/n} - M_\infty^{p_2/n}}{M_\infty^{p_2/n}}\), we are going to prove that the triangular array \(\bigl( \bigl( \frac{\lvert \xi_i \rvert^{p_2/n} - M_\infty^{p_2/n}}{M_\infty^{p_2/n}} \bigr)_{i \leq m} \bigr)_{n \geq 1}\) satisfies Lyapunov's condition with any \(\delta \in (0, \infty)\) (the components are centred already); so let \(\delta \in (0, \infty)\). Employing Lemma~\ref{lem:eigenschaften_xi}, 1.(b), again we have, for any \(i \in [1, m]\),
\begin{equation*}
\Erw\biggl[ \biggl\lvert \frac{\lvert \xi_i \rvert^{p_2/n} - M_\infty^{p_2/n}}{M_\infty^{p_2/n}} \biggr\rvert^{2 + \delta} \biggr] = \frac{\bigl( \frac{p_2}{n} \bigr)^{2 + \delta} \Erw[\lvert E - 1 \rvert^{2 + \delta}] (1 + \smallO(1))}{(1 + \smallO(1))^{2 + \delta}} = \frac{p_2^{2 + \delta} \Erw[\lvert E - 1 \rvert^{2 + \delta}]}{n^{2 + \delta}} (1 + \smallO(1)),
\end{equation*}
so together with the by now well\-/known \(\Var[Z_n^2] = \frac{m p_2^2}{n^2} \bigl( 1 + \BigO(\frac{1}{n}) \bigr)\) we get
\begin{align*}
\frac{\sum_{i = 1}^m \Erw\bigl[ \bigl\lvert \frac{\lvert \xi_i \rvert^{p_2/n} - M_\infty^{p_2/n}}{M_\infty^{p_2/n}} \bigr\rvert^{2 + \delta} \bigr]}{\Var[Z_n^2]^{1 + \delta/2}} &= \frac{m n^{-2 - \delta} p_2^{2 + \delta} \Erw[\lvert E - 1 \rvert^{2 + \delta}] (1 + \smallO(1))}{m^{1 + \delta/2} n^{-2 - \delta} p_2^{2 + \delta} \bigl( 1 + \BigO(\frac{1}{n}) \bigr)}\\
&= \frac{\Erw[\lvert E - 1 \rvert^{2 + \delta}]}{m^{\delta/2}} \, (1 + \smallO(1)) \Konv{n \to \infty} 0.
\end{align*}
So this shows \(\bigl( \Var[Z_n^2]^{-1/2} \, Z_n^1 \bigr)_{n \geq 1} = \bigl( \frac{n}{p_2 \sqrt{m}} \bigl( 1 + \BigO(\frac{1}{n}) \bigr) Z_n^2 \bigr)_{n \geq 1} \KiVert{} N \sim \Nvert(0, 1)\), and finally
\begin{equation*}
\biggl( \frac{n}{p_2 \sqrt{m}} \, Z_n^2 \biggr)_{n \geq 1} \KiVert{} N,
\end{equation*}
and through an application of Slutsky's theorem the proof is finished. \qedhere
\end{asparaenum}
\end{proof}

\subsection{Proofs of the corollaries}
\label{subsec:bew_volumen}

The key observation is the following: let \(X \sim \Unif(\Kug{p_1, q_1}{m, n})\), then \(\inv{(r_{p_1, q_1}^{m, n})} \, X \sim \Unif(\inv{(r_{p_1, q_1}^{m, n})} \, \Kug{p_1, q_1}{m, n})\), and therewith, for any \(t \in [0, \infty)\),
\begin{align}
V^{m, n}(t) &= \vol_{m, n}\bigl( \inv{(r_{p_1, q_1}^{m, n})} \, \Kug{p_1, q_1}{m, n} \cap t \inv{(r_{p_2, q_2}^{m, n})} \, \Kug{p_2, q_2}{m, n} \bigr)\notag\\
&= \Wsk\bigl[ \inv{(r_{p_1, q_1}^{m, n})} \, X \in t \inv{(r_{p_2, q_2}^{m, n})} \, \Kug{p_2, q_2}{m, n} \bigr] = \Wsk\biggl[ \lVert X \rVert_{p_2, q_2} \leq \frac{r_{p_1, q_1}^{m, n}}{r_{p_2, q_2}^{m, n}} \, t \biggr]. \label{eq:vmnt}
\end{align}

\begin{proof}[Proof of Corollary~\ref{sa:kritfall1}]
We use Theorem~\ref{sa:zgs1}, thus, continuing from Equation~\eqref{eq:vmnt},
\begin{align*}
V^{m, n}(t) &= \Wsk\biggl[ \sqrt{m} \biggl( \frac{m^{1/p_1 - 1/p_2}}{\bigl( M_{p_1/n}^{p_2/n} \Erw[\lVert \Theta_1 \rVert_{q_2}^{p_2}] \bigl)^{1/p_2}} \lVert X^m \rVert_{p_2, q_2} - 1 \biggr)\\
&\qquad\quad\leq \sqrt{m} \biggl( \frac{m^{1/p_1 - 1/p_2}}{\bigl( M_{p_1/n}^{p_2/n} \Erw[\lVert \Theta_1 \rVert_{q_2}^{p_2}] \bigl)^{1/p_2}} \, \frac{r_{p_1, q_1}^{m, n}}{r_{p_2, q_2}^{m, n}} \, t - 1 \biggr) \biggr].
\end{align*}
The last random variable converges in distribution to a centred nondegenerate normally distributed random variable \(N\), hence we must determine the limit of the right-hand-side. Recall the definition of the radii \(r_{p_i, q_i}^{m, n}\) at the beginning of Section~\ref{subsec:ss}; expanding the gamma\-/functions in the volumes \(\vol_{m n}(\Kug{p_i, q_i}{m, n})\) (which see Equation~\eqref{eq:kugvol}) using Stirling's formula we arrive at
\begin{equation}\label{eq:asymptotik_Ap1q1p2q2n}
\frac{m^{1/p_1 - 1/p_2}}{\bigl( M_{p_1/n}^{p_2/n} \Erw[\lVert \Theta_1 \rVert_{q_2}^{p_2}] \bigr)^{1/p_2}} \, \frac{r_{p_1, q_1}^{m, n}}{r_{p_2, q_2}^{m, n}} = \begin{cases} A_{p_1, q_1; p_2, q_2; n} \bigl( 1 + \BigO(\frac{1}{m}) \bigr) & \text{if } p_1 < \infty, \\ A_{p_1, q_1; p_2, q_2; n} \bigl( 1 + \BigO(\frac{\log(m)}{m}) \bigr) & \text{if } p_1 = \infty. \end{cases}
\end{equation}
This implies
\begin{equation*}
\lim_{m \to \infty} \frac{m^{1/p_1 - 1/p_2}}{\bigl( M_{p_1/n}^{p_2/n} \Erw[\lVert \Theta_1 \rVert_{q_2}^{p_2}] \bigr)^{1/p_2}} \, \frac{r_{p_1, q_1}^{m, n}}{r_{p_2, q_2}^{m, n}} = A_{p_1, q_1; p_2, q_2; n},
\end{equation*}
and therewith the cases \(t A_{p_1, q_1; p_2, q_2; n} < 1\) or \(t A_{p_1, q_1; p_2, q_2; n} > 1\) are immediately accounted for. In the threshold case \(t A_{p_1, q_1; p_2, q_2; n} = 1\) we need the information in Equation~\eqref{eq:asymptotik_Ap1q1p2q2n} that the correction\-/term is of order \(\smallO(m^{-1/2})\) in either case; this yields
\begin{align*}
\sqrt{m} \biggl( \frac{m^{1/p_1 - 1/p_2}}{\bigl( M_{p_1/n}^{p_2/n} \Erw[\lVert \Theta_1 \rVert_{q_2}^{p_2}] \bigr)^{1/p_2}} \, \frac{r_{p_1, q_1}^{m, n}}{r_{p_2, q_2}^{m, n}} \, t - 1 \biggr) &= \sqrt{m} \Bigl( t A_{p_1, q_1; p_2, q_2; n} \Bigl( 1 + \smallO\Bigl( \frac{1}{\sqrt{m}} \Bigr) \Bigr) - 1 \Bigr)\\
&= \sqrt{m} \smallO\Bigl( \frac{1}{\sqrt{m}} \Bigr) = \smallO(1),
\end{align*}
and finally
\begin{equation*}
\lim_{m \to \infty} V^{m, n}(t) = \Wsk[N \leq 0] = \frac{1}{2}. \qedhere
\end{equation*}
\end{proof}

\begin{proof}[Proof of Corollary~\ref{sa:kritfall2}]
\textit{Case~\(q_1 \neq q_2\):} Here we invest Theorem~\ref{sa:zgs2}, that is,
\begin{align*}
V^{m, n}(t) &= \Wsk\biggl[ \sqrt{n} \biggl( \frac{m^{1/p_1 - 1/p_2} \, n^{1/q_1 - 1/q_2}}{(M_{q_1}^{q_2})^{1/q_2}} \lVert X^n \rVert_{p_2, q_2} - 1 \biggr)\\
&\qquad\qquad \leq \sqrt{n} \biggl( \frac{m^{1/p_1 - 1/p_2} \, n^{1/q_1 - 1/q_2}}{(M_{q_1}^{q_2})^{1/q_2}} \, \frac{r_{p_1, q_1}^{m, n}}{r_{p_2, q_2}^{m, n}} \, t - 1 \biggr) \biggr].
\end{align*}
Like in the proof of Corollary~\ref{sa:kritfall1} we use asymptotic expansion; to wit, this reads
\begin{equation}\label{eq:asymptotik_Aq1q2_mfest1}
\frac{m^{1/p_1 - 1/p_2} \, n^{1/q_1 - 1/q_2}}{(M_{q_1}^{q_2})^{1/q_2}} \, \frac{r_{p_1, q_1}^{m, n}}{r_{p_2, q_2}^{m, n}} = \begin{cases} A_{q_1, q_2} \bigl( 1 + \BigO(\frac{1}{n}) \bigr) & \text{if } q_1 < \infty, \\ A_{q_1, q_2} \bigl( 1 + \BigO\bigl( \frac{\log(n)}{n} \bigr) \bigr) & \text{if } q_1 = \infty. \end{cases}
\end{equation}
On the one hand this yields
\begin{equation*}
\lim_{n \to \infty} \frac{m^{1/p_1 - 1/p_2} \, n^{1/q_1 - 1/q_2}}{(M_{q_1}^{q_2})^{1/q_2}} \, \frac{r_{p_1, q_1}^{m, n}}{r_{p_2, q_2}^{m, n}} = A_{q_1, q_2},
\end{equation*}
and therefore the result is immediate for \(t A_{q_1, q_2} < 1\) or \(t A_{q_1, q_2} > 1\). On the other hand, for the threshold\-/case \(t A_{q_1, q_2} = 1\) Equation~\eqref{eq:asymptotik_Aq1q2_mfest1} tells us that either way we are off by at most \(\smallO(n^{-1/2})\). Thus
\begin{align*}
\sqrt{n} \biggl( \frac{m^{1/p_1 - 1/p_2} \, n^{1/q_1 - 1/q_2}}{(M_{q_1}^{q_2})^{1/q_2}} \, \frac{r_{p_1, q_1}^{m, n}}{r_{p_2, q_2}^{m, n}} \, t - 1 \biggr) &= \sqrt{n} \Bigl( t A_{q_1, q_2} \Bigl( 1 + \smallO\Bigl( \frac{1}{\sqrt{n}} \Bigr) \Bigr) - 1 \Bigr)\\
&= \sqrt{n} \smallO\Bigl( \frac{1}{\sqrt{n}} \Bigr) = \smallO(1),
\end{align*}
and the conclusion follows as in the proof of Corollary~\ref{sa:kritfall1}.

\textit{Case~\(q_1 = q_2\) and \(p_1 < \infty\):} The starting point is similar to the above,
\begin{align*}
V^{m, n}(t) &= \Wsk\biggl[ m n \biggl( 1 - \frac{m^{1/p_1}}{m^{1/p_2}} \lVert X^n \rVert_{p_2, q_1} \biggr) \geq m n \biggl( 1 - \frac{m^{1/p_1}}{m^{1/p_2}} \, \frac{r_{p_1, q_1}^{m, n}}{r_{p_2, q_1}^{m, n}} \, t \biggr) \biggr].
\end{align*}
The asymptotic reads
\begin{equation}\label{eq:asymptotik_Aq1q2_mfest2}
\frac{m^{1/p_1}}{m^{1/p_2}} \, \frac{r_{p_1, q_1}^{m, n}}{r_{p_2, q_1}^{m, n}} = 1 - \frac{(m - 1) \log(p_1/p_2)}{2 m n} + \BigO\Bigl( \frac{1}{n^2} \Bigr);
\end{equation}
this implies
\begin{equation*}
\lim_{n \to \infty} \frac{m^{1/p_1}}{m^{1/p_2}} \, \frac{r_{p_1, q_1}^{m, n}}{r_{p_2, q_1}^{m, n}} = 1 = A_{q_1, q_1},
\end{equation*}
and again the result follows easily for \(t A_{q_1, q_1} = t < 1\) or \(t > 1\). In the threshold case \(t = 1\) there follows from Equation~\eqref{eq:asymptotik_Aq1q2_mfest2},
\begin{equation*}
m n \biggl( 1 - \frac{m^{1/p_1}}{m^{1/p_2}} \, \frac{r_{p_1, q_1}^{m, n}}{r_{p_2, q_1}^{m, n}} \, t \biggr) = \frac{(m - 1) \log(p_1/p_2)}{2} + \BigO\Bigl( \frac{1}{n} \Bigr),
\end{equation*}
and therefore
\begin{equation*}
\lim_{n \to \infty} V^{m, n}(1) = \Wsk\biggl[ E + \frac{p_1 - p_2}{2 p_1} \sum_{i = 1}^{m - 1} N_i^2 \geq \frac{(m - 1) \log(p_1/p_2)}{2} \biggr].
\end{equation*}
For \(m = 1\) this simplifies to
\begin{equation*}
\lim_{n \to \infty} V^{1, n}(1) = \Wsk[E \geq 0] = 1.
\end{equation*}
In the case \(m \geq 2\) note that \(\sum_{i = 1}^{m - 1} N_i^2 \sim \Gamma(\frac{m - 1}{2}, 2)\) (the chi\=/squared distribution with \(m - 1\) degrees of freedom). Because \(E\) and \(\sum_{i = 1}^{m - 1} N_i^2\) are independent we can compute the given probability explicitly as follows, where we have to distinguish whether \(p_1 < p_2\) or \(p_1 > p_2\). Here we treat only the former in detail; the latter is done analogously. First we have
\begin{multline*}
\Wsk\biggl[ E + \frac{p_1 - p_2}{2 p_1} \sum_{i = 1}^{m - 1} N_i^2 \geq \frac{(m - 1) \log(p_1/p_2)}{2} \biggr]\\
= \int_0^\infty \frac{x^{(m - 3)/2}}{2^{(m - 1)/2} \Gamma(\frac{m - 1}{2})} \ez^{-x/2} \int_{\frac{(m - 1) \log(p_1/p_2)}{2} - \frac{(p_1 - p_2) x}{2 p_1}}^\infty \ez^{-y} \chF_{\RZ_{\geq 0}}(y) \, \dif y \, \dif x;
\end{multline*}
then for any \(x > 0\) we have \(\frac{(m - 1) \log(p_1/p_2)}{2} - \frac{(p_1 - p_2) x}{2 p_1} \geq 0\) iff \(x \geq \frac{p_1 (m - 1) \log(p_1/p_2)}{p_1 - p_2}\), and since \(p_1 < p_2\) this last expression is positive. Hence for \(x < \frac{p_1 (m - 1) \log(p_1/p_2)}{p_1 - p_2}\) the inner integral is \(1\), and else it is
\begin{equation*}
\int_{\frac{(m - 1) \log(p_1/p_2)}{2} - \frac{(p_1 - p_2) x}{2 p_1}}^\infty \ez^{-y} \chF_{\RZ_{\geq 0}}(y) \, \dif y = \ez^{-\frac{(m - 1) \log(p_1/p_2)}{2} + \frac{(p_1 - p_2) x}{2 p_1}} = \Bigl( \frac{p_2}{p_1} \Bigr)^{(m - 1)/2} \, \ez^{\frac{x}{2} - \frac{p_2 x}{2 p_1}}.
\end{equation*}
So by splitting up the outer integral we reach
\begin{align*}
\Wsk\biggl[ E + \frac{p_1 - p_2}{2 p_1} &\sum_{i = 1}^{m - 1} N_i^2 \geq \frac{(m - 1) \log(p_1/p_2)}{2} \biggr]\\
&= \int_0^{\frac{p_1 (m - 1) \log(p_1/p_2)}{p_1 - p_2}} \frac{x^{(m - 3)/2}}{2^{(m - 1)/2} \Gamma(\frac{m - 1}{2})} \ez^{-x/2} \, \dif x\\
&\quad + \int_{\frac{p_1 (m - 1) \log(p_1/p_2)}{p_1 - p_2}}^\infty \frac{x^{(m - 3)/2}}{2^{(m - 1)/2} \Gamma(\frac{m - 1}{2})} \ez^{-x/2} \Bigl( \frac{p_2}{p_1} \Bigr)^{(m - 1)/2} \, \ez^{\frac{x}{2} - \frac{p_2 x}{2 p_1}} \, \dif x;
\end{align*}
now the first summand obviously equals
\begin{equation*}
\int_0^{\frac{p_1 (m - 1) \log(p_1/p_2)}{p_1 - p_2}} \frac{x^{(m - 3)/2}}{2^{(m - 1)/2} \Gamma(\frac{m - 1}{2})} \ez^{-x/2} \, \dif x = \Gamma\Bigl( \frac{m - 1}{2}, 2 \Bigr)\Biggl( \biggl( 0, \frac{p_1 (m - 1) \log(\frac{p_1}{p_2})}{p_1 - p_2} \biggr] \Biggr),
\end{equation*}
and the second summand evaluates to
\begin{align*}
\int_{\frac{p_1 (m - 1) \log(p_1/p_2)}{p_1 - p_2}}^\infty \frac{x^{(m - 3)/2}}{2^{(m - 1)/2} \Gamma(\frac{m - 1}{2})} &\ez^{-x/2} \Bigl( \frac{p_2}{p_1} \Bigr)^{(m - 1)/2} \, \ez^{\frac{x}{2} - \frac{p_2 x}{2 p_1}} \, \dif x\\
&= \int_{\frac{p_1 (m - 1) \log(p_1/p_2)}{p_1 - p_2}}^\infty \frac{x^{(m - 3)/2}}{\bigl( \frac{2 p_1}{p_2} \bigr)^{(m - 1)/2} \Gamma(\frac{m - 1}{2})} \ez^{-\frac{x}{2 p_1/p_2}} \, \dif x\\
&= \Gamma\Bigl( \frac{m - 1}{2}, \frac{2 p_1}{p_2} \Bigr)\Biggl( \biggl( \frac{p_1 (m - 1) \log(\frac{p_1}{p_2})}{p_1 - p_2}, \infty \biggr) \Biggr),
\end{align*}
and because of \(\min\{1, \frac{p_1}{p_2}\} = \frac{p_1}{p_2}\), \(\max\{1, \frac{p_1}{p_2}\} = 1\) this equals the claimed expression.

\textit{Case~\(q_1 = q_2\) and \(p_1 = \infty\):} Now we are working with
\begin{equation*}
V^{m, n}(t) = \Wsk\biggl[ m n (1 - m^{-1/p_2} \lVert X^n \rVert_{p_2, q_1}) \geq m n \biggl( 1 - \frac{r_{\infty, q_1}^{m, n}}{m^{1/p_2} \, r_{p_2, q_1}^{m, n}} \, t \biggr) \biggr].
\end{equation*}
This time we have
\begin{equation}\label{eq:asymptotik_Aq1q2_mfest3}
\frac{r_{\infty, q_1}^{m, n}}{m^{1/p_2} \, r_{p_2, q_1}^{m, n}} = 1 - \frac{(m - 1) \log(n)}{2 m n} + \BigO\Bigl( \frac{1}{n} \Bigr),
\end{equation}
whence we see
\begin{equation*}
\lim_{n \to \infty} \frac{r_{\infty, q_1}^{m, n}}{m^{1/p_2} \, r_{p_2, q_1}^{m, n}} = 1 = A_{q_1, q_1},
\end{equation*}
and the noncritical cases follow as before. For the threshold\-/value \(t = 1\), Equation~\eqref{eq:asymptotik_Aq1q2_mfest3} results in
\begin{equation*}
m n \biggl( 1 - \frac{r_{\infty, q_1}^{m, n}}{m^{1/p_2} \, r_{p_2, q_1}^{m, n}} \biggr) = \frac{(m - 1) \log(n)}{2} + \BigO(1),
\end{equation*}
which leads us to
\begin{equation*}
\lim_{n \to \infty} V^{m, n}(1) = \Wsk\biggl[ \sum_{i = 1}^m E_i \geq \infty \biggr] = 0. \qedhere
\end{equation*}
\end{proof}

\begin{proof}[Proof of Corollary~\ref{sa:kritfall3}]
\textit{Case~\(q_1 \neq q_2\):} Referring to Theorem~\ref{sa:zgs3} we have
\begin{align*}
V^{m, n}(t) &= \Wsk\biggl[ \sqrt{m n} \biggl( \frac{m^{1/p_1 - 1/p_2}}{(M_{p_1/n}^{p_2/n} \Erw[\lVert \Theta_1 \rVert_{q_2}^{p_2}])^{1/p_2}} \lVert X^n \rVert_{p_2, q_2} - 1 \biggr)\\
&\qquad\quad \leq \sqrt{m n} \biggl( \frac{m^{1/p_1 - 1/p_2}}{(M_{p_1/n}^{p_2/n} \Erw[\lVert \Theta_1 \rVert_{q_2}^{p_2}])^{1/p_2}} \, \frac{r_{p_1, q_1}^{m, n}}{r_{p_2, q_2}^{m, n}} \, t - 1 \biggr) \biggr].
\end{align*}
Notice that in any case
\begin{equation*}
\lim_{n \to \infty} \frac{m^{1/p_1} \, n^{1/q_1} \, r_{p_1, q_1}^{m, n}}{m^{1/p_2} \, n^{1/q_2} \, r_{p_2, q_2}^{m, n}} = (M_{q_1}^{q_2})^{1/q_2} A_{q_1, q_2}
\end{equation*}
and
\begin{equation*}
\lim_{n \to \infty} \frac{1}{(M_{p_1/n}^{p_2/n} \Erw[\lVert n^{1/q_1 - 1/q_2} \Theta_1 \rVert_{q_2}^{p_2}])^{1/p_2}} = (M_{q_1}^{q_2})^{-1/q_2},
\end{equation*}
and from these follow the limit values for the cases \(t A_{q_1, q_1} < 1\) and \(t A_{q_1, q_2} > 1\). In the remaining case \(t A_{q_1, q_2} = 1\), by our assumptions the probability converges to the claimed value \(\Wsk[\sigma N \leq M] = \Phi(\inv{\sigma} \, M)\).

\textit{Case \(q_1 = q_2\) and \(p_1 < \infty\):} Again we rewrite
\begin{align*}
V^{m, n}(t) &= \Wsk\biggl[ \sqrt{m} \, n \biggl( \frac{m^{1/p_1 - 1/p_2}}{(M_{p_1/n}^{p_2/n})^{1/p_2}} \lVert X^n \rVert_{p_2, q_1} - 1 \biggr)\\
&\qquad\quad \leq \sqrt{m} \, n \biggl( \frac{m^{1/p_1 - 1/p_2}}{(M_{p_1/n}^{p_2/n})^{1/p_2}} \, \frac{r_{p_1, q_1}^{m, n}}{r_{p_2, q_1}^{m, n}} \, t - 1 \biggr) \biggr].
\end{align*}
The asymptotic expansions needed here are
\begin{equation*}
\frac{m^{1/p_1}}{m^{1/p_2}} \, \frac{r_{p_1, q_1}^{m, n}}{r_{p_2, q_1}^{m, n}} = 1 + \frac{\log(p_2/p_1)}{2 n} \Bigl( 1 - \frac{1}{m} + \BigTheta\Bigl( \frac{1}{n} \Bigr) \Bigr) + \BigO\Bigl( \frac{1}{m n^2} \Bigr)
\end{equation*}
and
\begin{equation*}
(M_{p_1/n}^{p_2/n})^{-1/p_2} = 1 - \frac{p_2 - p_1}{2 p_1 n} + \BigO\Bigl( \frac{1}{n^2} \Bigr),
\end{equation*}
they lead to the claimed results concerning \(t A_{q_1, q_2} = t < 1\) or \(t > 1\). For \(t = 1\), plugging in yields
\begin{multline*}
\sqrt{m} \, n \biggl( \frac{m^{1/p_1 - 1/p_2}}{(M_{p_1/n}^{p_2/n})^{1/p_2}} \, \frac{r_{p_1, q_1}^{m, n}}{r_{p_2, q_1}^{m, n}} \, t - 1 \biggr)\\
\begin{aligned}
&= \sqrt{m} \, n \biggl( \Bigl( 1 - \frac{p_2 - p_1}{2 p_1 n} + \BigO\Bigl( \frac{1}{n^2} \Bigr) \Bigr) \Bigl( 1 + \frac{\log(p_2/p_1)}{2 n} \Bigl( 1 - \frac{1}{m} + \BigTheta\Bigl( \frac{1}{n} \Bigr) \Bigr) + \BigO\Bigl( \frac{1}{m n^2} \Bigr) \Bigr) - 1 \biggr)\\
&= \sqrt{m} \, n \biggl( \frac{1}{2 n} \Bigl( \log\Bigl( \frac{p_2}{p_1} \Bigr) - \frac{p_2 - p_1}{p_1} - \frac{\log(p_2/p_1)}{m} + \BigTheta\Bigl( \frac{1}{n} \Bigr) \Bigr) + \BigO\Bigl( \frac{1}{m n^2} \Bigr) \biggr)\\
&= \frac{\sqrt{m}}{2} \Bigl( \log\Bigl( \frac{p_2}{p_1} \Bigr) - \frac{p_2 - p_1}{p_1} - \frac{1}{m} + \BigTheta\Bigl( \frac{1}{n} \Bigr) \Bigr) + \BigO\Bigl( \frac{1}{\sqrt{m} \, n} \Bigr);
\end{aligned}
\end{multline*}
and since the logarithm is strictly concave and \(p_1 \neq p_2\) (forced by \(q_1 = q_2\)) we know
\begin{equation*}
\log\Bigl( \frac{p_2}{p_1} \Bigr) - \frac{p_2 - p_1}{p_1} < \frac{p_2}{p_1} - 1 - \frac{p_2 - p_1}{p_1} = 0,
\end{equation*}
hence the limit as \(n \to \infty\) is minus infinity. This gives the claimed limits.

\textit{Case \(q_1 = q_2\) and \(p_1 = \infty\):} Lastly we have
\begin{align*}
V^{m, n}(t) &= \Wsk\biggl[ \sqrt{m} \, n \biggl( \frac{\lVert X^n \rVert_{p_2, q_1}}{m^{1/p_2} (M_\infty^{p_2/n})^{1/p_2}} - 1 \biggr)\\
&\qquad\quad \leq \sqrt{m} \, n \biggl( \frac{1}{m^{1/p_2} (M_\infty^{p_2/n})^{1/p_2}} \, \frac{r_{\infty, q_1}^{m, n}}{r_{p_2, q_1}^{m, n}} \, t - 1 \biggr) \biggr].
\end{align*}
As usual we expand,
\begin{align*}
\frac{r_{\infty, q_1}^{m, n}}{m^{1/p_2} \, r_{p_2, q_1}^{m, n}} &= 1 - \frac{\log(2 \pi n/p_2)}{2 n} \Bigl( 1 + \BigO\Bigl( \frac{\log(n)}{n} \Bigr) \Bigr)\\
&\quad + \frac{\log(2 \pi m n/p_2)}{2 m n} \Bigl( 1 + \BigO\Bigl( \frac{\log(n)}{n} \Bigr) \Bigr) + \BigO\Bigl( \frac{\log(m n)^2}{m^2 n^2} \Bigr)
\end{align*}
and
\begin{equation*}
(M_\infty^{p_2/n})^{-1/p_2} = 1 + \frac{1}{n} + \BigO\Bigl( \frac{1}{n^2} \Bigr),
\end{equation*}
and they account for the cases \(t A_{q_1, q_2} = t < 1\) or \(t > 1\). Lastly concerning \(t = 1\), plugging in gives us
\begin{multline*}
\sqrt{m} \, n \biggl( \frac{1}{m^{1/p_2} (M_\infty^{p_2/n})^{1/p_2}} \, \frac{r_{\infty, q_1}^{m, n}}{r_{p_2, q_1}^{m, n}} - 1 \biggr)\\
\begin{aligned}
&= \sqrt{m} \, n \biggl( \Bigl( 1 + \frac{1}{n} + \BigO\Bigl( \frac{1}{n^2} \Bigr) \Bigr) \Bigl( 1 - \frac{\log(2 \pi n/p_2)}{2 n} \Bigl( 1 + \BigO\Bigl( \frac{\log(n)}{n} \Bigr) \Bigr)\\
&\qquad\qquad\quad + \frac{\log(2 \pi m n/p_2)}{2 m n} \Bigl( 1 + \BigO\Bigl( \frac{\log(n)}{n} \Bigr) \Bigr) + \BigO\Bigl( \frac{\log(m n)^2}{m^2 n^2} \Bigr) \Bigr) - 1 \biggr)\\
&= \sqrt{m} \, n \biggl( -\frac{\log(2 \pi n/p_2)}{2 n} \Bigl( 1 + \BigO\Bigl( \frac{1}{\log(n)} \Bigr) \Bigr)\\
&\qquad\qquad\quad + \frac{\log(2 \pi m n/p_2)}{2 m n} \Bigl( 1 + \BigO\Bigl( \frac{\log(n)}{n} \Bigr) \Bigr) + \BigO\Bigl( \frac{\log(m n)^2}{m^2 n^2} \Bigr) \biggr)\\
&= \sqrt{m} \log(2 \pi n/p_2) \biggl( -\frac{1}{2} \Bigl( 1 + \BigO\Bigl( \frac{1}{\log(n)} \Bigr) \Bigr)\\
&\qquad\qquad\quad + \frac{\log(2 \pi m n/p_2)}{2 m \log(2 \pi n/p_2)} \Bigl( 1 + \BigO\Bigl( \frac{\log(n)}{n} \Bigr) \Bigr) \biggr) + \BigO\Bigl( \frac{\log(m n)^2}{m^{3/2} n} \Bigr).
\end{aligned}
\end{multline*}
Now we observe
\begin{equation*}
\frac{\log(2 \pi m n/p_2)}{m \log(2 \pi n/p_2)} = \frac{\log(m)}{m \log(2 \pi n/p_2)} + \frac{1}{m},
\end{equation*}
which converges to zero. So in total we get
\begin{equation*}
\lim_{n \to \infty} \sqrt{m} \, n \biggl( \frac{1}{m^{1/p_2} (M_\infty^{p_2/n})^{1/p_2}} \, \frac{r_{\infty, q_1}^{m, n}}{r_{p_2, q_1}^{m, n}} - 1 \biggr) = -\infty,
\end{equation*}
and the conclusion easily follows. Notice that this result is consistent with the previous case \(p_2 < p_1 < \infty\).
\end{proof}

\appendix

\section{Appendix: Higher\-/order mixed\-/norm spaces}

It is not difficult to generalize the idea of mixed norms to higher orders in the following sense. Let \(k \in \NZ\), let \(\vek{p}_k = (p_j)_{j \leq k} \in (0, \infty]^k\) and \(\vek{n}_k = (n_j)_{j \leq k} \in \NZ^k\), then on \(\RZ^{\times \vek{n}_k} := \RZ^{n_1 \times \dotsb \times n_k}\) define the \emph{\(k\)\textsuperscript{th}\-/order mixed norm with exponents \(p_1, \dotsc, p_k\)} recursively by \(\lVert \cdot \rVert_{\vek{p}_1} := \lVert \cdot \rVert_{p_1}\), and for \(k \geq 2\),
\begin{equation*}
\lVert (x_i)_{i \in \times \vek{n}_k} \rVert_{\vek{p}_k} := \bigl\lVert \bigl( \lVert (x_{i, i_k})_{i \in \times \vek{n}_{k - 1}} \rVert_{\vek{p}_{k - 1}} \bigr)_{i_k \leq n_k} \bigr\rVert_{p_k}.
\end{equation*}
We call \(\ellpe{\vek{p}_k}{\vek{n}_k} := (\RZ^{\times \vek{n}_k}, \lVert \cdot \rVert_{\vek{p}_k})\) the (real) \emph{finite\-/dimensional \(k\)\textsuperscript{th}\-/order mixed\-/norm sequence space,} then we have the recursion \(\ellpe{\vek{p}_k}{\vek{n}_k} = \ellpe{p_k}{n_k}(\ellpe{\vek{p}_{k - 1}}{\vek{n}_{k - 1}})\). Let \(\Kug{\vek{p}_k}{\vek{n}_k}\) be its closed unit ball, whose (\(n_1 \dotsm n_k\))\-/dimensional Lebesgue volume we denote by \(\KugVol{\vek{p}_k}{\vek{n}_k}\), and \(\Sph{\vek{p}_k}{\vek{n}_k}\) its unit sphere with associated cone measure \(\KegM{\vek{p}_k}{\vek{n}_k}\). Via the usual identification \((\CZ, \lvert \cdot \rvert) \cong (\RZ^2, \lVert \cdot \rVert_2)\) we may also subsume complex mixed\-/norm spaces under the real ones, to wit
\begin{equation*}
\ellpe{\vek{p}_k}{\vek{n}_k}(\CZ) \cong \ellpe{\vek{p}_k}{\vek{n}_k}(\ellpe{2}{2}) = \ellpe{(2, \vek{p}_k)}{(2, \vek{n}_k)}.
\end{equation*}
Notice that compared to the definition in Section~\ref{subsec:ellpe} we have reversed the order of indices here in order to write the recursion in a more natural manner (append new indices at end, not at beginning), so what we have notated \(\Kug{p, q}{m, n}\) there, corresponds to \(\Kug{(q, p)}{(n, m)}\) here.

The recursive structure of \(\ellpe{\vek{p}_k}{\vek{n}_k}\) lends itself well to calculate \(\KugVol{\vek{p}_k}{\vek{n}_k}\) in a recursive way as well. For the following assume \(k \geq 2\). Then we have
\begin{align*}
\KugVol{\vek{p}_k}{\vek{n}_k} &= \int_{\RZ^{\times \vek{n}_k}} \chF_{\Kug{\vek{p}_k}{\vek{n}_k}}(x) \, \dif x\\
&= \int_{(\RZ^{\times \vek{n}_{k - 1}})^{n_k}} \chF_{\Kug{\vek{p}_k}{\vek{n}_k}}\bigl( ((x_{i, i_k})_{i \in \times\vek{n}_{k - 1}})_{i_k \leq n_k} \bigr) \, \dif x.
\end{align*}
On each of the \(n_k\) component spaces introduce \(\ellpe{\vek{p}_{k - 1}}{\vek{n}_{k - 1}}\)\-/polar coordinates, i.e., \((x_{i, i_k})_{i \in \times\vek{n}_{k - 1}} = r_{i_k} \theta_{i_k}\) with \(r_{i_k} \in [0, \infty)\) and \(\theta_{i_k} \in \Sph{\vek{p}_{k - 1}}{\vek{n}_{k - 1}}\) for each \(i_k \in [1, n_k]\); therewith we get
\begin{equation*}
\lVert ((x_{i, i_k})_{i \in \times\vek{n}_{k - 1}})_{i_k \leq n_k} \rVert_{\vek{p}_k} = \bigl\lVert \bigl( \lVert r_{i_k} \theta_{i_k} \rVert_{\vek{p}_{k - 1}} \bigr)_{i_k \leq n_k} \bigr\rVert_{p_k} = \lVert (r_{i_k})_{i_k \leq n_k} \rVert_{p_k}
\end{equation*}
and hence \(\chF_{\Kug{\vek{p}_k}{\vek{n}_k}}\bigl( ((x_{i, i_k})_{i \in \times\vek{n}_{k - 1}})_{i_k \leq n_k} \bigr) = \chF_{\Kug{p_k}{n_k}}((r_{i_k})_{i_k \leq n_k})\). This yields
\begin{align*}
\KugVol{\vek{p}_k}{\vek{n}_k} &= \bigl( n_1 \dotsm n_{k - 1} \KugVol{\vek{p}_{k - 1}}{\vek{n}_{k - 1}} \bigr)^{n_k} \int_{[0, \infty)^{n_k}} \int_{(\Sph{\vek{p}_{k - 1}}{\vek{n}_{k - 1}})^{n_k}} \chF_{\Kug{p_k}{n_k}}(r) \prod_{i_k = 1}^{n_k} r_{i_k}^{n_1 \dotsm n_{k - 1} - 1} \, \dif(\KegM{\vek{p}_{k - 1}}{\vek{n}_{k - 1}})^{\otimes n_k}(\theta) \, \dif r\\
&= (\KugVol{\vek{p}_{k - 1}}{\vek{n}_{k - 1}})^{n_k} \int_{[0, \infty)^{n_k}} \chF_{\Kug{p_k}{n_k}}(r) \prod_{i_k = 1}^{n_k} \bigl( n_1 \dotsm n_{k - 1} r_{i_k}^{n_1 \dotsm n_{k - 1} - 1} \bigr) \, \dif r.
\end{align*}
Now transform \(s_{i_k} := r_{i_k}^{n_1 \dotsm n_{k - 1}}\) for each \(i_k \in [1, n_k]\); recall \(\lVert (x_i)_{i \leq n} \rVert_p = \lVert (\lvert x_i \rvert^\alpha)_{i \leq n} \rVert_{p/\alpha}^{1/\alpha}\) for any \(\alpha \in (0, \infty)\), which gives
\begin{equation*}
\lVert r \rVert_{p_k} = \lVert (r_{i_k}^{n_1 \dotsm n_{k - 1}})_{i_k \leq n_k} \rVert_{p_k/(n_1 \dotsm n_{k - 1})}^{1/(n_1 \dotsm n_{k - 1})} = \lVert s \rVert_{p_k/(n_1 \dotsm n_{k - 1})}^{1/(n_1 \dotsm n_{k - 1})},
\end{equation*}
hence \(\chF_{\Kug{p_k}{n_k}}(r) = \chF_{\Kug{p_k/(n_1 \dotsm n_{k - 1})}{n_k}}(s)\) and thus
\begin{align*}
\KugVol{\vek{p}_k}{\vek{n}_k} &= (\KugVol{\vek{p}_{k - 1}}{\vek{n}_{k - 1}})^{n_k} \int_{[0, \infty)^{n_k}} \chF_{\Kug{p_k/(n_1 \dotsm n_{k - 1})}{n_k}}(s) \, \dif s\\
&= \frac{(\KugVol{\vek{p}_{k - 1}}{\vek{n}_{k - 1}})^{n_k} \, \KugVol{p_k/(n_1 \dotsm n_{k - 1})}{n_k}}{2^{n_k}},
\end{align*}
which is the desired recursive formula. Via induction on \(k\) this leads to the explicit formula
\begin{equation*}
\KugVol{\vek{p}_k}{\vek{n}_k} = 2^{n_1 \dotsm n_k} \prod_{j = 1}^k \frac{(\KugVol{p_j/(n_1 \dotsm n_{j - 1})}{n_j})^{n_{j + 1} \dotsm n_k}}{2^{n_j \dotsm n_k}} = 2^{n_1 \dotsm n_k} \prod_{j = 1}^k \frac{\Gamma(\frac{n_1 \dotsm n_{j - 1}}{p_j} + 1)^{n_j \dotsm n_k}}{\Gamma(\frac{n_1 \dotsm n_j}{p_j} + 1)^{n_{j + 1} \dotsm n_k}}.
\end{equation*}
In the special case of \(\ellpe{(p, q)}{(m, n)}(\CZ)\) this yields
\begin{equation*}
\KugVol{(2, p, q)}{(2, m, n)} = 2^{2 m n} \, \frac{\Gamma(\frac{1}{2} + 1)^{2 m n} \, \Gamma(\frac{2}{p} + 1)^{m n} \, \Gamma(\frac{2 m}{q} + 1)^n}{\Gamma(\frac{2}{2} + 1)^{m n} \, \Gamma(\frac{2 m}{p} + 1)^n \, \Gamma(\frac{2 m n}{q} + 1)} = \pi^{m n} \, \frac{\Gamma(\frac{2}{p} + 1)^{m n} \, \Gamma(\frac{2 m}{q} + 1)^n}{\Gamma(\frac{2 m}{p} + 1)^n \, \Gamma(\frac{2 m n}{q} + 1)},
\end{equation*}
and this equals the value given in \cite[Theorem~5]{KV2017}.

\subsection*{Acknowledgement}
Michael Juhos and Joscha Prochno have been supported by the Austrian Science Fund (FWF) Project P32405 \textit{Asymptotic Geometric Analysis and Applications} and by the FWF Project F5513-N26 which is a part of the Special Research Program \emph{Quasi-Monte Carlo Methods: Theory and Applications}. Zakhar Kabluchko has been supported by the German Research Foundation under Germany’s Excellence Strategy EXC 2044~-- 390685587, \textit{Mathematics M\"unster: Dynamics~-- Geometry~-- Structure} and by the DFG priority program SPP 2265 \textit{Random Geometric Systems}.

\bibliographystyle{plain}
\bibliography{mixedlplqspaces_bib}

\end{document}